\documentclass[a4paper,fleqn]{amsart}
\usepackage{amssymb}
\usepackage{amstext}
\usepackage{amsmath}
\usepackage {amscd}
\usepackage{amsthm}
\usepackage{amsfonts}
\usepackage{enumerate}
\usepackage[pdftex]{graphicx}
\usepackage{latexsym}
\usepackage{mathrsfs}

\usepackage{mathtools}
\usepackage[all]{xy}
\xyoption{all}

\usepackage{pstricks}
\usepackage{lscape}
\usepackage{comment}
\usepackage{tikz}

\newtheorem {thm}{Theorem}[section]
\newtheorem {dfn}[thm]{Definition}
\newtheorem {pro}[thm]{Proposition}
\newtheorem {lem}[thm]{Lemma}
\newtheorem {cor}[thm]{Corollary}
\newtheorem {rem}[thm]{Remark}
\newtheorem {ex}[thm]{Example}

\newcommand {\Y} {\mathrm{Y}}

\newcommand {\A} {\mathbb{A}}
\newcommand {\F} {\mathbb{F}}
\newcommand {\I} {\mathbb{I}}

\renewcommand {\S} {\mathbb{S}}
\newcommand {\T} {\mathbb{T}}
\newcommand {\Z} {\mathbb{Z}}

\renewcommand {\AA} {N}
\newcommand {\CC} {\mathcal{C}}
\newcommand {\DD} {\mathcal{D}}
\newcommand {\EE} {\mathcal{E}}
\newcommand {\FF} {\mathcal{F}}
\newcommand {\GG} {\mathcal{G}}

\newcommand {\LL} {L}
\newcommand {\LLL} {\mathcal{L}}

\newcommand {\NN} {N}
\newcommand {\NNN} {\mathcal{N}}

\newcommand {\DDD} {\mathsf{D}}

\newcommand {\XXX} {\mathsf{X}}
\newcommand {\YYY} {\mathsf{Y}}

\newcommand {\Hom}{\mathop{\mathrm{Hom}}\nolimits} 

\newcommand {\End} {\mathop{\mathrm{End}}\nolimits}
\newcommand {\Ext} {\mathop{\mathrm{Ext}}\nolimits}

\newcommand{\Rhom}
{\mathop{\mathbf{R}\mathrm{Hom}}\nolimits}
\newcommand{\REnd}
{\mathop{\mathbf{R}\mathrm{End}}\nolimits}
\newcommand{\Lotimes}{\mathop{\stackrel{\mathbf{L}}{\otimes}}\nolimits}

\newcommand {\Mod} {\mathop{\mathrm{Mod}}\nolimits}
\renewcommand {\mod} {\mathop{\mathrm{mod}}\nolimits}

\newcommand {\proj} {\mathop{\mathrm{proj}}\nolimits}
\newcommand {\inj} {\mathop{\mathrm{inj}}\nolimits}

\newcommand {\hd} {\mathop{\mathrm{gl.dim}}\nolimits}

\newcommand {\Id} {\mathop{\mathrm{Id}}\nolimits}

\newcommand {\Coker} {\mathop{\mathrm{Coker}}\nolimits}

\newcommand {\op} {\mathop{\mathrm{op}}\nolimits}

\newcommand {\rad} {\mathop{\mathrm{rad}}\nolimits}

\newcommand {\vp} {\mathsf{p}}
\newcommand {\vq} {\mathsf{q}}
\newcommand {\np} {\bar{p}}
\newcommand {\nq} {\bar{q}}

\newcommand {\supp}{\mathop{\mathrm{supp}}\nolimits}

\newcommand {\per} {\mathop{\mathrm{per}}
\nolimits}

\newcommand {\fin} {\mathop{\mathrm{fin}}\nolimits}

\newcommand{\nak} {\nu}
\newcommand {\fP} {\mathfrak{P}}

\numberwithin{equation}{section}

\begin{document}
\title{On derived equivalences of Nakayama algebras }
\author{Taro Ueda}
\address{T. Ueda: Graduate School of Mathematics, Nagoya University}
\email{m15013z@math.nagoya-u.ac.jp}
\begin{abstract}
In this paper, 
we investigate the derived category of the Nakayama algebra 
$\AA(n,\ell)=K\A_{n}/\rad(K\A_{n})^{\ell}$. 
We construct 
a derived equivalence between Nakayama algebras $\AA(n,\ell)$ and $\AA(n,\ell+1)$ 
where $n=p(p+1)q+p(p-1)r$ and $\ell=(p+1)q+pr$ 
for each integers $p\geqslant 2$, $q\geqslant 1$, $r\geqslant 0$. 
To achieve it, we introduce families of idempotent subalgebras of $K\A_{s}\otimes K\A_{t}$ and characterize their derived categories by the existence of a certain family of objects called $S$-families. 
\end{abstract}
\maketitle
\tableofcontents

\section{Introduction}
In the representation theory of finite dimensional algebras, it is an important problem to determine whether two given finite dimensional algebras are derived equivalent or not. 
We refer to \cite{HHK,Ric} for derived equivalences. 

In this paper, we study this problem for a certain class of Nakayama algebras 
(i.e.\,right and left serial algebras 
\cite{Nak}) over a field $K$. 
The structure of the module category of finite dimensional modules over Nakayama algebras is well-known \cite{ASS}, but little is known about the structure of their derived categories. 

We consider the Nakayama algebra 
\begin{equation*}
\AA(n,\ell)=K\A_{n}/(\rad K\A_{n})^{\ell}
\end{equation*} 
where 
\begin{gather*}
\A_{n}=[1\stackrel{a_{1}}{\to}2
\stackrel{a_{2}}{\to}3\stackrel{a_{3}}{\to}
\dots\stackrel{a_{n-1}}{\to}n]
\end{gather*}
and $\rad K\A_{n}$ is the Jacobson radical of the path algebra $K\A_{n}$. 
For any integers $s\geqslant 1$, $t\geqslant 1$, 
$0\leqslant u\leqslant t$, 
let \[\LL(s,t,u)
=(\AA(s,2)^{\op}\otimes \AA(t,2)^{\op})
/\langle\sum_{i=0}^{u-1} e_{s}\otimes e_{t-i}\rangle\] 
\[\LL^{!}(s,t,u)
=(K\A_{s}\otimes K\A_{t})
/\langle\sum_{i=0}^{u-1} e^{!}_{s}
\otimes e^{!}_{t-i}\rangle\]
where $\otimes=\otimes_{K}$ and 
$e_{i}\otimes e_{j}$ 
(resp.\,$e^{!}_{i}\otimes e^{!}_{j}$) is the 
idempotent of 
$\AA(s,2)^{\op}\otimes \AA(t,2)^{\op}$ 
(resp.\,$K\A_{s}\otimes K\A_{t}$) 
corresponding the vertex $(i,j)$. 
As is clear from the definition, 
$\LL(s,t,t)=\LL(s-1,t,0)$. 
In his paper \cite[Cor.\,1.2]{Lad}, Ladkani gave a triangle equivalence 
\begin{equation}\label{IntroLad}
\text{$\per\AA(st,t+1)
\to\per(K\A_{s}\otimes K\A_{t})$\ \  for any integers $s\geqslant1,t\geqslant 1$ such that $st>t+1$.} 
\end{equation}
We generalize this equivalence as follows: 
\begin{thm}[Proposition \ref{Duality}, Theorem \ref{Nak}]\label{Intro 1}
There exist triangle equivalences 
\[\per\AA(st-u,t+1)\to\per\LL(s,t,u)
\to\per\LL^{!}(s,t,u)\]
for any integers $s,t\geqslant 1$, 
$0\leqslant u\leqslant t$. 
In particular, $\AA(st-u,t+1)$ is derived equivalent to a $K$-algebra with 
global dimension $\leqslant 2$.
\end{thm}
For $u=0$, we obtain a triangle equivalence 
(\ref{IntroLad}). Equivalently, $\AA(n,\ell)$ is derived equivalent to $\LL\left(\lceil
\frac{n}{\ell-1}\rceil,\ell-1,r\right)$ 
for any integers $n\geqslant 3$ and 
$\ell\geqslant 2$, 
$r=(\ell-1)\lceil\frac{n}{\ell-1}\rceil-n$, 
where for any real number $x$, 
$\lceil x \rceil=\sup \{m\in\Z\mid x\leqslant m\}$. 
In this paper, we show the following result. 

\begin{thm}[Theorem \ref{Main 1}]
\label{Intro 2}
Let $s$, $t$, $u$ be 
three positive integers such that $1\leqslant u\leqslant t$. 
Suppose that one of the following 
conditions is satisfied. 
\begin{enumerate}[\rm (a)]
\item $u\in \mathbb{Z}s$ 
and $t-u\in \mathbb{Z}(s+1)$. 
\item $s=2$ and 
$t-u\in 3\mathbb{Z}$. 
\end{enumerate}
Then there exists a triangle equivalence 
$\per\LL(s,t,u)
\to \per\LL(s,t-1,u-s)$. 
\end{thm}
The integers satisfying the condition of Theorem \ref{Intro 2} (a) can be written as 
\[(s,t,u)=(p,(p+1)q+pr,pr).\] 
Applying it to Theorem \ref{Intro 2}, 
we obtain a triangle equivalence 
\[\per\LL(p,(p+1)q+pr,pr)
\to \per\LL(p,(p+1)q+pr-1,p(r-1))\]
for any integers $p\geqslant 1$, $q\geqslant 0$, 
$r\geqslant 1$. 
By Theorem \ref{Intro 1}, we get the 
following result:
\begin{cor}[Corollary \ref{Main 2}]
\label{Intro 3}
Let $p,q$ be two integers such that 
$p\geqslant 2$, $q\geqslant 1$. 
Suppose that one of the following 
conditions is satisfied. 
\begin{enumerate}[\rm (a)]
\item $r\in\Z_{\geqslant 0}$.
\item $p=2$ and 
$r\in \frac{1}{2}\Z_{\geqslant 0}$. 
\end{enumerate}
Then there exists a triangle equivalence 
\[\text{$\per\AA(n,\ell+1)
\to \per\AA(n,\ell)$\ \ 
where $n=p(p+1)q+p(p-1)r$, 
$\ell=(p+1)q+pr$}.\]
\end{cor}

% Pair of der eq 
\begin{figure} 
\begin{tikzpicture}
\draw[->,>=stealth,semithick] 
(0,0)--(13.5,0)node
[above]{$n$};
\draw[->,>=stealth,semithick] 
(0,0)--(0,-13.5)node
[left]{$\ell$}; 
\draw[thick,domain=0:13] 
plot(\x,-\x);

\draw[thick,domain=1.5:13] 
plot(\x,-\x+0.5); 
\draw[thick,domain=1.5:13] 
plot(\x,-\x+0.75);

\draw[thick,domain=3:13] 
plot(\x,-\x+1.25);
\draw[thick,domain=3:13] 
plot(\x,-\x+1.5);  
 
\draw[thick,domain=4.5:13] 
plot(\x,-\x+2); 
\draw[thick,domain=4.5:13] 
plot(\x,-\x+2.25);  

\draw[thick,domain=6:13] 
plot(\x,-\x+2.75); 
\draw[thick,domain=6:13] 
plot(\x,-\x+3); 

\draw[thick,domain=7.5:13] 
plot(\x,-\x+3.5);
\draw[thick,domain=7.5:13] 
plot(\x,-\x+3.75); 

\draw[thick,domain=9:13] 
plot(\x,-\x+4.25);
\draw[thick,domain=9:13] 
plot(\x,-\x+4.5); 

\draw[thick,domain=10.5:13] 
plot(\x,-\x+5); 
\draw[thick,domain=10.5:13]
plot(\x,-\x+5.25); 

\draw[thick,domain=12:13] 
plot(\x,-\x+5.75); 
\draw[thick,domain=12:13]
plot(\x,-\x+6); 

\draw[thick,domain=1.5:6.5] 
plot(2*\x,-\x+0.5); 
\draw[thick,domain=1.5:6.5]
plot(2*\x,-\x+0.25); 

\draw[thick,domain=1.5:6.5] 
plot(2*\x,-\x+0.5); 
\draw[thick,domain=1.5:6.5]
plot(2*\x,-\x+0.25); 

\draw[thick,domain=1.5:5] 
plot(2*\x+3,-\x-0.5); 
\draw[thick,domain=1.5:5]
plot(2*\x+3,-\x-0.75); 

\draw[thick,domain=1.5:3.5] 
plot(2*\x+6,-\x-1.5); 
\draw[thick,domain=1.5:3.5]
plot(2*\x+6,-\x-1.75); 

\draw[thick,domain=1.5:2] 
plot(2*\x+9,-\x-2.5); 
\draw[thick,domain=1.5:2]
plot(2*\x+9,-\x-2.75); 

\draw[thick,domain=6/4:25/6] 
plot(3*\x+2/4,-\x+1/4); 
\draw[thick,domain=6/4:25/6]
plot(3*\x+2/4,-\x); 

\draw[thick,domain=6/4:30/12] 
plot(3*\x+22/4,-\x-4/4); 
\draw[thick,domain=6/4:30/12]
plot(3*\x+22/4,-\x-5/4);

\draw[thick,domain=6/4:23/8] 
plot(4*\x+6/4,-\x); 
\draw[thick,domain=6/4:23/8]
plot(4*\x+6/4,-\x-1/4);  

\draw[thick,domain=1.5:2] 
plot(5*\x+3,-\x-0.25); 
\draw[thick,domain=1.5:2]
plot(5*\x+3,-\x-0.5);  

\foreach \z in {0,0.25,0.5,0.75,1,1.25,1.5,1.75,2,2.25,2.5,
2.75,3,3.25,3.5,3.75,4,4.25,4.5,4.75,5,5.25,5.5,5.75,6,6.25,6.5,6.75,7,7.25,7.5,7.75,8,8.25,8.5,8.75,9,9.25,9.5,9.75,10,10.25,10.5,10.75,11,11.25,11.5,11.75,12,12.25,12.5,12.75,13} 
\draw[ultra thin] (\z,0)--(\z,-13);

\foreach \w in 
{0,0.25,0.5,0.75,1,1.25,1.5,1.75,2,2.25,2.5,
2.75,3,3.25,3.5,3.75,4,4.25,4.5,4.75,5,5.25,5.5,5.75,6,6.25,6.5,6.75,7,7.25,7.5,7.75,8,8.25,8.5,8.75,9,9.25,9.5,9.75,10,10.25,10.5,10.75,11,11.25,11.5,11.75,12,12.25,12.5,12.75,13}
\draw[ultra thin] (0,-\w)--(13,-\w);

\foreach \z in {1.25,2.5,3.75,5,6.25,7.5,8.75,10,
11.25,12.5} 
\draw[thick] (\z,0)--(\z,-13);

\foreach \z in {1.25,2.5,3.75,5,6.25,7.5,8.75,10,
11.25,12.5} 
\draw[thick] (0,-\z)--(13,-\z);

\foreach \s in 
{0,0.25,0.5,0.75,1,1.25,1.5,1.75,2,2.25,2.5,
2.75,3,3.25,3.5,3.75,4,4.25,4.5,4.75,5,5.25,5.5,5.75,6,6.25,6.5,6.75,7,7.25,7.5,7.75,8,8.25,8.5,8.75,9,9.25,9.5,9.75,10,10.25,10.5,10.75,11,11.25,11.5}
 \draw[<->,very thick] (\s+1.5,-\s-1)--
(\s+1.5,-\s-0.75);

\foreach \s in 
{0,0.25,0.5,0.75,1,1.25,1.5,1.75,2,2.25,2.5,
2.75,3,3.25,3.5,3.75,4,4.25,4.5,4.75,5,5.25,5.5,5.75,6,6.25,6.5,6.75,7,7.25,7.5,7.75,8,8.25,8.5,8.75,9,9.25,9.5,9.75,10}
\draw[<->,very thick] (\s+3,-\s-1.75)--
(\s+3,-\s-1.5);

\foreach \t in 
{0,0.25,0.5,0.75,1,1.25,1.5,1.75,2,2.25,2.5,
2.75,3,3.25,3.5,3.75,4,4.25,4.5,4.75,5,5.25,5.5,5.75,6,6.25,6.5,6.75,7,7.25,7.5,7.75,8,8.25,8.5} 
\draw[<->,very thick] (\t+4.5,-\t-2.5)--
(\t+4.5,-\t-2.25); 

\foreach \u in 
{0,0.25,0.5,0.75,1,1.25,1.5,1.75,2,2.25,2.5,
2.75,3,3.25,3.5,3.75,4,4.25,4.5,4.75,5,5.25,5.5,5.75,6,6.25,6.5,6.75,7} 
\draw[<->,very thick] (\u+6,-\u-3.25)--
(\u+6,-\u-3); 

\foreach \v in 
{0,0.25,0.5,0.75,1,1.25,1.5,1.75,2,2.25,2.5,
2.75,3,3.25,3.5,3.75,4,4.25,4.5,4.75,5,5.25,5.5}
\draw[<->,very thick] (\v+7.5,-\v-4)--
(\v+7.5,-\v-3.75); 

\foreach \x in 
{0,0.25,0.5,0.75,1,1.25,1.5,1.75,2,2.25,2.5,
2.75,3,3.25,3.5,3.75,4} 
\draw[<->,very thick] (\x+9,-\x-4.75)--
(\x+9,-\x-4.5);

\foreach \y in 
{0,0.25,0.5,0.75,1,1.25,1.5,1.75,2,2.25,2.5} 
\draw[<->,very thick] (\y+10.5,-\y-5.5)--
(\y+10.5,-\y-5.25); 

\foreach \y in 
{0,0.25,0.5,0.75,1} 
\draw[<->,very thick] (\y+12,-\y-6.25)--
(\y+12,-\y-6); 

\foreach \y in 
{0,3/4,6/4,9/4,12/4,15/4,18/4} 
\draw[<->,very thick] (2*\y+3,-\y-1)--
(2*\y+3,-\y-5/4); 

\foreach \y in 
{0,3/4,6/4,9/4,12/4} 
\draw[<->,very thick] (2*\y+6,-\y-2)--
(2*\y+6,-\y-9/4); 

\foreach \y in 
{0,3/4,6/4} 
\draw[<->,very thick] (2*\y+9,-\y-3)--
(2*\y+9,-\y-13/4); 

\foreach \y in 
{0} 
\draw[<->,very thick] (2*\y+12,-\y-4)--
(2*\y+12,-\y-17/4); 

\foreach \y in 
{0,1,2} 
\draw[<->,very thick] (3*\y+5,-\y-5/4)--
(3*\y+5,-\y-6/4); 

\foreach \y in 
{0,1} 
\draw[<->,very thick] (3*\y+10,-\y-10/4)--
(3*\y+10,-\y-11/4);

\foreach \y in 
{0,5/4} 
\draw[<->,very thick] (4*\y+30/4,-\y-6/4)--
(4*\y+30/4,-\y-7/4); 

\foreach \y in 
{0} 
\draw[<->,very thick] (5*\y+42/4,-\y-7/4)--
(5*\y+42/4,-\y-8/4); 

\draw (0,0) node[above]{0};
\draw (1.25,0) node[above]{5};
\draw (2.5,0) node[above]{10};
\draw (3.75,0) node[above]{15}; 
\draw (5,0) node[above]{20};
\draw (6.25,0) node[above]{25}; 
\draw (7.5,0) node[above]{30};
\draw (8.75,0) node[above]{35};
\draw (10,0) node[above]{40};
\draw (11.25,0) node[above]{45};
\draw (12.5,0) node[above]{50};

\draw (0,-1.25) node[left]{5};
\draw (0,-2.5) node[left]{10};
\draw (0,-3.75) node[left]{15}; 
\draw (0,-5) node[left]{20};
\draw (0,-6.25) node[left]{25}; 
\draw (0,-7.5) node[left]{30};
\draw (0,-8.75) node[left]{35};
\draw (0,-10) node[left]{40};
\draw (0,-11.25) node[left]{45};
\draw (0,-12.5) node[left]{50};
\end{tikzpicture}
\caption{ 
In the above graph, the arrow 
$\leftrightarrow$ shows the pairs of Nakayama algebras in Corollary \ref{Intro 3}.}
\end{figure}
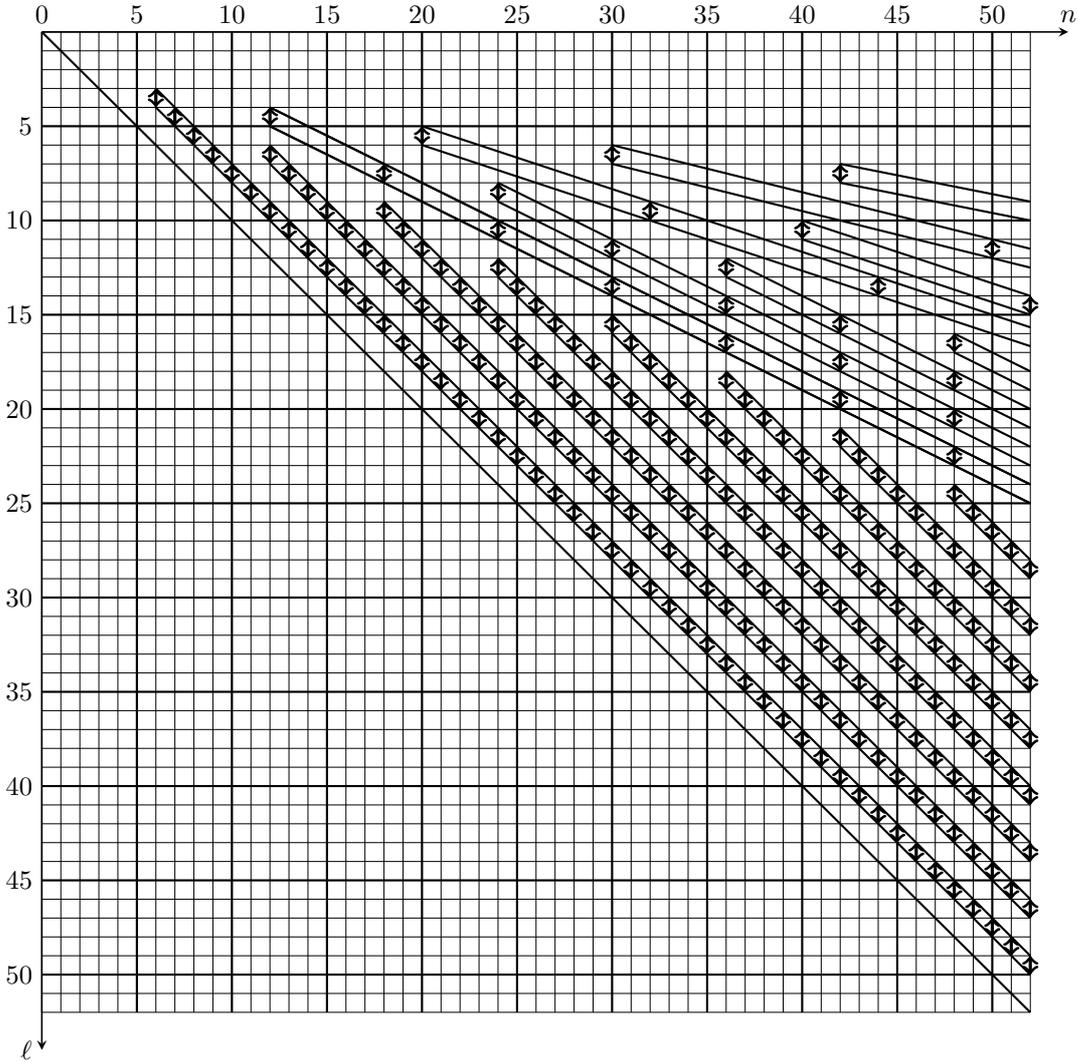
For $p=2$, $r=\frac{s}{2}$, we obtain 
a triangle equivalence 
\begin{equation}\label{HS sing}
\text{$\per \NN(s+6q,s+3q+1)\to \per \NN(s+6q,s+3q)$\ \ for any integers 
$q\geqslant 1$, $s\geqslant 0$.}
\end{equation}
For $q=1$, 
(\ref{HS sing}) is a triangle equivalence 
\[\text{$\per N(s+6,s+4)
\to \per N(s+6,s+3)$\ \ for any integer 
$s\geqslant 0$}\]
due to Happel-Seidel \cite[Prop.\,2.3]{HS}
where they proved the above two derived categories are triangle equivalent to a derived category of a path algebra of a star quiver with three branches of length $2,3,s+3$ respectively. 

For $q=2$, (\ref{HS sing}) is a triangle equivalence 
\[\text{$\per N(s+12,s+7)
\to \per N(s+12,s+6)$\ \ for any integer 
$s\geqslant 0$}\]
due to Lenzing-Meltzer-Ruan \cite[Thm.\,6.3]{LMR} where they proved the above two derived categories are triangle equivalent 
to a stable category of vector bundles for 
a weighted projective line 
$\mathbb{X}(2,3,s+7)$. 

A central role in our proof above is played 
by the notion of $S$-families 
(see Definition \ref {def LS}). 
By using the following result, 
we can construct a triangle equivalence 
between a given triangulated category and 
$\per \LL(S)$. 

\begin{thm}[Theorem \ref{Triviality}]\label{Intro 0}
Let $\DD$ be an algebraic, idempotent complete, 
and $\Ext$-finite triangulated category with a Serre functor. 
If there exists a full $S$-family 
$(X_{i,j})_{(i,j)\in S}$, then there exists a 
triangle equivalence $F:\DD\to \per \LL(S)$ 
such that $F(X_{i,j})\simeq P(i,j)$ for any 
$(i,j)\in S$. 
\end{thm}

% Definitions of S-families
The following figure presents the 
relationship of definitions of $S$-families.
\\
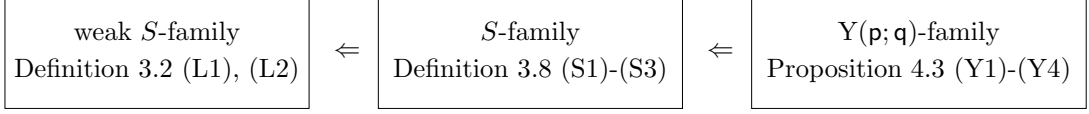
\begin{tikzpicture}
\draw[ultra thin] (-1,0)--(3,0)--(3,-1.5)
--(-1,-1.5)--(-1,0); 
\draw (1,-1/2) node {weak $S$-family};
\draw (1,-1) node 
{\text{Definition \ref{def WLS} (L1), (L2)}};
\draw (3.5,-3/4) node {$\Leftarrow$};
\end{tikzpicture}
\begin{tikzpicture}
\draw[ultra thin] (-1,0)--(3,0)--(3,-1.5)
--(-1,-1.5)--(-1,0); 
\draw (1,-0.5) node {$S$-family};
\draw (1,-1) node 
{\text{Definition \ref{def LS} (S1)-(S3)}};
\draw (3.5,-3/4) node {$\Leftarrow$};\end{tikzpicture}
\begin{tikzpicture}
\draw[ultra thin] (-1.2,0)--(3.2,0)--(3.2,-1.5)
--(-1.2,-1.5)--(-1.2,0); 
\draw (1,-0.5) node {$\Y(\vp;\vq)$-family};
\draw (1,-1) node {\text{Proposition \ref{Lpq} 
(Y1)-(Y4)}};
\end{tikzpicture}

As another application of Theorem \ref{Intro 0}, we have the following result related to (\ref{IntroLad}).
\begin{thm}[Theorem \ref{Main3}]\label{Intro 6}
There exists a triangle equivalence 
\[\text{$\per\LL(p+1,q,q-1)
\to \per\LL(q+1,p,p-1)$ 
\ \ for any integers $p\geqslant 2$, 
$q\geqslant 2$. }\]
\end{thm}
By Theorem \ref{Intro 1}, the above result gives the following 
triangle equivalence due to  
Lenzing-Meltzer-Ruan \cite[Prop.\,4.1]{LMR}. 
\begin{cor}[Corollary \ref{Main 4}]
\label{Intro 7}
There exists a triangle equivalence 
\[\text{$\per\AA(pq+1,q+1)
\to \per\AA(pq+1,p+1)$ 
\ \ for any integers $p\geqslant 2$, 
$q\geqslant 2$. }\]
\end{cor}

In the rest, we describe the summary of each chapter of this paper. 
In Chapter 2, we give basic results for 
Serre functors, semi-orthogonal decompositions, 
admissible subcategories, tilting objects, 
and exceptional sequences. 
In Chapter 3, we characterize the 
perfect derived category of the algebra 
$L(S)$ by using the terminology of $S$-families 
which are sequences of exceptional objects satisfying some conditions 
(Definition \ref{def LS}). By showing the existence of an $S$-family, 
we can construct a triangle equivalence between a 
triangulated category satisfying some conditions and the perfect derived category of the algebra $L(S)$ (Theorem \ref{Triviality}). 
In Chapter 4, we study $S$-families when 
$S$ is a Young diagram. And we show that 
if $S$ is the Young diagram $Y(s,t,u)$,  
the algebra $L(S)$ is derived 
equivalent to the Nakayama algebra 
$N(st-u,t+1)$ (Theorem \ref{Nak}). 
In Chapter 5, we introduce mutations of 
$S$-families under some assumption for $S$ 
which is a mutation as exceptional sequences 
(Theorem \ref{Mutation Lemma 1}, 
\ref{Mutation Lemma 2}). 
By using results for mutations of $S$-families, we prove main results in this paper 
(Theorem \ref{Main 1}). 

\medskip\noindent{\bf  Conventions}
In this paper, $K$ is a field and all modules over $K$-linear categories are right modules. 
We denote by $\otimes$ the tensor product over $K$. For any $K$-vector space $V$, we denote 
by $V^{\ast}$ the dual of $V$. 
For any set $\{X_{i}\mid i\in J\}$ 
of objects in a triangulated category $\DD$, we denote by $\langle X_{i}\mid i\in J \rangle$ 
the thick subcategory of $\DD$ 
generated by $\{X_{i}\mid i\in J\}$. 
For any $K$-linear category $\CC$, we denote by 
$\DDD(\CC)$ the derived category of $\CC$, and 
$\per\CC$ the perfect derived category of $\CC$ i.e.\,the thick subcategory $\langle P_{\CC}(i)\mid i\in \CC\rangle$ of $\DDD(\CC)$ where 
$P_{\CC}(i):=\Hom_{\CC}(-,i)$.
 
For any two arrows $\alpha:i\to j$ and $\beta:j\to k$ in a quiver $Q$, 
we denote the their composition  
by $\beta\alpha:i\to k$. 
For any admissible ideal $I$ of $KQ$, we denote by $S_{KQ/I}(i)$ (resp.\,$P_{KQ/I}(i)$, $I_{KQ/I}(i)$), 
the simple (resp.\,projective, injective) 
$KQ/I$-module corresponding a vertex $i$ of $Q$. 
In our conventions, for any source $i$ in $Q$, $S_{KQ/I}(i)\simeq P_{KQ/I}(i)$. 
We often simply denote 
by $S(i)$ (resp.\,$P(i)$, $I(i)$) instead of 
$S_{A}(i)$ (resp.\,$P_{A}(i)$, $I_{A}(i)$). 

\medskip\noindent{\bf Acknowledgements}
I am grateful to Osamu Iyama for 
his comments to this paper and to 
Ryo Takahashi for his support. 
I would like to thank the referees 
for their careful reading of this paper.

\section{Preliminaries} 
\subsection{$K$-linear categories and modules}
We refer to \cite{GR} for the representation theory of $K$-linear categories. 
A category $\CC$ is called a \emph{$K$-linear category} if each hom-set is a $K$-vector space and each composition map 
\[\Hom_{\CC}(X,Y)\times\Hom_{\CC}(Y,Z)
\to \Hom_{\CC}(X,Z);
(f,g)\mapsto g\circ f\]
is a bilinear map. 
Let $\CC$ and $\CC'$ be two $K$-linear categories. 
A functor $F:\CC\to \CC'$ is called a 
\emph{$K$-linear functor} if each mapping 
\[\Hom_{\CC}(X,Y)\to \Hom_{\CC'}(F(X),F(X'));
f\mapsto F(f)\]
is a $K$-linear map for any two objects $X$ and $X'$ in $\CC$. 

Let $\Mod K$ be the $K$-linear category consisting of $K$-vector spaces. A \emph{(right) $\CC$-module} is a $K$-linear functor 
$X:\CC^{\op}\to\Mod K$. For any small $K$-linear 
category $\CC$, we denote by $\Mod\CC$ the $K$-linear category consisting of $\CC$-modules. Then we define $\CC$-module $P_{\CC}(i)$ as
\[\text{$P_{\CC}(i)=P(i):=\Hom_{\CC}(-,i)$}.\]
Then $P_{\CC}(i)$ is a projective module for any $i\in \CC$. For any $\CC$-module $X$, there exists a surjective morphism 
\[\bigoplus_{i\in \CC} P_{\CC}(i)^{\oplus E_{i}}\to X\]
where $E_{i}$ is a basis of 
$\Hom_{\Mod\CC}(P_{\CC}(i),X)$. 
In particular, any projective $\CC$-module $P$ is a direct summand of a direct sum of projective modules $P_{\CC}(i)$. 

A $\CC$-module $X$ is \emph{finitely generated} if 
there exist a family $(n_{i})_{i\in I}$ of nonnegative integers $n_{i}$ indexed by 
a finite set $I$ of objects in $\CC$ 
and a surjective morphism 
$\displaystyle p:\bigoplus_{i\in I} P_{\CC}(i)^{\oplus n_{i}}\to X$. We denote by $\mod \CC$, the category of finitely generated modules. 
A $K$-linear category $\CC$ is said to be 
\emph{$Hom$-finite} over $K$ if
$\dim_{K}\Hom_{\CC}(i,j)<\infty$ for any 
$i$, $j\in\CC$. 
 Let $\CC$ be a $\Hom$-finite $K$-linear category. For any $i\in\CC$, define 
\[S_{\CC}(i):=P_{\CC}(i)/\rad P_{\CC}(i).\]
Then $S_{\CC}(i)$ is a finitely generated 
semi-simple $\CC$-module. 

A category $\CC$ is said to be \emph{svelte} if 
for any objects $i,j\in\CC$, the 
relation $i\simeq j$ implies the relation 
$i=j$. 
A svelte $K$-linear category $\CC$ is said to be 
\emph{locally bounded} if for any $i\in\CC$, 
the set \[\{j\in\CC\mid 
\text{$\Hom_{\CC}(i,j)\neq 0$ or $\Hom_{\CC}(j,i)\neq 0$}\}\] is a finite set. 
We define the support of a $\CC$-module $X$ as the following:
\[\supp X=\{i\in \CC\mid X(i)\neq 0\}.\]
A svelte $K$-linear category $\CC$ is locally bounded if and only if $\supp P_{\CC}(i)$ is a finite set for any object $i$. 
 
Let $\CC$ be a $\Hom$-finite and locally bounded $K$-linear category, $X\in\Mod\CC$. 
It is elementary that the following conditions are equivalent:
\begin{enumerate}[\rm (i)]
\item $X$ is a finitely generated module.
\item $\supp X$ is a finite set and 
$\dim_{K} X(i)<\infty$ for any $i\in\CC$.
\end{enumerate}
If $\CC$ is a $\Hom$-finite and locally bounded 
$K$-linear category, 
then $\mod \CC$ is a $\Hom$-finite abelian category.  
 We define $\CC$-module $I_{\CC}(i)$ by 
 \[I_{\CC}(i)=I(i):=\Hom_{\CC}(i,-)^{\ast}.\]
 We denote by $\proj \CC$ (resp.\,$\inj \CC$), the full subcategory of $\mod\CC$ consisting of finitely generated projective 
(resp.\,injective) 
$\CC$-modules. 

Let $C$ be a $\Hom$-finite and locally bounded $K$-linear category. 
Then the $K$-linear functor $D=(-)^{\ast}:(\mod \CC^{\op})^{\op}\to\mod \CC$ is an equivalence of $K$-linear categories and the restriction functor 
$D|_{(\proj\CC^{\op})^{\op}}:(\proj\CC^{\op})^{\op}\to\inj\CC$ is an equivalence of $K$-linear categories. 
Let 
\[\text{$\nak=(-)\otimes_{\CC}\CC^{\ast}:
\Mod \CC\to \Mod \CC$, \ \ 
$\nak^{-}=\Hom_{\CC}(\CC^{\ast},-):
\Mod \CC\to \Mod \CC$}\]
be Nakayama functors. More precisely, they are 
defined as 
\[\nak(X)(i)=\Coker(\bigoplus_{j,k\in\CC}X_{j}\otimes \Hom_{\CC}(k,j)\otimes 
\Hom_{\CC}(k,i)^{\ast}\stackrel{f}{\to} 
\bigoplus_{k\in\CC} X_{k}\otimes \Hom_{\CC}(k,i)^{\ast}),\] 
\[\nak^{-}(X)(i)=\Hom_{\Mod\CC}(\Hom_{\CC}(i,-)^{\ast},X)\]
for any $X\in\Mod\CC$ and $i\in\CC$ where 
\[f: x\otimes\alpha\otimes y
\mapsto x\alpha\otimes y
-x\otimes \alpha y.\]
By the definitions of $\nak$ and $\nak^{-}$, 
$(\nak^{-},\nak)$ is a pair of adjoint functors. 

\begin{pro}
Let $\CC$ be a $\Hom$-finite and locally bounded 
$K$-linear category. 
Then the functor 
\[\nak|_{\proj \CC}:\proj \CC\to\inj \CC\]
is an equivalence of $K$-linear categories 
satisfying $\nak(P_{\CC}(i))\simeq I_{\CC}(i)$ 
and there exists a functorial isomorphism 
\[\text{$\Hom_{\mod \CC}(P,\nak(Q))
\simeq\Hom_{\mod\CC}(Q,P)^{\ast}$ 
for any $P,Q\in\proj\CC$.}\]
\end{pro}

\subsection{Serre functors}
Let $\DD$ be a triangulated category. 
Recall that a \emph{Serre functor} of $\DD$ is a $K$-linear autoequivalence $\S:\DD\to\DD$ such that 
there exists a functorial isomorphism   
\[\Hom_{\DD}(X,\S(Y))\tilde{\to} 
\Hom_{\DD}(Y,X)^{\ast}\]
for any two objects $X,Y\in\DD$. 
By the definition of a Serre functor, 
any two Serre functors are isomorphic. 
\begin{pro}\label{SF}\cite{BK}
Let $\DD$ be a triangulated category with 
a Serre functor $\S:\DD\to\DD$. 
\begin{enumerate}[\rm (a)]
\item There exists a natural isomorphism 
$\alpha:\S[1]\to[1]\S$ such that 
$(\S,\alpha):\DD\to\DD$ is 
a triangle autoequivalence. 
\item Any Serre functor of $\DD$ is isomorphic to $\S$. 
\item For any triangle equivalence 
$F:\DD\to\DD'$, $F\S F^{-1}:\DD'\to\DD'$ is a Serre functor of $\DD'$. 
\end{enumerate}
\end{pro}
A $K$-linear category $\CC$ is called an \emph{Iwanaga-Gorenstein category} if 
$\CC$ is $\Hom$-finite and for any $i\in \CC$, 
$I_{\CC}(i)\in \per\CC$ and $I_{\CC^{\op}}(i)\in \per \CC^{\op}$. 
The following result is well-known \cite{Hap}: 
\begin{pro}\label{Iwa}
For any $\Hom$-finite $K$-linear category $\CC$, 
the following conditions are equivalent:
\begin{enumerate}[\rm (i)]
\item $\CC$ is an Iwanaga-Gorenstein category. 
\item $\per \CC$ has a Serre functor $\S:\per\CC\to\per\CC$.
\end{enumerate}
If the above conditions are satisfied, the functor 
$\nak=(-)\Lotimes_{\CC} \CC^{\ast}:\per \CC\to\per \CC$ is a Serre functor of $\per \CC$ and 
$\nak^{-1}=\Rhom_{\CC}(C^{\ast},-):\per \CC\to\per \CC$ is an inverse of $\nak$. 
\end{pro}

\begin{proof}
(i)$\Rightarrow$(ii): It follows from \cite[10.4]{Ke}.\\
(ii)$\Rightarrow$(i): Let $\DD=\per \CC$. 
Since 
\[\Hom_{\DD}(P_{\CC}(j),\mathbb{S}(P_{\CC}(i))[n])
\simeq\Hom_{\DD}(P_{\CC}(i),P_{\CC}(j)[-n])^{\ast}\simeq \begin{cases}
\Hom_{\DD}(P_{\CC}(j),I_{\CC}(i)) & n=0,\\
0 & n\neq 0, 
\end{cases}\]
for any $j\in\CC$, we have 
\[H^{n}\S(P_{\CC}(i))\simeq
\begin{cases}
I_{\CC}(i) & n=0,\\
0 & n\neq 0.
\end{cases}\]
Thus $I_{\CC}(i)\simeq
\S(P_{\CC}(i))\in\DD$. 
Dually, we see that $I_{\CC^{\op}}(i)\in\per\CC^{\op}$. 
Thus the assertion follows. 
\end{proof}

\begin{lem}\label{CY lemma}
For any positive integer $p$, 
the category $\per K\A_{p}$ is a 
fractional Calabi-Yau category of dimension 
$\frac{p-1}{p+1}$ i.e.\,there exists an 
isomorphism $\nu_{K\A_{p}}^{p+1}
\tilde{\to} [p-1]$ 
of functors. 
\end{lem}

\subsection{Semi-orthogonal decompositions and 
Admissible subcategories}
Let $\DD_{1}$ and $\DD_{2}$ be full subcategories of $\DD$. 
We denote by $\DD_{1}\ast \DD_{2}$ the full subcategory consisting of objects $X$ in $\DD$ 
such that there exist 
objects $X_{1}\in \DD_{1}$, $X_{2}\in\DD_{2}$ and a triangle 
\[X_{1}\to X\to X_{2}\to X_{1}[1]\ \ 
\text{in $\DD$.}\] 
Then the operation 
$\ast$ is associative and we define the full subcategory 
$\DD_{1}\ast\DD_{2}\ast\dots\ast \DD_{n}$ 
for full subcategories 
$\DD_{1},\DD_{2},\dots,\DD_{n}$ 
inductively. 

The sequence $(\DD_{1},\DD_{2},\dots,\DD_{n})$ of thick subcategories $\DD_{i}$ of $\DD$ is called a \emph{semi-orthogonal decomposition} of $\DD$ if 
$\DD=\DD_{1}\ast \DD_{2}\ast\dots\ast \DD_{n}$ 
and $\Hom_{\DD}(\mathcal{D}_{k},\mathcal{D}_{k'})=0$ for any integers $k, k'\in[1,n]$ such that $k<k'$. 
In this case, we denote 
$\DD_{1} \ast\DD_{2}\ast\dots\ast\mathcal{D}_{n}$
by \[\DD_{1} \perp\DD_{2}\perp\dots\perp\DD_{n}.\]

The properties of admissible subcategories are detailed in \cite{Huy}. Let $\DD$ be a triangulated category. 
A thick subcategory $\EE$ of $\DD$ is said to be 
\emph{right admissible} 
(resp.\,\emph{left admissible}) 
if the natural embedding $\I_{\EE}^{\DD}
:\EE\to \DD$ 
has a right adjoint functor 
$\T^{\DD}_{\EE}:\DD\to\EE$ 
(resp.\,a left adjoint functor 
$\F^{\DD}_{\EE}:\DD\to\EE$). 
If a thick subcategory $\EE$ of $\DD$ is right admissible and left admissible, $\EE$ is said to be \emph{admissible}. 
We often simply denote 
$\I_{\EE}^{\DD}$ (resp.\,$\T_{\EE}^{\DD}$, 
$\F_{\EE}^{\DD}$) by 
$\I_{\EE}$ (resp.\,$\T_{\EE}$, $\F_{\EE}$). 
By the definitions of $\T^{\DD}_{\EE}$ and 
$\F^{\DD}_{\EE}$, we have 
\begin{gather}\numberwithin{equation}{section}
\text{$\Hom_{\DD}(X,\I_{\EE}^{\DD}(X'))
\simeq\Hom_{\EE}(\F_{\EE}^{\DD}(X),X')$ for 
any objects $X\in \DD$, $X'\in \EE$,}\\
\text{$\Hom_{\DD}(\I_{\EE}^{\DD}(Y'),Y)\simeq\Hom_{\EE}
(Y',\T_{\EE}^{\DD}(Y))$ for any objects 
$Y\in \DD$, $Y'\in \EE$.}
\end{gather}
A thick subcategory $\EE$ of $\DD$ is 
admissible if and only if 
\begin{gather}
\DD=\EE\perp\EE^{\perp}
={}^{\perp}\EE\perp\EE
\end{gather}
where 
\[\EE^{\perp_{\DD}}=\EE^{\perp}
:=\{X\in\DD\mid 
\text{$\Hom_{\DD}(Y,X)=0$ 
for any $Y\in\EE$}\},\]
\[{}^{\perp_{\DD}}\EE={}^{\perp}\EE
:=\{X\in\DD\mid 
\text{$\Hom_{\DD}(X,Y)=0$ 
for any $Y\in\EE$}\}.\]
In particular, the unit and counit morphisms induce triangles 
\begin{gather}\numberwithin{equation}{section}
\T^{\DD}_{\EE}(X)\to X
\to \F^{\DD}_{\EE^{\perp}}(X)\to \T^{\DD}_{\EE}(X)[1],\\
\T^{\DD}_{^{\perp}\EE}(X)\to X\to \F^{\DD}_{\EE}(X)
\to \T^{\DD}_{^{\perp}\EE}(X)[1].
\end{gather} 

The following three facts are elementary:
\begin{lem}\cite{Mac}\label{Adj1}
Let $\DD$ be a triangulated category, 
$\EE$ an admissible subcategory of $\DD$, 
and $\FF$ a thick subcategory of $\EE$.
\begin{enumerate}[\rm (a)]
\item $\FF$ is an admissible subcategory of $\EE$ if and only if 
$\FF$ is an admissible subcategory of $\DD$. 
\item If the condition in $(\rm a)$ is satisfied, then 
\[\T^{\DD}_{\FF}\simeq\T^{\EE}_{\FF}\T^{\DD}_{\EE} 
,\ \ \F^{\DD}_{\FF}
\simeq\F^{\EE}_{\FF}\F^{\DD}_{\EE} ,\ \  
\T^{\EE}_{\FF}\simeq\T^{\DD}_{\FF}|_{\EE},\ \ 
\F^{\EE}_{\FF}\simeq\F^{\DD}_{\FF}|_{\EE}.\]
\end{enumerate}

\end{lem}

\begin{lem}\label{Adj2}
Let $\DD$ be a triangulated category, 
$\EE$ a thick subcategory of $\DD$. If there exist admissible 
subcategories $\FF$ and $\FF'$ of $\DD$ 
such that $\EE=\FF\perp \FF'$, 
then $\EE$ is an admissible 
subcategory of $\DD$.
\end{lem}

\begin{proof}
Since $\FF$ is an admissible subcategory of 
$\DD$, we have 
$\DD=\FF\perp\GG$ 
where $\GG=\FF^{\perp_{\DD}}$. 
Since $\FF'$ is an admissible subcategory of $\DD$ and $\FF'\subset\GG$, 
it follows from Lemma \ref{Adj1} that 
$\FF'$ is an admissible subcategory of $\GG$. 
So $\GG=\FF'\perp\GG'$ where 
$\GG'=\FF'^{\perp_{\GG}}$. 
Since \[\DD=\FF\perp\GG
=\FF\perp 
\FF'\perp\GG'=\EE\perp\GG',\] 
$\EE$ is a right admissible 
subcategory of $\DD$. 
Dually,  
$\EE$ is a left admissible subcategory of $\DD$.  
\end{proof}

\begin{lem}\label{Adj3}
Let $\DD$ be a triangulated category and 
let $\EE$, $\EE'$, $\FF$, $\FF'$ be four admissible subcategories of $\DD$ such that 
$\DD=\EE\perp\FF,\ \ \EE'\subset\EE,\ \ 
\FF'\subset\FF.$ 
Then for $\DD'=\EE'\perp\FF'$, we have 
\[\T_{\EE}^{\DD}|_{\DD'}
\simeq\T_{\EE'}^{\DD'},\ \ \F_{\FF}^{\DD}|_{\DD'}
\simeq\F_{\FF'}^{\DD'}.\]
\end{lem}

\begin{proof}
For any $X\in\DD'$, there exists a triangle 
\[\T_{\EE'}^{\DD'}(X)
\to X
\to\F_{\FF'}^{\DD'}(X)
\to\T_{\EE'}^{\DD'}(X)[1].\]
Since $\T_{\EE'}^{\DD'}(X)\in\EE$ and 
$\F_{\FF'}^{\DD'}(X)\in\FF$, the assertion follows. 
\end{proof}
The following observation should be well-known but we could not find a reference. 
We give a proof for the convenience of the reader. 
\begin{pro}\label{Serre functor}
Let $\DD$ be a triangulated category with a Serre functor $\S=\S_{\DD}$, 
$\EE$ a left or right admissible subcategory of $\DD$. 
\begin{enumerate}[\rm (a)]
\item $\EE$ is an admissible subcategory of $\DD$ if and only if $\EE$ has a Serre functor 
$\S_{\EE}$.
\end{enumerate}
If the conditions in $(\rm{a})$ are satisfied, 
then the following two assertions hold:
\begin{enumerate}[\rm(a)]
\item[\rm(b)] The functor 
$\S_{\EE}$ in \rm{(a)} satisfies $\S_{\EE}
\simeq\T_{\EE}\S|_{\EE}$ and $\S^{-1}_{\EE}\simeq\F_{\EE}\S^{-1}|_{\EE}$. 
\item[\rm(c)] If $X$, $\S(X)\in\EE$, then 
$\S_{\EE}(X)\simeq \S(X)$. 
\end{enumerate}\end{pro} 

\begin{proof}
(a) We prove ``only if'' part. 
We show that $\T_{\EE}\S|_{\EE}$ is a Serre functor. 
Since for any $X,Y\in \EE$, there exist functorial isomorphisms 
\begin{gather}
\numberwithin{equation}{section}
\text{$\Hom_{\EE}(X,\T_{\EE}\S(Y))
=\Hom_{\DD}(X,\S(Y))
\simeq \Hom_{\DD}(Y,X)^{\ast}
\simeq \Hom_{\EE}(Y,X)^{\ast},$}\\
\text{$\Hom_{\EE}(\F_{\EE}\S^{-1}(X),Y)
=\Hom_{\DD}(\S^{-1}(X),Y)
\simeq \Hom_{\DD}(Y,X)^{\ast}
\simeq \Hom_{\EE}(Y,X)^{\ast}$}, 
\end{gather}
it follows from \cite[Lem.\,1.1.5]{RV}, that 
$\T_{\EE}\S|_{\EE}$ is a Serre functor of 
$\EE$ with an inverse $\F_{\EE}\S^{-1}|_{\EE}$.\\ 
We prove ``if'' part. 
Let $\T_{\EE}$ be a right adjoint functor of $\I_{\EE}$, and $\S_{\EE}$ a Serre functor of 
$\EE$. 
For any $X, Y\in \EE$, 
there exist functorial isomorphisms 
\begin{align*}
\Hom_{\EE}(\S^{-1}_{\EE} \T_{\EE}\S_{\DD}(Y),X)
&\simeq\Hom_{\EE}(X,\T_{\EE}\S_{\DD}(Y))^{\ast}
\simeq\Hom_{\DD}(X,\S_{\DD}(Y))^{\ast}
\simeq\Hom_{\DD}(Y,X).
\end{align*}
Thus $\S^{-1}_{\EE}\T_{\EE}\S_{\DD}$ is a left adjoint functor of $\I_{\EE}$. \\
(b) In the proof of (a), we proved that 
$\T_{\EE}\S|_{\EE}$ is a Serre functor. By the uniqueness of Serre functor, the assertion follows.\\
(c) This is clear by (b). 
\end{proof}

The proof of the following result is clear from 
Proposition \ref{Serre functor}.
 
\begin{lem}\label{App functor}
Let $\DD$ be a triangulated category with a Serre functor $\S=\S_{\DD}$, and let 
$\EE$ and $\FF$ be two admissible subcategories of $\DD$. 
\begin{enumerate}[\rm (a)] 
\item The functor 
$\F_{\FF}^{\DD}|_{\EE}:\EE\to\FF$ has a right adjoint functor $\T_{\EE}^{\DD}|_{\FF}:\FF\to\EE$.
\item If $\S^{-1}(\FF)\subset\EE$, 
we have an isomorphism 
$\T_{\EE}^{\DD}|_{\FF}
\simeq \S_{\EE}\S^{-1}|_{\FF}:\FF\to\EE$ of 
functors.
\item If $\S(\EE)\subset\FF$, we have an isomorphism $\F_{\FF}^{\DD}|_{\EE}
\simeq \S^{-1}_{\FF}\S|_{\EE}:\EE\to\FF$ of 
functors. 
\item If $\S(\EE)=\FF$, then the functors 
$\F_{\FF}^{\DD}|_{\EE}:\EE\to\FF$ and 
$\T_{\EE}^{\DD}|_{\FF}:\FF\to\EE$ are mutually 
inverse triangle equivalences. 
\end{enumerate}
\end{lem}

\begin{proof}
(a) For any $X\in\EE$ and $Y\in\EE$, we have 
functorial isomorphisms 
\begin{align*}
\Hom_{\FF}(\F_{\FF}^{\DD}(X),Y)
\simeq\Hom_{\DD}(X,Y)
=\Hom_{\EE}(X,\T_{\EE}^{\DD}(Y)).
\end{align*}
Thus the assertion follows.\\
(b) Since $\EE$ is an admissible subcategory, 
$\EE$ has a Serre functor 
$\S_{\EE}=\T_{\EE}^{\DD}\S|_{\EE}:\EE\to\EE$  
by Proposition \ref{Serre functor}. 
Since $\S^{-1}(\FF)\subset\EE$, 
Lemma \ref{Serre functor} (b) implies 
$\S_{\EE}\S^{-1}|_{\FF}\simeq\T_{\EE}^{\DD}\S\S^{-1}|_{\FF}
\simeq\T_{\EE}^{\DD}|_{\FF}.$\\
(c) This is the dual of (b).\\
(d) Since the functor $\F_{\FF}^{\DD}|_{\EE}$ is 
the composition of $\S|_{\EE}:\EE\to\FF$ and 
$\S_{\FF}:\FF\to\FF$ by (c), we have that 
it is a triangle equivalence. 
By (a), $\T_{\EE}^{\DD}|_{\FF}$ is an inverse of $\F_{\FF}^{\DD}|_{\EE}$. 
\end{proof}

\subsection{Tilting objects and Exceptional sequences}
Let $\DD$ be a triangulated category. 
An object $T\in\DD$ is called 
a \emph{pretilting object} if 
$\Hom_{\DD}(T,T[n])\simeq 0$ for any integer $n\neq0$. 
A pretilting object $T$ is called a \emph{tilting object} if $\DD=\langle T\rangle$. 
An additive category $\DD$ is said to be \emph{idempotent complete} if any idempotent morphism $e:X\to X$ (i.e.\,endomorphism $e:X\to X$ satisfying $e^{2}=e$) in $\DD$, 
there exist two morphisms $f:X\to Y$ and 
$g:Y\to X$ in $\DD$ such that $fg=1_{Y}$ and $gf=e$. 
A triangulated category $\DD$ is said to be 
\emph{algebraic} if $\DD$ is 
triangle equivalent to the stable category of a Frobenius category. 
\begin{pro}\cite{Ke}\label{tilt1}
Let $\DD$ be an algebraic triangulated 
category, 
$T$ a pretilting object in $\DD$, 
and $A=\End_{\DD}(T)$.  
If $\DD$ is idempotent complete, 
there exists a triangle equivalence 
$F:\langle T\rangle\to\per A$ such that 
$F(T)\simeq A$. 
\end{pro}

The following result is clear by Proposition 
\ref{tilt1} and $F$ is directly given. 
\begin{ex}\label{tilt2}
Let $A$ be a $K$-algebra such that $\hd A<\infty$, 
and let $(e_{k})_{k\in[1,n]}$ be a 
complete set of primitive orthogonal idempotents of $A$, and 
$\displaystyle e=\sum_{k\in[1,m]} e_{k}$, $B=eAe$. Let 
\[P=eA=\bigoplus_{k\in[1,m]} P_{A}(k),\ \ 
P'=(1-e)A=\bigoplus_{k\in[m+1,n]} P_{A}(k),\] \[S=\bigoplus_{k\in[1,m]} S_{A}(k),\ \  S'=\bigoplus_{k\in[m+1,n]} S_{A}(k).\]
\begin{enumerate}[\rm (a)] 
\item The sequence 
$(e_{k})_{k\in[1,m]}$ is a complete set of primitive orthogonal idempotents of $B$, and 
there exists a triangle equivalence 
$F=\Rhom_{A}(P,-):\langle P\rangle
\to \per B$ 
 such that  
\begin{gather}
\numberwithin{equation}{section}
\text{$F(P_{A}(i))\simeq P_{B}(i)$ for any 
$i\in[1,m]$.}
\end{gather}
\item If $\Hom_{\per A}(P',P)=0$, then 
the following conditions are satisfied:
\begin{gather}
\numberwithin{equation}{section} 
\text{$\per A=\langle P'\rangle\perp 
\langle P\rangle=\langle S\rangle\perp 
\langle S'\rangle$,\ \ 
$\langle P\rangle=\langle S\rangle$. }\\
\text{$F(S_{A}(i))\simeq S_{B}(i)$ for any $i\in[1,m]$. }
\end{gather}
\end{enumerate}
\end{ex}

\begin{proof}
(a) The functor $F=(-)\Lotimes_{A} Ae:\per A\to \per B $ has a left adjoint functor 
$G=(-)\Lotimes_{B} eA:\per B \to \per A$. 
Since 
$eA\Lotimes_{A} Ae\simeq B$ 
as $(B,B)$-bimodules, we have 
$FG\simeq \Id$. Thus  
$F:\langle P\rangle\to \per B $ is a triangle equivalence such that 
\[F(P_{A}(i))\simeq\Rhom_{A}(P,P_{A}(i))
\simeq \Hom_{\per A }(P,P_{A}(i))\simeq e_{i}Ae\simeq P_{B}(i).\] 
(b) Since $\Hom_{\per A}(P,P'[n])\simeq 0$ 
for any integer $n$, we have $\langle P\rangle
=\langle S\rangle$. Since 
$\Hom_{\per A }(P,S'[n])\simeq 0$ for any integer $n$, 
we have that 
$\Hom_{\per A }(S,S'[n])\simeq 0$ for any integer $n$. 
Thus 
\[\per A=\langle S\oplus S'\rangle
=\langle S\rangle\perp \langle S'\rangle.\]
Since 
\[F(S_{A}(i))\simeq 
\Rhom_{A}(P,S_{A}(i))
\simeq \Hom_{\per A}(P,S_{A}(i))
\simeq(e_{i}A/A(1-e_{i})A)e
\simeq S_{B}(i)\]
as $B$-modules, we have $F(S_{A}(i))\simeq 
S_{B}(i)$.
\end{proof}

A pretilting object $E$ is called an 
\emph{exceptional object} in $\DD$ if $\End_{\DD}(E)\simeq K$. 
\begin{lem}\cite[Lem.\,1.58]{Huy}\label{Exc0}
Let $E$ be an exceptional object in $\DD$. Then for any object $X\in\langle E\rangle$, 
\[X\simeq\bigoplus_{n\in\Z} E^{\oplus d_{n}}[n]\]
where $d_{n}=\dim_{K}\Hom_{\DD}(E,X[-n])
=\dim_{K}\Hom_{\DD}(X,E[n])$. 
\end{lem}

We recall the notion of exceptional sequences 
which is slightly modified for later use in 
this paper. In fact 
we allow the index set to be a finite 
totally ordered set. 

\begin{dfn}\cite{Huy}
Let $\DD$ be a triangulated category, 
$S$ a finite totally ordered set. 
A family $(E_{k})_{k\in S}$ of objects 
in $\DD$ indexed by $S$ is called an 
\emph{exceptional 
sequence} if the following conditions are satisfied:
\begin{enumerate}[\rm (E1)]
\item $E_{k}$ is an exceptional object 
for any $k\in S$. 
\item If $k< k'$, then 
$\Hom_{\DD}(E_{k},E_{k'}[n])=0$ for any integer $n$.
\item $\langle E_{k}\mid k\in S\rangle$ is an admissible subcategory of $\DD$. 
\end{enumerate}
An exceptional sequence $(E_{k})_{k\in S}$ is 
said to be full if 
\[\langle E_{k}\mid k\in S\rangle=\DD.\] 
\end{dfn}
The properties of exceptional sequences are detailed in \cite{Huy}. 
In general, for any finite totally orderd set $S$, by using the unique ordered isomorphism  
$s:[1,n]\to S$, 
we can identify $(E_{k})_{k\in S}$ with 
$(E_{s(k)})_{k\in [1,n]}$. If $S=[1,n]$, we denote the family $(E_{k})_{k\in S}$ 
by $(E_{1},E_{2},\dots,E_{n})$. 

A triangulated category $\DD$ is said to be 
\emph{$Ext$-finite} over $K$ if 
$\displaystyle\bigoplus_{n\in\Z} \Hom_{\DD}(X,Y[n])$ has a finite dimension for any objects 
$X, Y\in \DD$.

\begin{lem}\cite[Lem.\,1.58]{Huy}
\label{Exc1}
Let $\DD$ be an $\Ext$-finite triangulated category, $E$ an exceptional object in 
$\DD$. Then $(\rm E3)$ is satisfied. 
\end{lem}

Let $(E_{k})_{k\in S}$ be an exceptional sequence in $\DD$. 
For any $T\subset S$, define 
\[E_{T}
:=\bigoplus_{k\in T}E_{k}.\]
\begin{lem}\label{Exc2}
Let $\DD$ be an $\Ext$-finite triangulated category, $(E_{k})_{k\in S}$ an 
exceptional sequence in $\DD$. 
For any $T\subset S$, $\langle E_{T}\rangle$ is an admissible subcategory of $\DD$.
\end{lem}

\begin{proof}
Without loss of generality, we can assume $T=[1,n]$. 
Since 
\[\langle E_{T}\rangle
=\langle E_{1}\rangle\perp
\langle E_{2}\rangle\perp\dots
\perp\langle E_{n}\rangle\]
and $\langle E_{k}\rangle$ are 
admissible subcategories of $\DD$ by Lemma 
\ref{Exc1}, 
$\langle E_{T}\rangle$ is also an admissible 
subcategory of $\DD$ by Lemma \ref{Adj2}. 
\end{proof}

\begin{lem}\label{Exc3}
Let $\DD$ be an $\Ext$-finite triangulated category with a Serre functor $\S:\DD\to\DD$, 
and let \[(E_{1},E_{2},\dots,E_{n})\] be an exceptional sequence in $\DD$. 
\begin{enumerate}[\rm (a)]
\item The sequence 
\[(E_{2},\dots,E_{n},\S_{\langle E_{[1,n]}\rangle}(E_{1}))\] is an exceptional sequence 
and $\langle E_{[1,n]}\rangle
=\langle E_{[2,n]}\rangle\perp
\langle\S_{\langle E_{[1,n]}\rangle}(E_{1})\rangle.$ 
\item The sequence 
\[(\S^{-1}_{\langle E_{[1,n]}\rangle}(E_{n}),
E_{1},\dots,E_{n-1})\] is an exceptional sequence 
and $\langle E_{[1,n]}\rangle
=\langle\S^{-1}_{\langle E_{[1,n]}\rangle}(E_{n})\rangle\perp\langle E_{[1,n-1]}\rangle.$ 
\end{enumerate}
\end{lem}

\begin{proof}
(a) Since $\S_{\langle E_{[1,n]}\rangle}:\langle E_{[1,n]}\rangle\to\langle E_{[1,n]}\rangle$ is 
a triangle equivalence and $E_{1}$ is an exceptional object, $\S_{\langle E_{[1,n]}\rangle}(E_{1})$ is also an exceptional object. 
If $k\in [2,n]$, then 
\[\Hom_{\DD}(E_{k},\S_{\langle E_{[1,n]}\rangle}(E_{1})[n])
\simeq\Hom_{\DD}(E_{1},E_{k}[-n])^{\ast}=0\]
for any integer $n$. Thus the sequence 
\[(E_{2},\dots,E_{n},\S_{\langle E_{[1,n]}\rangle}(E_{1}))\]
is an exceptional sequence. 
Since 
\[\T_{\langle E_{1}\rangle}\S_{\langle E_{[1,n]}\rangle}(E_{1})
\stackrel{\ref{Serre functor}\rm{(b)}}{=}\S_{\langle E_{1}\rangle}(E_{1})
\stackrel{\ref{Exc0}}{=} E_{1},\]
there exists a triangle 
\[E_{1}\to \S_{\langle E_{[1,n]}\rangle}(E_{1})\to \F_{\langle E_{[2,n]}\rangle}\S_{\langle E_{[1,n]}\rangle}(E_{1})\to E_{1}[1],\]
thus the last assertion follows.\\
(b) This is the dual of (a). 
\end{proof}

\begin{lem}\label{Exc4}
Let $\DD$ be an $\Ext$-finite triangulated category with a Serre functor $\S:\DD\to\DD$, 
and let $(E_{k})_{k\in[1,n]}$ be a full exceptional sequence in $\DD$. 
If there exist integers 
$p<q$ such that 
\begin{gather}
\numberwithin{equation}{section}
\label{north}\text{$\Hom_{\DD}(E_{[q+1,n]},\S_{\langle E_{[p,q]}\rangle}(E_{p})[n])=0$ 
for any integer $n$,}
\end{gather}
then \[\T_{\langle E_{[1,n]\backslash\{p\}}\rangle}^{\DD}(E_{p})
\simeq\T_{\langle E_{[p+1,q]}\rangle}
^{\langle E_{[p,q]}\rangle}(E_{p}).\]
\end{lem}

\begin{proof}
Let $\DD'=\langle E_{[p,q]}\rangle$, 
$\EE=\langle E_{[1,n]\backslash\{p\}}\rangle
$, $\FF=\langle \S_{\langle E_{[p,q]}\rangle}(E_{p})\rangle$, $\EE'=\langle E_{[p+1,q]}\rangle$. 
Then we have semi-orthogonal decompositions 
\[\DD'
=\langle E_{p}\rangle\perp
\langle E_{[p+1,q]}\rangle
\stackrel{\ref{Exc3}(a)}{=}
\EE'\perp\FF \ \  \text{and}\]
\begin{align*} 
\DD &=\langle E_{[1,p-1]}\rangle\perp
\langle E_{p}\rangle\perp
\EE'\perp
\langle E_{[q+1,n]}\rangle
\stackrel{\ref{Exc3}(a)}{=}\langle E_{[1,p-1]}\rangle\perp
\EE'\perp
\FF\perp
\langle E_{[q+1,n]}\rangle\\
&\stackrel{(\ref{north})}
{=}\langle E_{[1,p-1]}\rangle\perp
\EE'\perp
\langle E_{[q+1,n]}\rangle\perp
\FF
=\EE\perp\FF. 
\end{align*}
Thus $\T_{\EE}
^{\DD}(E_{p})
\stackrel{\ref{Adj3}}{=}
\T_{\EE'}^{\DD'}(E_{p})$ 
and the assertion follows.
\end{proof}

\begin{lem}\label{EB}
Let $A$, $B$ be two $K$-algebras, 
$E$ an exceptional object in $\per A$. 
Then 
\[F=E\otimes(-):\per B
\to\per(A\otimes B)\]
is a fully faithful triangle functor and 
induces a triangle equivalence 
$F:\per B\to \langle E\otimes B\rangle$. 
\end{lem}
 
\begin{proof}
Since $E\otimes B$ is a pretilting object and 
the morphism $f:B\to\REnd_{A\otimes B}(E\otimes B);b\to 1\otimes b$ is a quasi-isomorphism of dg algebras, the functor 
\[F'=\Rhom_{A\otimes B}(E\otimes B,-):\langle E\otimes B\rangle
\to \per B\]
is a triangle equivalence. 
Since  $F'F\simeq \Id$, the assertion follows. 
\end{proof}

\begin{lem}\label{Id and Serre}
Let $A$ be a $K$-algebra, $B$ an Iwanaga-Gorenstein $K$-algebra, $C=A\otimes B$. 
Then $G=(-)\Lotimes_{C}(A\otimes B^{\ast})
:\per C\to \per C$ is a triangle autoequivalence, 
and for any exceptional object $E\in\per A$, 
$G|_{\langle E\otimes B\rangle}:
\langle E\otimes B\rangle\to 
\langle E\otimes B\rangle$ is a Serre 
functor. 
\end{lem}

\begin{proof}
Since $B$ is an Iwanaga-Gorenstein algebra, 
$G=(-)\Lotimes_{C}(A\otimes B^{\ast})
:\per C\to \per C$ is a triangle equivalence. 
By Lemma \ref{EB}, 
for any object $X\in \langle E\otimes B\rangle$, 
there exists an object $Y\in \per B$ such that 
$X\simeq E\otimes Y$. 
Since there exist functorial isomorphisms 
\[\Hom_{\per C}(E\otimes Y,G(E\otimes Y'))
\simeq\Hom_{\per C}
(E\otimes Y,E\otimes \nu_{B}(Y'))
\simeq \Hom_{\per C}(E\otimes Y',E\otimes Y)^{\ast},\]
we have that $G|_{\langle E\otimes B\rangle}$ is a Serre functor. 
\end{proof}

\section{$S$-families}
\subsection{Weak $S$-families}
In this section, let $\DD$ be a triangulated category satisfying the following conditions: 
\begin{gather}\label{Good}
\text{$\DD$ is algebraic, idempotent complete, 
$\Ext$-finite and has 
a Serre functor $\S$.}
\end{gather}
Let $S$ be a finite subset of $\Z^{2}$. 
For any element $(i,j)\in S$, 
\[S_{i,j}:=([i-1,i]\times[j-1,j])\cap S.\]

\begin{ex}\label{ExS}
Let $I_{1}=[1,5]$, $I_{2}=I_{3}=[1,3]$, 
$I_{4}=\{1\}$ be intervals of $\Z$. 
The figure of $\displaystyle S=\bigsqcup_{i\in[1,4]} \{i\}\times I_{i}$ 
is the following:
\[\small\begin{tikzpicture}[auto]
\draw (0,0)node[]{$(1,1)$};
\draw (1,0)node[]{$(1,2)$};
\draw (2,0)node[]{$(1,3)$};
\draw (3,0)node[]{$(1,4)$};
\draw (4,0)node[]{$(1,5)$};
\draw (0,-2/3)node[]{$(2,1)$};
\draw (1,-2/3)node[]{$(2,2)$};
\draw (2,-2/3)node[]{$(2,3)$};
\draw (0,-4/3)node[]{$(3,1)$};
\draw (1,-4/3)node[]{$(3,2)$};
\draw (2,-4/3)node[]{$(3,3)$};
\draw (0,-2)node[]{$(4,1)$};
\draw (1/2,-1/3)--(5/2,-1/3)--(5/2,-5/3)
--(1/2,-5/3)--(1/2,-1/3);
\end{tikzpicture}\]
In the above figure, the square shows the subset $S_{3,3}$. 
\end{ex}

\begin{dfn}\label{def WLS}
Let $\DD$ be a triangulated category satisfying 
$(\ref{Good})$. 
A family $(X_{i,j})_{(i,j)\in S}$ of objects 
in $\DD$ indexed by a finite subset $S$ of 
$\Z^{2}$ is called a \emph{weak $S$-family} if the following conditions are satisfied: 
\begin{enumerate}[\rm (L1)]
\item $X_{i,j}$ is an exceptional object for any $(i,j)\in S$.
\item $\Hom_{\DD}(X_{i,j},X_{i',j'}[n])=0$ for any integer $n$ unless $(i',j')\in S_{i,j}$.
\end{enumerate}
A weak $S$-family $(X_{i,j})_{(i,j)\in S}$ is said to be full if 
\[\langle X_{i,j}\mid (i,j)\in S\rangle=\DD.\]
\end{dfn}

The name of $S$-family comes from 
lattices. 

\begin{rem}\label{rem of L2}
The condition $(\rm L2)$ is satisfied 
if and only if the following conditions are 
satisfied: 

\begin{enumerate}[\rm (L2.1)]
\item $\Hom_{\DD}(X_{i,j},X_{i',j'}[n])=0$ 
for any integer $n$ unless 
$i'\in [i-1,i]$. 
\item $\Hom_{\DD}(X_{i,j},X_{i',j'}[n])=0$ 
for any integer $n$ unless 
$j'\in [j-1,j]$.
\end{enumerate}
\end{rem}

If $(X_{i,j})_{(i,j)\in S}$ is a weak 
$S$-family, then for any $T\subset S$, a family 
$(X_{i,j})_{(i,j)\in T}$ is a weak $T$-family.  
Let  
$(X_{i,j})_{(i,j)\in S}$ be a family of 
objects in $\DD$. 
For any finite subset $T$ of $S$, let 
\begin{gather}\label{}
X_{T}:=\bigoplus_{(i,j)\in T} X_{i,j}\in\DD.
\end{gather} 
In particular, $X_{k}$ and $X^{k}$ 
are defined as 
\begin{gather}\label{}
X_{k}:=X_{\{k\}\times S_{k}}
=\bigoplus_{j\in S_{k}} X_{k,j}, 
\ \ X^{k}:
=X_{S^{k}\times\{k\}}
=\bigoplus_{i\in S_{k}} X_{i,k}
\end{gather} 
\[\text{where $S_{k}=\{j\in\Z \mid 
(k,j)\in S\},\ \ 
S^{k}=\{i\in\Z\mid (i,k)\in S\}$.}\]
\begin{lem}
Let $(X_{i,j})_{(i,j)\in S}$ be a weak 
$S$-family in $\DD$. Then 
$\langle X_{S}\rangle$ is an admissible subcategory of $\DD$. 
In particular, $\langle X_{i}\rangle$ and 
$\langle X^{j}\rangle$ are admissible subcategories.  \end{lem}

\begin{proof}
By restricting the lexicographic order 
$\preccurlyeq $ of $\Z^{2}$, we regard $S$ as a totally ordered set. 
The lexicographic order of $S$ in Example \ref{ExS} is as the following. 
\begin{gather}\label{lex}
\small \vcenter{
\xymatrix@=15pt{
(1,1)\ar[r] & (1,2) \ar[r] & (1,3) \ar[r]  
& (1,4) \ar[r]  & (1,5)\ar[lllld]\\
(2,1)\ar[r] & (2,2) \ar[r] & (2,3)\ar[lld] \\
(3,1)\ar[r] & (3,2) \ar[r] & (3,3)\ar[lld] \\
(4,1)
}
}
\end{gather}

The condition 
$(\rm L2)$ implies the following one: 
\begin{gather}
\numberwithin{equation}{section}
\label{WLS4} \tag{\rm E2}
\text{If $(i,j)\prec(i',j')$, then 
$\Hom_{\DD}(X_{i,j},X_{i',j'}[n])=0$ for any integer $n$.}
\end{gather}
So we can regard a weak $S$-family 
$(X_{i,j})_{(i,j)\in S}$ as an exceptional sequence. Thus the 
assertion follows from Lemma \ref{Exc2}
\end{proof}
The following observation is clear from (\rm L2). 
\begin{lem}\label{Pair Lemma}
Let $(X_{i,j})_{(i,j)\in S}$ 
be a weak $S$-family in $\DD$. 
Then $X_{S}$ 
is a pretilting object 
if and only if 
$X_{S_{i,j}}$ is a pretilting object 
for any $(i,j)\in S$. 
\end{lem}

\begin{proof}
It suffices to show that 
$\Hom_{\DD}(X_{i,j},X_{i',j'}[n])=0$ for each $(i,j)$, $(i',j')\in S$ and 
$n\neq 0$. 
If $(i',j')\notin S_{i,j}$, then this holds by 
(\ref{App functor}). Otherwise, this holds since 
$X_{S_{i,j}}$ is a pretilting object.\end{proof}
\subsection{Algebras $\LL(S)$}
Let $\NNN$ be the $K$-linear category defined as 
\[\NNN
:=(K\A_{\infty}^{\infty})^{\op}/
(\rad (K\A_{\infty}^{\infty})^{\op})^{2}\] 
where $\A_{\infty}^{\infty}$ is the quiver 
\[\A_{\infty}^{\infty}
=[\dotsm\stackrel{a_{-3}}{\to} -2\stackrel{a_{-2}}{\to} -1
\stackrel{a_{-1}}{\to} 0\stackrel{a_{0}}{\to} 1\stackrel{a_{1}}{\to} 
2\stackrel{a_{2}}{\to}\dotsm],\]
and $\rad K\A_{\infty}^{\infty}
=\langle a_{n}\mid n\in\Z\rangle$. 
For any subset $I$ of $\Z$, 
let $Q_{I}$ be the full subquiver of 
$\A_{\infty}^{\infty}$ with the set $I$ of vertices. 
We define $\NN(I)$ as 
\[\NN(I):=KQ_{I}^{\op}/(\rad KQ_{I}^{\op})^{2}.\]
In particular, we define $\NN(k)$ as 
\begin{gather*}
\NN(k):=\NN([1,k]).
\end{gather*}
The typical example of a weak $S$-family is 
given by the following:
\begin{ex}\label{adm lemma}
Let 
$\LLL=\NNN\otimes \NNN.$ 
The category $\per\LLL$ is a triangulated category satisfying 
$(\ref{Good})$. 
For any finite subset $S$ of $\Z^{2}$, 
the family $(P_{\LLL}(i,j))_{(i,j)\in S}$ is a weak $S$-family in $\per\LLL$. 
\end{ex}
For any finite subset $S$ of 
$\Z^{2}$, we define the algebra $\LL(S)$ as 
\[\LL(S)=\End_{\per\LLL}(\bigoplus_{(i,j)\in S} 
P_{\LLL}(i,j)).\]
Then there exists a triangle equivalence 
\[\langle P_{\LLL}(i,j)\mid 
(i,j)\in S\rangle_{\per\LLL}\to\per\LL(S).\]
\begin{lem}
The $K$-linear category $\LLL$ is self-injective, and the category $\per\LLL$ is a triangulated category satisfying $(\ref{Good})$.  
\end{lem}

\begin{proof} 
Since \[I_{\LLL}(i,j)=P_{\LLL}(i-1,j-1)\] for any $(i,j)\in\Z^{2}$, 
$\LLL$ is self-injective. 
By Proposition \ref{Iwa},
the assertion follows.\end{proof}

\subsection{$S$-families}
\begin{dfn}\label{def LS}
Let $\DD$ be a triangulated category satisfying $(\ref{Good})$. 
Let $S$ be a finite subset of $\Z^{2}$. 
A weak $S$-family 
$(X_{i,j})_{(i,j)\in S}$ in $\DD$ is called an 
\emph{$S$-family} if the following conditions 
are satisfied:
\begin{enumerate}[\rm (S1)]
\item If $(i,j),(i,j-1)\in S$, then $\S_{\langle X_{i}\rangle}(X_{i,j})\simeq X_{i,j-1}$. 
\item If $(i,j),(i-1,j)\in S$, then $\S_{\langle X^{j}\rangle}(X_{i,j})
\simeq X_{i-1,j}$.
\item If $(i,j),(i-1,j-1)\in S$, then $\S_{\langle X_{S}\rangle}(X_{i,j})\simeq X_{i-1,j-1}$.
\end{enumerate}
An $S$-family $(X_{i,j})_{(i,j)\in S}$ is said to be full if 
\[\langle X_{i,j}\mid (i,j)\in S\rangle=\DD.\]
\end{dfn} 

A typical example of an $S$-family 
is obtained by the following result:
\begin{ex}\label{Trivial family}
For any finite subset $S\subset\Z^{2}$, 
the family $(P_{\LLL}(i,j))_{(i,j)\in S}$ 
is an $S$-family in $\per\LLL$. 
\end{ex}
\begin{proof} 
Let $X_{i,j}=P_{\LLL}(i,j)$. 
The family $(X_{i,j})_{(i,j)\in S}$  
is a weak $S$-family in $\per\LLL$ 
by Example \ref{adm lemma}. 
If $(i,j)$, $(i,j-1) \in S$, 
since 
\[\nak_{\langle P_{\LLL}(i,j')\mid j'\in\Z\rangle}
(X_{i,j})=X_{i,j-1}
\in\langle X_{i}\rangle,\]
we have 
$\nak_{\langle X_{i}\rangle}(X_{i,j})
\stackrel{\ref{Serre functor}(c)}{=} X_{i,j-1}.$ 
Dually, we have that 
if $(i,j)$, $(i-1,j)\in S$, then 
$\nak_{\langle X^{j}\rangle}(X_{i,j})\simeq X_{i-1,j}.$ 
Thus $(\rm S1)$ and $(\rm S2)$ are satisfied. If $(i,j)$, $(i-1,j-1) \in S$, 
\[\nak(X_{i,j})
\simeq X_{i-1,j-1}\in\langle X_{S}\rangle.\] 
So we have 
$\nak_{\langle X_{S}\rangle}(X_{i,j})
\stackrel{\ref{Serre functor}(c)}{=}
X_{i-1,j-1}$ and $(\rm S3)$ is satisfied.
Thus 
the assertion follows. \end{proof}

\begin{ex}{\label{Sq}}
Let $S=[1,p]\times [1,q]$. 
The family $(S(i,j)[-i-j])_{(i,j)\in S}$ is a full $S$-family in $\per (K\A_{p}\otimes K\A_{q})$. There exists a triangle equivalence 
$F:\per(K\A_{p}\otimes K\A_{q})
\to\per(\NN(p)\otimes \NN(q))$ such that 
$F(S(i,j)[-i-j])\simeq P(i,j)$. 

\end{ex}

\begin{proof}
Let $A=K\A_{p}\otimes K\A_{q}$, 
$B=\NN(p)\otimes \NN(q)$, 
\[X_{i,j}=S_{A}(i,j)[-i-j].\] 
We construct a triangle 
equivalence $G:\per A\to \per B$
such that $G(X_{i,j})=P_{B}(i,j)$. 
Since 
\[\Hom_{\per K\A_{m}}
(S_{K\A_{m}}(j)[-j],S_{K\A_{m}}(i)[-i+k])\simeq \begin{cases}
K & \text{$j\in\{i,i+1\}$, $k=0$,}\\
0 & \text{otherwise,}
\end{cases}
\]
the object $\displaystyle T_{m}=\bigoplus_{i\in[1,m]} S_{K\A_{m}}(i)[-i]$ is a tilting object in $\per K\A_{m}$ such that 
$\End_{\per K\A_{m}}(T_{m})\simeq \NN(m)$. 
Then there exists a triangle equivalence 
$F:\per K\A_{m}\to \per \NN(m)$ such that 
$F(S_{K\A_{m}}(i)[-i])\\
\simeq P_{\NN(m)}(i)$. 
Thus $T=T_{p}\otimes T_{q}$ 
is a tilting object such that 
$\End_{\per A}(T)\simeq B$, and there exists a triangle exivalence 
$G:\per A\to \per B$ such that 
$G(X_{i,j})\simeq P_{B}(i,j)$. 
By Example \ref{Trivial family}, the assertion follows. 
\end{proof}

In general, by using the following result, 
if there exists a full $S$-family 
$(X_{i,j})_{(i,j)\in S}$ in $\DD$, then 
there exists a triangle equivalence 
$F:\DD\to \per\LL(S)$ and 
$F$ sends $(X_{i,j})_{(i,j)\in S}$ to 
$(P(i,j))_{(i,j)\in S}$ 
in Example \ref{Trivial family}. 

\begin{thm}\label{Triviality}
Let $\DD$ be a triangulated category satisfying $(\ref{Good})$, 
$S$ a finite subset of $\Z^{2}$. 
For any $S$-family $(X_{i,j})_{(i,j)\in S}$ in $\DD$, $X_{S}$ is a pretilting object 
such that $\End_{\DD}(X_{S})\simeq \LL(S)$, 
and there exists a triangle equivalence 
\[F:\langle X_{S}\rangle\to\per\LL(S)\] 
such that $F(X_{i,j})\simeq P_{\LLL}(i,j)$ 
for any $(i,j)\in S$. In particular, 
there exists a fully faithful triangle functor 
$G:\langle X_{S}\rangle\to\per\LLL$ such 
that $G(X_{i,j})=P_{\LLL}(i,j)$ 
for any $(i,j)\in S$. 
\end{thm}
 
The following result is a key step of the proof of Theorem $\ref{Triviality}$:
\begin{pro}\label{nu condition}
Let $(X_{i,j})_{(i,j)\in S}$ be a weak $S$-family satisfying $(\rm S3)$. 

\begin{enumerate}[\rm (a)] 
\item The condition $(\rm S1)$ is 
equivalent to the following one:
\begin{enumerate}[\rm (S1$'$)]
\item If $(i,j),(i,j-1)\in S$, then 
$\F_{\langle X^{j-1}\rangle}(X_{i,j})
=X_{i,j-1}$.
\end{enumerate}

\item The condition $(\rm S2)$ is 
equivalent to the following one:
\begin{enumerate}[\rm (S2$'$)]
\item If $(i,j),(i-1,j)\in S$, 
then $\F_{\langle X_{i-1}\rangle}(X_{i,j})
=X_{i-1,j}$. 
\end{enumerate}
\end{enumerate}
\end{pro} 

To prove Proposition \ref{nu condition}, we prove the following result:
\begin{lem}\label{Row Lemma}
Let $S$ be a finite subset of $\Z^{2}$, 
$i$ an integer such that $S_{i}\neq \emptyset$. 
Let $(X_{i,j})_{(i,j)\in S}$ be a weak $S$-family 
satisfying $(\rm S1')$. 
Then $X_{i}$ is a pretilting object 
such that $\displaystyle\End_{\DD}(X_{i})\simeq 
\NN(S_{i})$, and there exists a triangle 
equivalence $F:\langle X_{i}\rangle\to\per\NN(S_{i})$ such that 
$F(X_{i,j})\simeq P(j)$. 
\end{lem}

\begin{proof} 
If $j, j-1\in S_{i}$, then 
\begin{align*}
\Hom_{\DD}(X_{i,j},X_{i,j-1}[n])
&\stackrel{(2.1)}{=}
\Hom_{\DD}(\F_{\langle X^{j-1}\rangle}(X_{i,j}),X_{i,j-1}[n])\\
&\stackrel{(\rm S1')}{=} \Hom_{\DD}(X_{i,j-1},X_{i,j-1}[n])
\stackrel{(\rm L1)}{=}
\begin{cases}
K & n=0,\\
0 & n\neq 0. 
\end{cases}
\end{align*}
By Lemma \ref{Pair Lemma}, $X_{i}$ is a pretilting. 
If $j'\notin[j-1,j]$, since 
$\Hom_{\DD}(X_{i,j},X_{i,j'})
\stackrel{(\rm L2)}{=}0$, 
we have 
$\End_{\DD}(X_{i})\simeq \NN(S_{i})$.\end{proof}
By symmetry, we have the following result:
\begin{lem}\label{Column Lemma}
Let $S$ be a finite subset of $\Z^{2}$, 
and $j$ an integer such that $S^{j}\neq \emptyset.$
Let $(X_{i,j})_{(i,j)\in S}$ be a weak $S$-family 
satisfying $(\rm S2')$. 
Then $X^{j}$ is a pretilting object 
such that $\displaystyle\End_{\DD}(X^{j})
\simeq \NN(S^{j})$, and there exists a triangle 
equivalence $G:\langle X^{j}\rangle\to\per\NN(S^{j})$ such that 
$G(X_{i,j})\simeq P(i)$. 
\end{lem}
Now we are ready to prove Proposition 
\ref{nu condition}. 
\begin{proof}[Proof of Proposition 
\ref{nu condition}]
(a) We prove ``if'' part. 
Since (\rm S1)$'$ is satisfied, 
by Lemma \ref{Row Lemma}, 
there exists a triangle equivalence 
$F:\langle X_{i}\rangle\to\per A$ 
such that $F(X_{i,j})\simeq P_{A}(j)$ 
where $\displaystyle A=\NN(S_{i})$. 
If $(i,j), (i,j-1)\in S$, then 
\[\S_{\langle X_{i}\rangle}(X_{i,j})
\simeq\S_{\langle X_{i}\rangle}F^{-1}(P_{A}(j))
\stackrel{\ref{SF}}{=} F^{-1}(\nak_{A}(P_{A}(j)))
\simeq F^{-1}(P_{A}(j-1))\simeq X_{i,j-1}.\]
Thus (\rm S1) is satisfied. 

We prove ``only if'' part. 
Suppose that $(i,j),(i,j-1)\in S$. 
If $(i-1,j-1)\notin S$, then 
\begin{align*}
\langle X^{j-1}\rangle
&=\langle X_{\leqslant i-2, j-1}\rangle\perp\langle X_{i,j-1}\rangle
\perp\langle X_{\geqslant i+1, j-1}\rangle\\
&\stackrel{(\rm L2)}{=}\langle X_{i,j-1}\rangle\perp
\langle X_{\leqslant i-2, j-1}\rangle
\perp\langle X_{\geqslant i+1, j-1}\rangle.
\end{align*}
Since 
\[\Hom_{\DD}(\F_{\langle X^{j-1}\rangle}(X_{i,j}),
X_{\geqslant i+1, j-1}[n])
\simeq \Hom_{\DD}(X_{i,j},
X_{\geqslant i+1 ,j-1}[n])
\stackrel{(\rm L2)}{=} 0\ \  \text{and}\]
\[\Hom_{\DD}(\F_{\langle X^{j-1}\rangle}(X_{i,j}),
X_{\leqslant i-2, j-1}[n])
\simeq \Hom_{\DD}(X_{i,j},
X_{\leqslant i-2 ,j-1}[n])
\stackrel{(\rm L2)}{=} 0\]
for any integer $n$, 
we have 
$\F_{\langle X^{j-1}\rangle}
(X_{i,j})\in \langle X_{i,j-1}\rangle$. 
Since 
\begin{align*}
\Hom_{\DD}(\F_{\langle X^{j-1}\rangle}(X_{i,j}),X_{i,j-1}[n])
&\simeq \Hom_{\DD}(X_{i,j},X_{i,j-1}[n])
\stackrel{(\rm S1)}{=} \Hom_{\DD}(X_{i,j},
\S_{\langle X_{i}\rangle}(X_{i,j}))\\
&\simeq\Hom_{\DD}(X_{i,j},X_{i,j}[n])^{\ast}\stackrel{(\rm L1)}{=}\begin{cases}
K & \text{$n=0$,}\\
0 & \text{$n\neq 0$,}
\end{cases}
\end{align*}
we have $\F_{\langle X^{j-1}\rangle}(X_{i,j})
\stackrel{\ref{Exc0}}{=} X_{i,j-1}$. 
If $(i-1,j-1)\in S$, then 
\begin{equation}\label{form0}
\F_{\langle X^{j-1}\rangle}(X_{i,j})
\stackrel{(\rm S3)}{=}\F_{\langle X^{j-1}\rangle}\S^{-1}(X_{i-1,j-1})
\stackrel{\ref{Serre functor}(b)}{=} \S^{-1}_{\langle X^{j-1}\rangle}(X_{i-1,j-1}).
\end{equation}
Since 
\begin{align*}
\Hom_{\DD}(\S^{-1}_{\langle X^{j-1}\rangle}(X_{i-1,j-1}), 
X_{\geqslant i+1,j-1}[n])
&\stackrel{(\ref{form0})}{=}\Hom_{\DD}(\F_{\langle X^{j-1}\rangle}(X_{i,j}), 
X_{\geqslant i+1,j-1}[n])\\
&\simeq \Hom_{\DD}(X_{i,j},X_{\geqslant i+1,j-1}[n])
\simeq 0
\end{align*}
 for any integer $n$, 
we have $\S^{-1}_{\langle X^{j-1}\rangle}(X_{i-1,j-1})\in 
\langle X_{\leqslant i,j-1}\rangle$, 
and so 
\begin{equation}\label{form1}
\F_{\langle X^{j-1}\rangle}(X_{i,j})
\stackrel{(\ref{form0})}{=} \S^{-1}_{\langle X^{j-1}\rangle}(X_{i-1,j-1})
\stackrel{\ref{Serre functor}(c)}{=} 
\S^{-1}_{\langle X_{\leqslant i,j-1}\rangle}(X_{i-1,j-1}).
\end{equation}
Let $T=X_{i,j-1}\oplus X_{i-1,j-1}$. Since 
\[\T_{\langle X_{i}\rangle}(X_{i-1,j-1})
\stackrel{(\rm S3)}{=} \T_{\langle X_{i}\rangle}
\S_{\langle X_{S}\rangle}(X_{i,j})
\stackrel{(\rm S1)}{=}X_{i,j-1},\]
we have 
\begin{align}\label{form2}
\Hom_{\DD}(X_{i,j-1},X_{i-1,j-1}[n])
&\simeq\Hom_{\DD}(X_{i,j-1},\T_{\langle X_{i}\rangle}(X_{i-1,j-1})[n])\\
&\simeq\Hom_{\DD}(X_{i,j-1},X_{i,j-1}[n])
\stackrel{(\rm L1)}{=}\begin{cases}
K & \text{$n=0$},\\
0 & \text{$n\neq 0$}.\notag
\end{cases}
\end{align}
So $T$ is a pretilting object such that 
$\End_{\DD}(T)\simeq \NN(2)$, 
and so there exists a triangle equivalence 
$F:\langle T\rangle\to \per B$ 
such that $F(P_{B}(1))\simeq X_{i-1,j-1}$ and 
$F(P_{B}(2))\simeq X_{i,j-1}$ 
where $B=\NN(2)$. 
Let $\EE=\langle X_{\leqslant i,j-1}\rangle$, 
$\FF=\langle X_{\leqslant i-2,j-1}\rangle$ 
and $\FF'=\langle \S^{-1}_{\langle T\rangle}(X_{i,j-1})\rangle$. 
Then 
\begin{align*}
\EE
&\stackrel{(\rm L2)}{=}\FF
\perp\langle X_{i-1,j-1}\rangle
\perp \langle X_{i,j-1}\rangle
\stackrel{\ref{Exc3}(b)}{=}\FF\perp\FF'
\perp \langle X_{i-1,j-1}\rangle\\
&\stackrel{\ref{Exc3}(b)}{=}
\langle \S_{\EE}^{-1}(X_{i-1,j-1})\rangle
\perp\FF
\perp\FF'.
\end{align*}
On the other hand, 
\begin{align*}
\EE
&\stackrel{(\rm L2)}{=}\FF\perp\langle X_{i-1,j-1}\rangle
\perp \langle X_{i,j-1}\rangle
\stackrel{\ref{Exc3}(a)}{=}\FF\perp\langle X_{i,j-1}\rangle
\perp \langle \S_{\langle T\rangle}(X_{i-1,j-1})\rangle\\
&\stackrel{(\rm L2)}{=}\langle X_{i,j-1}\rangle
\perp
\FF\perp \langle \S_{\langle T\rangle}(X_{i-1,j-1})\rangle
\stackrel{\ref{CY lemma}}{=}\langle X_{i,j-1}\rangle\perp\FF\perp\FF'.
\end{align*}
Thus $\langle \S_{\EE}^{-1}(X_{i-1,j-1})\rangle
=\langle X_{i,j-1}\rangle$. 
Since 
\[\Hom_{\DD}(\S_{\EE}^{-1}(X_{i-1,j-1}),X_{i,j-1}[n])
\simeq \Hom_{\DD}(X_{i,j-1},X_{i-1,j-1}[-n])^{\ast}
\stackrel{(\ref{form2})}{=} \begin{cases}
K & \text{$n=0$,}\\
0 & \text{$n\neq 0$,}
\end{cases}\]
we have 
$\S_{\EE}^{-1}(X_{i-1,j-1})
\stackrel{\ref{Exc0}}{=} X_{i,j-1}$, 
and so 
$\F_{\langle X^{j-1}\rangle}(X_{i,j})
\stackrel{(\ref{form1})}{=}\S_{\EE}^{-1}(X_{i-1,j-1})
\simeq X_{i,j-1}$. 
Thus (\rm S1)$'$ is satisfied.\\ 
(b) This is the dual of (a). 
\end{proof}

The following result is clear by the definition of $\LL(S)$. 
The proof is left to the reader:
\begin{lem}\label{TS Lemma 1}
Let $(X_{i,j})_{(i,j)\in S}$ be a family 
of objects $X_{i,j}\in\DD$. 
Suppose that there exists a family 
$(\upsilon_{i',j'}^{i,j}:
X_{i,j}\to X_{i',j'})_{(i,j),(i',j')\in S}$ of (possibly zero) morphisms and 
families $(c_{i,j})_{(i,j)\in S}$, 
$(c'_{i,j})_{(i,j)\in S}$ of non-zero scalars 
$c_{i,j}, c'_{i,j}\in K^{\times}$
satisfying the following conditions: 
\begin{gather}
\numberwithin{equation}{section}
\label{LSM1}\text{$\Hom_{\DD}(X_{i,j},X_{i',j'})
=K\upsilon_{i',j'}^{i,j}$.}\\
\label{LSM2}\text{$(i',j')\in S_{i,j}$ if and only if 
$\upsilon_{i',j'}^{i,j}\neq 0$.}\\
\label{LSM3}\text{If $(i,j)$, 
$(i,j-1)$, $(i-1,j-1)\in S$, then 
$c_{i,j}\upsilon_{i-1,j-1}^{i,j-1}
\upsilon_{i,j-1}^{i,j}
=\upsilon_{i-1,j-1}^{i,j}$.}\\
\label{LSM4}\text{If $(i,j)$, 
$(i-1,j)$, $(i-1,j-1)\in S$, then 
$c'_{i,j}\upsilon_{i-1,j-1}^{i-1,j}
\upsilon_{i-1,j}^{i,j}
=\upsilon_{i-1,j-1}^{i,j}$.}
\end{gather}
Then there exists an isomorphism 
$f:\End_{\DD}(X_{S})\to\LL(S)$ of $K$-algebras such that $P(i,j)_{f}
\simeq \Hom_{\DD}(X_{S},X_{i,j})$ as 
$\End_{\DD}(X_{S})$-modules where 
$P(i,j)_{f}$ is an $\End_{\DD}(X_{S})$-module 
associated with $f$ and $P(i,j)$. 
\end{lem}
Now we are ready to prove Theorem \ref{Triviality}.
\begin{proof}[Proof of Theorem \ref{Triviality}]
Since $(X_{i,j})_{(i,j)\in S}$ is a weak 
$S$-family, we have 
\[\dim_{K}\Hom_{\DD}(X_{i,j},X_{i-1,j}[n])
\stackrel{\ref{Row Lemma}}{=} \begin{cases}
1 & n=0,\\
0 & n\neq 0, 
\end{cases}\]
\[\dim_{K}\Hom_{\DD}(X_{i,j},X_{i,j-1}[n])
\stackrel{\ref{Column Lemma}}{=}
\begin{cases}
1 & n=0,\\
0 & n\neq 0, 
\end{cases}\] 
\[\dim_{K}\Hom_{\DD}(X_{i,j},X_{i-1,j-1}[n])
\stackrel{(\rm S3)}{=} 
\dim_{K}\Hom_{\DD}(X_{i,j},
\S_{\langle X_{S}\rangle}(X_{i,j})[n])\]
\[=\dim_{K}\Hom_{\DD}(X_{i,j},X_{i,j}[-n])^{\ast}
\stackrel{(\rm L1)}{=} \begin{cases}
1 & n=0,\\
0 & n\neq 0. 
\end{cases}\]
In particular, 
$X_{S_{i,j}}$ is a pretilting object. 
Thus $X_{S}$ is a tilting object by Lemma $\ref{Pair Lemma}$. 

Thanks to the above calculation, the following condition is satisfied:
\begin{gather}
\label{LSH1}\text{If $(i',j')\in S_{i,j}$, then $\dim_{K}\Hom_{\DD}(X_{i,j},X_{i',j'})=1$.}
\end{gather}
By the conditions $(\rm L2)$ and 
$(\ref{LSH1})$, 
thete exists a family $(\upsilon_{i',j'}^{i,j}:X_{i,j}\to X_{i',j'})_{(i,j),(i',j')\in S}$ of morphisms satisfying 
(\ref{LSM1}) and (\ref{LSM2}). 
By (\rm S1$'$) and (\rm S2$'$),  
the conditions (\ref{LSM3}) and (\ref{LSM4}) 
are satisfied. 
\[\xymatrix{
X_{i-1,j-1} & X_{i-1,j} 
\ar[l]_{\upsilon_{i-1,j-1}^{i-1,j}}\\
X_{i,j-1} \ar[u]^{c_{i,j}\upsilon_{i-1,j-1}^{i,j-1}} & X_{i,j} \ar[u]_{c'_{i,j}\upsilon_{i-1,j}^{i,j}}\ar[l]^{\upsilon_{i,j-1}^{i,j}}
\ar[ul]_{\upsilon_{i-1,j-1}^{i,j}}}\]
Thus the assertion follows from Lemma \ref{TS Lemma 1}. 
\end{proof}

\subsection{A property of $S$-families}
Let $(X_{i,j})_{(i,j)\in S}$ be an $S$-family. For any $J\subset \Z$, define 
\[X_{i,J}
=\bigoplus_{j\in J\cap S_{i}} X_{i,j}.\]
In particular 
\[X_{i,>j}=X_{(j,\infty)},\ \ 
X_{i,\geqslant j}=X_{[j,\infty)},\ \ 
X_{i,<j}=X_{(-\infty,j)},\ \ 
X_{i,\leqslant j}=X_{(-\infty,j]}.\]

In this section, we prove the following 
result which is the key step of the proof of 
Theorem \ref{Lpq} in the next section. 
Let $(X_{i,j})_{(i,j)\in S}$ be an $S$-family and 
$(i,j)\in S$. 
If $(i-1,j), (i-1,j-1)\in S$, 
by (\rm S1) and (\rm S3), we have 
$\S_{\langle X_{S}\rangle}(X_{i,j})\simeq X_{i-1,j-1}\simeq\S_{\langle X_{i-1}\rangle}(X_{i-1,j})$. 
The following shows that 
$\S_{\langle X_{S}\rangle}(X_{i,j})
\simeq\S_{\langle X_{i-1}\rangle}(X_{i-1,j})$ 
holds without assuming $(i-1,j-1)\in S$ if 
$S_{i}\subset S_{i-1}$. 

\begin{pro}\label{LSS1}
Let $(X_{i,j})_{(i,j)\in S}$ be an 
$S$-family in $\DD$, and $(i,j)\in S$. 
If $S_{i}\subset S_{i-1}$, then 
$\S_{\langle X_{S}\rangle}(X_{i,j})
\simeq\S_{\langle X_{i-1}\rangle}(X_{i-1,j})
\in\langle X_{i-1}\rangle$.
\end{pro}

To prove Proposition \ref{LSS1}, we prove the following result:
\begin{lem}\label{triangle}
Let $(X_{i,j})_{(i,j)\in S}$ be an $S$-family in $\DD$ 
and $(i,j)$, $(i,j-1)\in S$. Then 
\[\T_{\langle X_{i,\geqslant j}\rangle}
^{\langle X_{i}\rangle}(X_{i,j-1})
=\S_{\langle X_{i,S_{i}
\backslash\{j-1\}}\rangle}(X_{i,j}).\]
\end{lem}

\begin{proof}
Let $Y=\T_{\langle X_{i,\geqslant j}\rangle}
^{\langle X_{i}\rangle}(X_{i,j-1})$. 
Since 
$(X_{i,j})_{j\in S_{i}}$
is regarded as the exceptional sequence,
\[\langle X_{i,\geqslant j}\rangle
=\langle X_{i,j}\rangle\perp\langle 
X_{i,> j}\rangle
\stackrel{\ref{Exc3}(a)}{=}\langle 
X_{i,> j}\rangle
\perp\langle E\rangle\]
where $E=\S_{\langle X_{i,\geqslant j}\rangle}(X_{i,j})$. 
Then 
\[\Hom_{\DD}(X_{i,>j},
Y[n])
\simeq \Hom_{\DD}(X_{i,>j},X_{i,j-1}[n])=0\]
for any integer $n$. We have 
$Y\in\langle E\rangle.$ 
Since $E$ is an exceptional object and 
\[\Hom_{\DD}(Y,E[n])
\simeq\Hom_{\DD}(X_{i,j},
Y[-n])^{\ast}
\simeq\Hom_{\DD}(X_{i,j},X_{i,j-1}[-n])^{\ast}\stackrel{(\rm S1)}{=}
\begin{cases}
K & n=0,\\
0 & \text{otherwise},
\end{cases}\]
we have 
$Y\stackrel{\ref{Exc0}}{=}E.$ 
Since 
\[\langle X_{i,S_{i}
\backslash\{j-1\}}\rangle=
\langle X_{i,<j-1}\rangle
\perp\langle X_{i,\geqslant j}\rangle
\stackrel{(\rm L2)}{=}
\langle X_{i,\geqslant j}\rangle
\perp\langle X_{i,<j-1}\rangle
,\]
we have 
\[\S_{\langle X_{i,S_{i}\backslash\{j-1\}}
\rangle}(X_{i,j})
\simeq \T_{\langle X_{i,S_{i}\backslash\{j-1\}}\rangle}
^{\langle X_{i}\rangle}\S_{\langle X_{i}\rangle}(X_{i,j})
\simeq \T_{\langle X_{i,S_{i}\backslash\{j-1\}}\rangle}
^{\langle X_{i}\rangle}(X_{i,j-1})
\simeq Y\simeq E.\]
Thus the assertion follows.
\end{proof}
Now we are ready to prove Proposition \ref{LSS1}. 
\begin{proof}[Proof of Proposition \ref{LSS1}]
By Theorem \ref{Triviality}, we can assume that 
$\DD=\per\LLL$, $X_{i,j}=P_{\LLL}(i,j)\in\per\LLL$. 
If $(i-1,j-1)\in S$, 
 \[\S_{\langle X_{S}\rangle}(X_{i,j})\simeq X_{i-1,j-1}
\simeq\S_{\langle X_{i-1}\rangle}(X_{i-1,j}).\]
If $(i-1,j-1)\notin S$, let 
$T:=S\cup \{(i-1,j-1)\}$. 
Then $(X_{i,j})_{(i,j)\in T}$ is 
an $T$-family in $\DD$. 
By restricting the lexicographic order 
$\preccurlyeq$ of $\Z^{2}$, we regard $T$ as a totally ordered set (see the figure (\ref{lex})). Let $s:[1,n]\to T$ be the ordered isomorphism, and let 
$E_{k}=X_{s(k)}$. 
Let $p,q$ be two integers such that 
\begin{gather}\label{def p,q}
\text{$p=s^{-1}(i-1,j-1)$, 
$q=s^{-1}(i-1,j_{0})$ where $j_{0}=\sup\{j'\in \Z\mid (i-1,j')\in T\}$,}
\end{gather} 
and let 
$Y=\S_{\langle X_{i-1,\geqslant j-1}\rangle}
(X_{i-1,j-1})=\S_{\langle E_{[p,q]}\rangle}(E_{p})$. \\
If $(i',j')\in S$ and $i'> i$, 
\[\text{$\Hom_{\DD}(X_{i',j'},Y[n])
\stackrel{(\rm L2)}{=}0$ 
for any integer $n$.}\]
If $(i',j')\in S$, $i'=i$, and $j'<j-1$, 
\[\text{$\Hom_{\DD}(X_{i',j'},Y[n])
\stackrel{(\rm L2)}{=}0$ 
for any integer $n$.}\]
If $(i',j')\in S$, $i'=i$, and $j'>j-1$, 
\begin{align*}
\Hom_{\DD}(X_{i',j'},
Y[n])
\stackrel{(\rm S2')}{=}\Hom_{\DD}(X_{i-1,j'},Y[n])
\simeq\Hom_{\DD}(X_{i-1,j-1},
X_{i-1,j'}[-n])^{\ast}
\stackrel{(\rm L2)}{=}0
\end{align*}
for any integer $n$. So we have that 
\[\text{$\Hom_{\DD}(E_{[q+1,n]},\S_{\langle E_{[p,q]}\rangle}(E_{p})[n])=0$ 
for any integer $n$.}\] 
Thus 
\begin{align*}
\S_{\langle X_{S}\rangle}(X_{i,j})
&\stackrel{\ref{Serre functor}(b)}{=} 
\T_{\langle X_{S}\rangle}\S_{\langle X_{T}\rangle}(X_{i,j})
\simeq\T_{\langle X_{S}\rangle}(X_{i-1,j-1})
=\T_{\langle E_{[1,n]\backslash\{p\}}\rangle}(E_{p})\\
&\stackrel{\ref{Exc4}}{=}\T_{\langle E_{[p+1,q]}\rangle}
^{\langle E_{[p,q]}\rangle}(E_{p})
\stackrel{(\ref{def p,q})}{=}\T_{\langle X_{i-1,J\backslash\{j-1\}}
\rangle}
^{\langle X_{i-1,J}
\rangle}(X_{i-1,j-1})
\stackrel{\ref{triangle}}{=}
\S_{\langle X_{i-1}\rangle}(X_{i-1,j})
\end{align*}
where $J=[j-1,j_{0}]\cap S_{i-1}$. 
\end{proof}

\section{$S$-families in derived categories of Nakayama algebras}
\subsection{The algebras 
$\LL(\vp;\vq)$ and $\LL^{!}(\vp;\vq)$}

Let $\vp=(p_{k})_{k\in [1,r]}$ and 
$\vq=(q_{k})_{k\in [1,r]}$ be 
two sequences of positive integers. 
For any integer $s\in[0,r]$, define 
\[\np_{s}:=\sum_{k=1}^{s}p_{k},\ \ 
\nq_{s}:=\sum_{k=1}^{s}q_{k},\ \  
\np:=\np_{r}, \ \ \nq:=\nq_{r}.\] 
Consider the Young diagram 
\[\Y(\vp;\vq):=\bigcup_{k\in[0,r-1]}[1+\np_{k},\np_{k+1}]
\times [1,\nq_{r-k}]
= \bigcup_{k\in[0,r-1]}
[1,\np_{r-k}]\times[1+\nq_{k},\nq_{k+1}]\]
(see Figure 2).
\begin{figure}

\small\begin{tikzpicture}[scale=2/3]
\draw (0,0) 
-- (2,0) 
-- (2,1) 
-- (4,1) 
-- (4,2); 
\draw[dashed] (4,2) 
-- (5,3); 
\draw (5,3) 
-- (7,3);
\draw (7,3) 
-- (7,4);
\draw[dashed] (7,4) 
-- (8,5);
\draw(8,5) 
-- (10,5) 
-- (10,6) 
-- (12,6) 
-- (12,7) 
-- (0,7)
-- (0,0);
\draw (1,0)node[below]{$q_{1}$};
\draw (2,0.5)node[right]{$p_{r}$};
\draw (3,1)node[below]{$q_{2}$};
\draw (4,1.5)node[right]{$p_{r-1}$};

\draw (6,3)node[below]{$q_{r-k+1}$};
\draw (7,3.5)node[right]{$p_{k}$};

\draw (9,5)node[below]{$q_{r-1}$};
\draw (10,5.5)node[right]{$p_{2}$};
\draw (11,6)node[below]{$q_{r}$};
\draw (12,6.5)node[right]{$p_{1}$};
\end{tikzpicture}
\caption{The shape of the Young 
diagram $\Y(\vp;\vq)$}
\end{figure}
Let \[e(\vp;\vq)
:=\sum_{(i,j)\in \Y(\vp;\vq)} 
e_{i}\otimes e_{j},
\ \  e^{!}(\vp;\vq)
:=\sum_{(i,j)\in \Y(\vp;\vq)} e^{!}_{i}\otimes e^{!}_{j}\]
where $e_{i}\otimes e_{j}$ 
(resp.\,$e^{!}_{i}\otimes e^{!}_{j}$) is the primitive idempotent in 
$\NN(\np)\otimes\NN(\nq)$ 
(resp.\,$K\A_{\np}\otimes K\A_{\nq}$) corresponding to the vertex 
$(i,j)\in\Y(\vp;\vq)$. 
Then the algebras $\LL(\vp;\vq)$ and 
$\LL^{!}(\vp;\vq)$ 
are defined as 
\begin{align}
\label{def Lpq}
\LL(\vp;\vq)
&=\LL(p_{1},\dots,p_{r};q_{1},\dots,q_{r})
:=\LL(\Y(\vp;\vq))
=e(\vp;\vq)(\NN(\np)\otimes\NN(\nq))
e(\vp;\vq)\\
&\simeq (\NN(\np)\otimes\NN(\nq))
/\langle 1-e(\vp;\vq)\rangle,\notag\\
\label{def L!pq}\LL^{!}(\vp;\vq)&
=\LL^{!}(p_{1},\dots,p_{r};q_{1},\dots,q_{r})
:=e^{!}(\vp;\vq)(K\A_{\np}\otimes K\A_{\nq})
e^{!}(\vp;\vq)\\
&\simeq (K\A_{\np}\otimes K\A_{\nq})
/\langle 1-e^{!}(\vp;\vq)\rangle.\notag
\end{align}
By the definition of $\LL(\vp;\vq)$ and 
$\LL^{!}(\vp;\vq)$, there exist  
natural isomorphisms $\LL(\vp;\vq)\to\LL(\vq;\vp)$ and $\LL^{!}(\vp;\vq)
\to\LL^{!}(\vq;\vp)$. 

\begin{ex}
The quivers of $\LL(3;4)$ and $\LL^{!}(3;4)$ are the following respectively:
\[\small\xymatrix@=15pt{
(1,1) & (1,2) \ar[l]_-{u_{1,1}} 
& (1,3) \ar[l]_-{u_{1,2}} 
& (1,4)\ar[l]_-{u_{1,3}}\\
(2,1)\ar[u]_-{v_{1,1}} 
& (2,2) \ar[l]_-{u_{2,1}}\ar[u]_-{v_{1,2}} 
& (2,3)\ar[l]_-{u_{2,2}}\ar[u]_-{v_{1,3}} 
& (2,4)\ar[l]_-{u_{2,3}}\ar[u]_-{v_{1,4}}\\
(3,1) \ar[u]_-{v_{2,1}}
& (3,2)\ar[l]_-{u_{3,1}}\ar[u]_-{v_{2,2}} 
& (3,3)\ar[l]_-{u_{3,2}} \ar[u]_-{v_{2,3}}
& (3,4)\ar[l]_-{u_{3,3}}\ar[u]_-{v_{2,4}}}
\ \ \ \ \ \ \ \ \ \ \ 
\small\xymatrix@=15pt{
(1,1)\ar[r]^-{u_{1,1}^{!}}\ar[d]^-{v_{1,1}^{!}} 
& (1,2) \ar[r]^-{u_{1,2}^{!}}\ar[d]^-{v_{1,2}^{!}}
 & (1,3) \ar[r]^-{u_{1,3}^{!}}
 \ar[d]^-{v_{1,3}^{!}} 
 & (1,4)\ar[d]^-{v_{1,4}^{!}}\\
(2,1)\ar[r]^-{u_{2,1}^{!}}\ar[d]^-{v_{2,1}^{!}} 
& (2,2) \ar[r]^-{u_{2,2}^{!}}\ar[d]^-{v_{2,2}^{!}} 
& (2,3)\ar[r]^-{u_{2,3}^{!}}\ar[d]^-{v_{2,3}^{!}} 
& (2,4)\ar[d]^-{v_{2,4}^{!}}\\
(3,1)\ar[r]^-{u_{3,1}^{!}} 
& (3,2) \ar[r]^-{u_{3,2}^{!}} 
& (3,3)\ar[r]^-{u_{3,3}^{!}} 
& (3,4)}\]
The quivers of $\LL(1,2,1;1,2,2)$ and 
$\LL^{!}(1,2,1;1,2,2)$ are the following respectively:
\[\small\xymatrix@=15pt{
(1,1)  & (1,2) \ar[l]_-{u_{1,1}} & (1,3) \ar[l]_-{u_{1,2}}  
& (1,4) \ar[l]_-{u_{1,3}}  & (1,5)\ar[l]_-{u_{1,4}} \\
(2,1)\ar[u]_-{v_{1,1}} & (2,2) \ar[l]_-{u_{2,1}}\ar[u]_-{v_{1,2}} 
& (2,3)\ar[l]_-{u_{2,2}} \ar[u]_-{v_{1,3}} \\
(3,1)\ar[u]_-{v_{2,1}} & (3,2)\ar[u]_-{v_{2,2}} \ar[l]_-{u_{3,1}} & (3,3)\ar[u]_-{v_{2,3}}\ar[l]_-{u_{3,2}}\\
(4,1)\ar[u]_-{v_{3,1}}}
\ \ \ \ \ \ 
\small\xymatrix@=15pt{
(1,1)\ar[r]^-{u_{1,1}^{!}} \ar[d]^-{v_{1,1}^{!}} & (1,2) \ar[r]^-{u_{1,2}^{!}}\ar[d]^-{v_{1,2}^{!}} & (1,3) \ar[r]^-{u_{1,3}^{!}} \ar[d]^-{v_{1,3}^{!}} 
& (1,4) \ar[r]^-{u_{1,4}^{!}}  & (1,5)\\
(2,1)\ar[r]^-{u_{2,1}^{!}}\ar[d]^-{v_{2,1}^{!}} 
& (2,2) \ar[r]^-{u_{2,2}^{!}}\ar[d]^-{v_{2,2}^{!}} 
& (2,3)\ar[d]^-{v_{2,3}^{!}} \\
(3,1)\ar[r]^-{u_{3,1}^{!}}\ar[d]^-{v_{3,1}^{!}} & (3,2) \ar[r]^-{u_{3,2}^{!}} & (3,3) \\
(4,1)
}\]
and the relations is the following respectively: 
\[u_{i,j}u_{i,j+1}=0,\ \ v_{i,j}v_{i+1,j}=0,\ \ u_{i,j}v_{i,j+1}-v_{i,j}u_{i+1,j}=0,\ \ 
 v^{!}_{i,j+1}u^{!}_{i,j}
-u^{!}_{i+1,j}v^{!}_{i,j}=0.\]
\end{ex} 
From the following result, $\LL(\vp;\vq)$
is derived equivalent to $\LL^{!}(\vp;\vq)$. 
\begin{pro}\label{Duality} 
The family $(X_{i,j})_{(i,j)\in \Y(\vp;\vq)}$ 
of objects 
\[X_{i,j}=S_{\LL^{!}(\vp;\vq)}(i,j)[-i-j]\]
is a full $\Y(\vp;\vq)$-family in $\per\LL^{!}(\vp;\vq)$. 
Then the object 
$\displaystyle X_{\Y(\vp;\vq)}=\bigoplus_{(i,j)\in \Y(\vp;\vq)}
X_{i,j}$ 
is a tilting object in 
$\per\LL^{!}(\vp;\vq)$ such that 
$\End_{\per\LL^{!}(\vp;\vq)}(X_{\Y(\vp;\vq)})\simeq 
\LL(\vp;\vq)$, 
and there exists a triangle equivalence 
$F:\per\LL^{!}(\vp;\vq)
\to\per\LL(\vp;\vq)$ 
such that 
$F(X_{i,j})\simeq P(i,j)$ for any 
$(i,j)\in \Y(\vp;\vq)$. 
\end{pro}

\begin{proof}
Let 
$A=\LL^{!}(\np;\nq)=K\A_{\np}\otimes K\A_{\nq}$,  
$B=\NN(\np)\otimes\NN(\nq)$. 
By Example \ref{Sq}, there exists a triangle 
equivalence $G:\per A\to \per B$
such that $G(S_{A}(i,j)[-i-j])
\simeq P_{B}(i,j)$ for any $(i,j)\in\Y(\np,\nq)$. Let 
\[P=e^{!}(\vp;\vq)A, \ \ 
P'=(1-e^{!}(\vp;\vq))A,\ \  
T=\bigoplus_{(i,j)\in \Y(\vp;\vq)}
S_{A}(i,j)[-i-j],  \ \ 
R=G(T)\simeq e(\vp;\vq)B.\] 
Since  
$\End_{\per B}(R)\simeq\LL(\vp;\vq)$, 
there exists a triangle equivalence 
$F_{1}:\langle R\rangle\to \per\LL(\vp;\vq)$ 
such that $F_{1}(P_{B}(i,j))
\simeq P_{\LL(\vp;\vq)}(i,j)$ 
for any $(i,j)\in\Y(\vp;\vq)$ by Example \ref{tilt2}. 

Since $\End_{\per A}(P)\simeq\LL^{!}(\vp;\vq)$ and 
$\Hom_{\per A}(P',P)\simeq 0,$ 
there exists a triangle equivalence 
$F_{2}:\langle P\rangle\to\per \LL^{!}(\vp;\vq)$ such that 
$F_{2}(P_{A}(i,j))\simeq 
P_{\LL^{!}(\vp;\vq)}(i,j)$ and 
$F_{2}(S_{A}(i,j))
\simeq S_{\LL^{!}(\vp;\vq)}(i,j)$ 
for any $(i,j)\in\Y(\vp;\vq)$ by Example \ref{tilt2}. 
Since 
\[F_{1}GF_{2}^{-1}(X_{i,j})
\simeq P_{\LL(\vp;\vq)}(i,j),\]
the triangle equivalence 
$F_{1}GF_{2}^{-1}:\per \LL^{!}(\vp;\vq)\to \per \LL(\vp;\vq)$ 
send the family 
$(X_{i,j})_{(i,j)\in \Y(\vp;\vq)}$ to 
the full $\Y(\vp;\vq)$-family
$(P_{\LL(\vp;\vq)}(i,j))_{(i,j)\in\Y(\vp;\vq)}$ 
in $\per\LL(\vp;\vq)$ given by 
Example \ref{Trivial family}. 
Thus the assertion follows. 
\begin{equation}\label{cd}
\vcenter{
\xymatrix{
\per \LL^{!}(\vp;\vq) 
&\langle P\rangle\ar@{^{(}->}[d]
\ar@{=}[r]\ar[l]_-{F_{2}}^-{\simeq}
&\langle T\rangle 
\ar[r]^-{G}_-{\simeq}
\ar@{^{(}->}[d] 
& \langle R\rangle \ar[r]^-{F_{1}}_-{\simeq}
\ar@{^{(}->}[d] 
& \per \LL(\vp;\vq)\\
& \per A \ar@{=}[r]
&\per A\ar[r]_-{G}^-{\simeq} & \per B. 
& 
}
}
\end{equation}

\vspace*{-1.5\baselineskip}
\end{proof}

Recall that $\DD$ satisfies the condition $(\ref{Good})$ if 
\begin{gather*}
\text{$\DD$ is algebraic, idempotent complete, $\Ext$-finite and has 
a Serre functor $\S$.}
\end{gather*} 
For any two sequences of positive integers 
$\vp=(p_{k})_{k\in [1,r]}$ and 
$\vq=(q_{k})_{k\in [1,r]}$, define 
\[\lambda_{\vp;\vq}(i)=\lambda(i)
:=\sup\{j'\in\Z\mid (i,j')\in \Y(\vp;\vq)\}.\]
When $S=\Y(\vp;\vq)$, 
the definition of $S$-families are 
characterized by the following much simpler 
conditons than $(\rm L1)$, $(\rm L2)$, 
$(\rm S1)$-$(\rm S3)$. 
\begin{thm}\label{Lpq}
Let $\DD$ be a triangulated category satisfying $(\ref{Good})$, 
$(X_{i,j})_{(i,j)\in \Y(\vp;\vq)}$ a 
family of exceptional objects in $\DD$. 
Then 
 $(X_{i,j})_{(i,j)\in \Y(\vp;\vq)}$ is 
a full $\Y(\vp;\vq)$-family in $\DD$ if and 
only if 
 $(X_{i,j})_{(i,j)\in \Y(\vp;\vq)}$ 
satisfies the following conditions: 
\begin{enumerate}[\rm (Y1)]
\item $\DD=\langle X_{1}\rangle\perp\langle X_{2}\rangle\perp\dots\perp\langle X_{\np}\rangle$. 
\item  $\langle X_{1}\rangle
=\langle X_{1,1}\rangle
\perp\langle X_{1,2}\rangle
\perp\dots\perp\langle X_{1,\lambda(1)}\rangle$. 
\item  $\S_{\langle X_{1}\rangle}(X_{1,j})\simeq X_{1,j-1}$ for any integer $j\in(1,\lambda(i)]$. 
\item $\S(X_{i,j})\simeq \S_{\langle X_{i-1}\rangle}(X_{i-1,j})$ for any integers $i\in(1,\np]$ and 
$j\in[1,\lambda(i)]$.
\end{enumerate}
In this case, there exists a triangle equivalence 
\[F:\DD\to\per \LL(\vp;\vq)\]
such that $F(X_{i,j})\simeq P(i,j)$ 
for any $(i,j)\in\Y(\vp;\vq)$.
\end{thm}
To prove Theorem \ref{Lpq}, we prepare the following two results. 
\begin{lem}\label{orth}
Let $\DD$ be a triangulated category satisfying $(\ref{Good})$, 
$(X_{i})_{i\in[1,n]}$ a 
family of objects in $\DD$. 
If the conditons 
\begin{gather}
\numberwithin{equation}{section}
\label{orth1}\text{$\DD=\langle X_{1}\rangle\perp 
\langle X_{2}\rangle\perp\dots \perp\langle X_{n}\rangle$}\\
\label{orth2}\text{$\S(X_{i})\in\langle X_{i-1}\rangle$ 
for any integer $i\in(1,n]$}
\end{gather}
are satisfied, then 
\begin{gather}
\numberwithin{equation}{section}
\text{$\Hom_{\DD}(X_{i},X_{i'}[n])=0$ for any integer $n$ unless $i'\in[i-1,i]$.}
\end{gather}
\end{lem}

\begin{proof}
If $i<i'$, 
$\Hom_{\DD}(X_{i},X_{i'}[n])\stackrel{(\ref{orth1})}{=}0$ 
for any integer $n$. 
If $i-i'\geqslant 2$, then $\S(X_{i})\in \langle X_{i-1}\rangle$ by (\ref{orth2}) and 
$\Hom_{\DD}(X_{i},X_{i'}[n])
\simeq \Hom_{\DD}(X_{i'},\S(X_{i})[-n])^{\ast}
\stackrel{(\ref{orth1})}{=}0$ 
for any integer $n$. 
Thus the assertion follows.
\end{proof}

\begin{lem}\label{ind}
Let $\DD$ be a triangulated category satisfying $(\ref{Good})$, 
$(X_{i,j})_{(i,j)\in \Y(\vp; \vq)}$ a 
family of exceptional objects in $\DD$ 
satisfying $(\rm Y1)$-$(\rm Y4)$. 
Then $(X_{i,j})_{(i,j)\in \Y(\vp;\vq)}$ 
satisfies the following conditions:
\begin{gather}
\numberwithin{equation}{section}
\tag{\rm Y2$'$}\text{ $\langle X_{i}\rangle
=\langle X_{i,1}\rangle
\perp\langle X_{i,2}\rangle
\perp\dots\perp\langle X_{i,\lambda(i)}\rangle$ for any integer $i\in[1,\np]$.}\\
\tag{\rm S1}\text{ $\S_{\langle X_{i}\rangle}(X_{i,j})\simeq X_{i,j-1}$ for any integers $i\in[1,\np]$, $j\in(1,\lambda(i)]$.}
\end{gather}
\end{lem}

\begin{proof}
When $i=1$, (\rm Y2$'$) is nothing but (\rm Y2). Assume that (\rm Y2$'$) holds for $i$. 
If $j<j'$, 
\begin{align*}
\Hom_{\DD}(X_{i+1,j},X_{i+1,j'}[n])
&\stackrel{(\rm Y4)}{=}\Hom_{\DD}
(\S^{-1}\S_{\langle X_{i}\rangle}(X_{i,j}),
\S^{-1}\S_{\langle X_{i}\rangle}(X_{i,j'})[n])\\
&\simeq\Hom_{\DD}(X_{i,j},X_{i,j'}[n])=0
\end{align*}
for any integer $n$. So (\rm Y2$'$) holds for $i+1$. 

When $i=1$, (\rm S1) 
is nothing but (\rm Y3). Assume that (\rm S1) holds for $i$. 
Then $F=\S^{-1}\S_{\langle X_{i}\rangle}
:\langle X_{i,\leqslant\lambda(i+1)}\rangle
\to\langle X_{i+1}\rangle$ is a triangle equivalence by (\rm Y4) and 
\begin{align*} 
\S_{\langle X_{i+1}\rangle}
(X_{i+1,j})
&\stackrel{(\rm Y4)}{=} 
\S_{\langle X_{i+1}\rangle}F(X_{i,j})
\stackrel{\ref{SF}(b)(c)}{=} 
F\S_{\langle X_{i,\leqslant\lambda(i+1)}\rangle}(X_{i,j})\\
&\stackrel{\ref{Serre functor}(c)}{=}
F\S_{\langle X_{i}\rangle}(X_{i,j})
\stackrel{(\rm S1)}{=} F(X_{i,j-1})
\stackrel{(\rm Y4)}{=} X_{i+1,j-1}.
\end{align*}
Thus (\rm S1) holds for $i+1$. 
\begin{equation*}\begin{xymatrix}
{\langle X_{i}\rangle
\ar[d]^{\S_{\langle X_{i}\rangle}} 
& \langle X_{i,\leqslant \lambda(i+1)}\rangle\ar@{_{(}->}[l]\ar[r]^{F}
\ar[d]^{\S_{\langle X_{i,\leqslant\lambda(i+1)}\rangle}} 
& \langle X_{i+1}\rangle
\ar[d]^{\S_{\langle X_{i+1}\rangle}}\\
\langle X_{i}\rangle
\ar[r]_{\T_{\langle X_{i,\leqslant \lambda(i+1)}\rangle}} & \langle X_{i,\leqslant \lambda(i+1)}\rangle\ar[r]_{F} & \langle X_{i+1}\rangle}\end{xymatrix}
\end{equation*}

\vspace*{-2\baselineskip}
\end{proof}

Now we are ready to prove Theorem \ref{Lpq}.
\begin{proof}[Proof of Theorem \ref{Lpq}]
If $(X_{i,j})_{(i,j)\in\Y(\vp;\vq)}$ is a full 
$\Y(\vp;\vq)$-family, 
there exists a triangle equivalence 
$F:\DD\to\per\LL(\vp;\vq)$ by Theorem \ref{Triviality}. 
So we will prove the equivalence of two conditions. 

We prove ``only if'' part. Since 
$(X_{i,j})_{(i,j)\in \Y(\vp;\vq)}$ is a 
full weak $\Y(\vp;\vq)$-family, 
(\rm Y1) and (\rm Y2) are satisfied. 
By (\rm S2), we have that (\rm Y3) is satisfied. 
By Proposition $\ref{LSS1}$, (\rm Y4) is satisfied. Thus the assertion follows. 

We prove ``if'' part. By the assumption, 
(\rm L1) is satisfied. By (\rm Y1) and (\rm Y4), we see that (\rm L2.1) is satisfied by Lemma \ref{orth}. 
Since 
\[\F_{\langle X_{i-1}\rangle}(X_{i,j})
\simeq
\S_{\langle X_{i-1}\rangle}^{-1}\S(X_{i,j})
\stackrel{(\rm Y4)}{=}
\S_{\langle X_{i-1}\rangle}^{-1}\S_{\langle X_{i-1}\rangle}(X_{i-1,j})\simeq X_{i-1,j},\] 
(\rm S2)$'$ is satisfied. 
If $j-j'\geqslant 2$,
\begin{align*}
\Hom_{\DD}(X_{i,j},X_{i-1,j'}[n])
\simeq\Hom_{\DD}(\F_{\langle X_{i-1}\rangle}(X_{i,j}),X_{i-1,j'}[n])
\stackrel{(\rm S2)'}{=}
\Hom(X_{i-1,j},X_{i,j'})
\stackrel{(\rm L2.1)}{=}0.
\end{align*}
So (\rm L2.2) is satisfied. 
Thus $(X_{i,j})_{(i,j)\in \Y(\vp;\vq)}$ is 
a weak $S$-family satisfying (\rm S2)$'$. 

In the last, we prove that $(X_{i,j})_{(i,j)\in \Y(\vp;\vq)}$ is an $S$-family. 
By Lemma \ref{ind}, $(\rm S1)$ is satisfied. By (\rm S1) and (\rm Y4), 
we see that (\rm S3) is satisfied. Thus the assertion follows. 
 \end{proof}

\subsection{$\Y(p,q,r)$-families}
The aim of this section is to prove Theorem 
\ref{Nak}. Let $p$, $q$, and $r$ be three positive integers such that $pq>q+1$, $0\leqslant r\leqslant q$. 
Let 
\[\Y(p,q,r)=\begin{cases}
\Y(p-1,1;q-r,r) & r<q,\\
\Y(p-1;q) & r=q,
\end{cases}\]
\[\LL(p,q,r)=\LL(\Y(p,q,r)).\]

The following application of Theorem \ref{Lpq} is useful to construct a $\Y(p;q)$-family:
\begin{pro}\label{Lad0}
Let $\DD$ be a triangulated category satisfying $(\ref{Good})$, $(E_{k})_{k\in[1,q]}$ an exceptional sequece in $\DD$, and $\displaystyle E=\bigoplus_{k\in[1,q]} E_{k}$. 
Suppose that the following conditions are satisfied:
\begin{gather}
\numberwithin{equation}{section}
\label{E1}\text{$\DD=\langle \S^{p-1}(E)\rangle\perp\langle \S^{p-2}(E)\rangle\perp\dots \perp\langle E\rangle$.}\\
\label{E2}\text{$\S_{\langle E\rangle}(E_{i})\simeq E_{i-1}$ for any integer $i\in(1,p]$.}
\end{gather}
Then the family $(X_{i,j})_{(i,j)\in \Y(p;q)}$ of objects 
\[X_{i,j}:=\F_{\langle \S^{p-i}(E)\rangle}
\F_{\langle \S^{p-i-1}(E)\rangle}\cdots\F_{\langle \S(E)\rangle}(E_{j})\] is a full $\Y(p;q)$-family in $\DD$. 
In particular, there exists a triangle equivalence 
\[F:\DD\to\per \LL(p;q)\]
such that $F(X_{i,j})\simeq P(i,j)$ 
for any $(i,j)\in \Y(p;q)$. 
\end{pro}

\begin{proof}
It suffices to check the conditions 
(\rm Y1)-(\rm Y4). 
Since $\F_{\langle \S^{i+1}(E)\rangle}
:\langle \S^{i}(E)\rangle\to\langle\S^{i+1}(E)\rangle$ is a  
triangle equivalence for each $i\in\Z$ by Lemma 
\ref{App functor}(d), 
$\langle X_{i}\rangle
=\langle \S^{p-i}(E)\rangle$. 
Since $(\ref{E1})$ is satisfied, $(\rm Y1)$ is satisfied. 
Then we see that 
$\S(\langle X_{i}\rangle)=\langle X_{i-1}\rangle.$ 
Since 
\[\S_{\langle X_{i-1}\rangle}^{-1}\S(X_{i,j})
\stackrel{\ref{App functor}(c)}{=} 
\F_{\langle X_{i-1}\rangle}(X_{i,j})
=X_{i-1,j},\]
(\rm Y4) is satisfied. 
Since $\F_{\langle \S^{i+1}(E)\rangle}
:\langle \S^{i}(E)\rangle\to\langle\S^{i+1}(E)\rangle$ is a  
triangle equivalence for each $i\in\Z$ by Lemma 
\ref{App functor}(d), 
\[G=\F_{\langle \S^{p-1}(E)\rangle}
\F_{\langle \S^{p-2}(E)\rangle}\cdots\F_{\langle \S(E)\rangle}:\langle E\rangle\to\langle X_{1}\rangle\]
is a triangle equivalence. 
Since $G$ sends the exceptional sequence 
$(E_{i})_{i\in[1,q]}$ to 
the family $(X_{1,j})_{j\in[1,q]}$, 
it follows that (\rm Y2) is satisfied. 
If $j\in(1,q]$, 
\[\S_{\langle X_{1}\rangle}
(X_{1,j})
\simeq \S_{\langle X_{1}\rangle}G(E_{j})
\simeq 
G\S_{\langle E\rangle}(E_{j})
\stackrel{(\ref{E2})}{=}
G(E_{j-1})\simeq X_{1,j-1}.\]
So (\rm Y3) is satisfied. Thus the assertion 
follows from Theorem \ref{Lpq}. 
\end{proof}

The following result is one of our main results: 
\begin{thm}\label{Nak}
Let $p$, $q$, and $r$ be three positive integers such that $pq>q+1$, $0\leqslant r\leqslant q-1$. Then 
there exists a full $\Y(p,q,r)$-family $(X_{i,j})_{(i,j)\in \Y(p,q,r)}$ in 
$\per \AA(pq-r,q+1)$. In particular, 
there exists a triangle equivalence 
\[F:\per\AA(pq-r,q+1)\to\per\LL(p,q,r)\]
such that $F(X_{i,j})=P_{\LL(p,q,r)}(i,j)$ 
for any $(i,j)\in\Y(p,q,r)$.
\end{thm}

We first prove the case $r=0$ of Theorem 
\ref{Nak} by applying Proposition \ref{Lad0}.  

\begin{pro}\label{Lad1}
Let $p$ and $q$ be two positive integers such that $pq>q+1$, $A=\AA(pq,q+1)$. 
\begin{enumerate}[\rm (a)]
\item 
Let $(X_{i,j})_{(i,j)\in \Y(p;q)}$ be a family of objects $X_{i,j}\in\per A$ defined as 
\[X_{i,j}=\F_{\langle \nak^{p-i}(P)\rangle}
\F_{\langle \nak^{p-i-1}(P)\rangle}\cdots\F_{\langle\nak(P)\rangle}(S(j))[-j]\]
where $\displaystyle P=\bigoplus_{k\in[1,q]} P(k)$. 
Then $(X_{i,j})_{(i,j)\in \Y(p;q)}$ is a full $\Y(p;q)$-family in $\per A$.
\item Let $(X'_{i,j})_{(i,j)\in \Y(p;q)}$ be a family of objects $X'_{i,j}\in\per A$ defined as 
\[X'_{i,j}=\F_{\langle \nak^{p-i}(S)\rangle}
\F_{\langle \nak^{p-i-1}(S)\rangle}\cdots\F_{\langle\nak(S)\rangle}(S((p-1)q+j))
[-j]\]
where 
$\displaystyle S=\bigoplus_{k\in [1,q]} S((p-1)q+k)$. 
Then 
$(X'_{i,j})_{(i,j)\in \Y(p;q)}$ is a 
full $\Y(p;q)$-family in $\per A$. 
\end{enumerate}
\end{pro}

\begin{proof}
(a) It suffices to check the conditions 
(\ref{E1})-(\ref{E2}) by Proposition \ref{Lad0}. 
Let $E_{i}=S(i)[-i]$, $\displaystyle E=\bigoplus_{i\in[1,q]} E_{i}$. 
Since $\langle E\rangle=\langle P\rangle$ and 
$\nak^{k}(P(i))\simeq P(pk+i)$, we have 
\[\DD=\langle \nak^{p-1}(P)\rangle
\perp\langle \nak^{p-2}(P)\rangle
\perp\dots\perp\langle P\rangle
=\langle \nak^{p-1}(E)\rangle
\perp\langle \nak^{p-2}(E)\rangle
\perp\dots\perp\langle E\rangle,\]
and so (\ref{E1}) is satisfied. 

Since $\displaystyle 
E=E_{[1,q]}$ is a pretilting object such that 
$\End_{\DD}(E)\simeq \NN(q)$, there exists a triangle functor $F:\langle E\rangle\to \per\NN(q)$ such that $F(E_{i})\simeq P(i)$, and so (\ref{E2}) is satisfied. 
Thus the assertion follows.\\
(b) Let 
$\displaystyle S'=\bigoplus_{k\in[1,q]} S(k).$ 
Since $\nu^{p-1}(S(k))\simeq S((p-1)q+k)$ for 
any $k\in[1,q]$, we have $\langle \nu^{p-1}(P)\rangle=\langle \nu^{p-1}(S')\rangle=\langle S\rangle$, and so 
 the assertion follows from (a). 
 \end{proof}

Now we are ready to prove Theorem 
\ref{Nak}. 
\begin{proof}[Proof of Theorem \ref{Nak}]
Let $A=\AA(pq,q+1)$, $B=\AA(pq-r,q+1)$ and let 
\[P=\bigoplus_{k\in I}P(k),\ \ 
P'=\bigoplus_{k\in [1,pq]\backslash I}P(k),\ \ S=\bigoplus_{k\in I} S(k),\ \ 
S'=\bigoplus_{k\in [1,pq]\backslash I}S(k)\]
where $I=[1,pq-r]$. 
Since $\Hom_{\per A}(P',P)\simeq 0$ and 
$\End_{\per A}(P)\simeq B$, we have 
$\langle P\rangle
=\langle S\rangle$ and 
there exists a triangle equivalence 
$F:\langle P\rangle\to \per B$ such that 
$F(P_{A}(i))=P_{B}(i)$, $F(S_{A}(i))\simeq S_{B}(i)$ by Example \ref{tilt2}. 
Thus it suffices to show that there exists 
a full $\Y(p,q,r)$-family in $\langle S\rangle$. 

Let $(X'_{i,j})_{(i,j)\in \Y(p;q)}$ be a full 
$\Y(p;q)$-family in $\per N(pq,q+1)$ given by 
Proposition \ref{Lad1}(b). Since 
$Y_{p,q-j}=S(pq-j)[-q+j]$, 
we have $\langle S'\rangle
=\langle X'_{p,[q-r+1,q]}\rangle$, and so 
$\langle X'_{\Y(p,q,r)}\rangle
=\langle S'\rangle^{\perp_{\per A}}=\langle S\rangle$ 
by Example $\ref{tilt2}$.  
Thus the subfamily 
$(X'_{i,j})_{(i,j)\in\Y(p,q,r)}$ is a full $\Y(p,q,r)$-family in $\langle S\rangle$. 
\end{proof}

\section{Mutations of $S$-families}
In this section, let $\DD$ be a triangulated category satisfying 
$(\ref{Good})$. 
The purpose of this section is to 
introduce mutations of $S$-families on 
some assumption for $S$, and to prove 
Theorem \ref{Main 1} and Theorem \ref{Main3} by using those results. 
Let 
\[\fP_{\fin}(\Z^{2})
:=\{S\subset\Z^{2}\mid |S|<\infty\}.\]
In $\fP_{\fin}(\Z^{2})$, the relation $S\equiv S'$ defined as 
\[\text{``there exists an element $v\in\Z^{2}$ 
such that $S'=S+v$''}\]
is an equivalence relation on 
$\fP_{\fin}(\Z^{2})$. 
Clearly, if a family $(X_{i,j})_{(i,j)\in S}$ 
is an $S$-family, then 
for $S'=S+(a,b)$ with $(a,b)\in\Z^{2}$ , 
 the family $(X_{i-a,j-b})_{(i,j)\in S'}$ is an $S'$-family. 
 For any $S\in\fP_{\fin}(\Z^{2})$, 
define 
\[{}^{t}S:=\{(i,j)\in \Z^{2}\mid (j,i)\in S\}.\]
Clearly, a family $\XXX=(X_{i,j})_{(i,j)\in S}$ 
is an $S$-family if and only if the family 
$^{t}\XXX=(X_{i,j})_{(i,j)\in {}^{t}S}$ is an 
${}^{t}S$-family. 

\subsection{Gluings of $S$-families}
In this section, we introduce gluings of 
$S$-families (Proposition \ref{Glu1}, 
Proposition \ref{Glu2}). 
For any interval $I$ of $\Z$, 
\[S_{I}:=\{(i,j)\in S\mid i\in I\},\ \ 
S^{I}:=\{(i,j)\in S\mid j\in I\}.\]
In particular, 
\[S_{\leqslant k}:=S_{(\infty,k]},\ \ 
S_{\geqslant k}:=S_{[k,\infty)}, \ \ 
S^{\leqslant k}:=S^{(\infty,k]},\ \ 
S^{\geqslant k}:=S^{[k,\infty)}.\]
Let $\XXX=(X_{i,j})$ be a family of objects 
in $\DD$ indexed by 
$S\in \fP_{\fin}(\Z^{2})$. For any integer $k$, define 
\[\XXX_{\leqslant k}=(X_{i,j})_{(i,j)\in S_{\leqslant k}},
\ \ \XXX_{\geqslant k}=(X_{i,j})_{(i,j)\in S_{\geqslant k}},\] 
\[\XXX^{\leqslant k}=(X_{i,j})_{(i,j)\in S^{\leqslant k}},
\ \ \XXX^{\geqslant k}=(X_{i,j})_{(i,j)\in S^{\geqslant k}}.\]

\begin{lem}\label{Glu}
Let $\XXX=(X_{i,j})_{(i,j)\in S}$ be a family 
of objects in $\DD$ indexed by $S\in \fP_{\fin}(\Z^{2})$ satisfying 
$(\rm L2.1)$. 
Let $k$ be an integer.
\begin{enumerate}[\rm (a)]
\item \text{If $(i,j)\in S_{\geqslant k+1}$, then $\S_{\langle X_{S}\rangle}(X_{i,j})
\simeq \S_{\langle X_{S_{\geqslant k}}\rangle}(X_{i,j})$ and $\S_{\langle X^{j}\rangle}(X_{i,j})
\simeq \S_{\langle X_{\geqslant k,j}\rangle}(X_{i,j}).$}
\item \text{If $(i,j)\in S_{\leqslant k-1}$, then $\S^{-1}_{\langle X_{S}\rangle}(X_{i,j})
\simeq \S^{-1}_{\langle X_{S_{\leqslant k}}\rangle}(X_{i,j})$ and $\S^{-1}_{\langle X^{j}\rangle}(X_{i,j})
\simeq \S^{-1}_{\langle X_{\leqslant k,j}\rangle}(X_{i,j}).$}
\end{enumerate}
\end{lem}

\begin{proof}
(a) From (\rm L2.1), we have 
$\langle X_{S}\rangle=\langle X_{S_{\leqslant k-1}}\rangle\perp\langle X_{S_{\geqslant k}}\rangle.$ 
Since 
\[\Hom_{\DD}(X_{S_{\leqslant k-1}},
\S_{\langle X_{S}\rangle}(X_{i,j})[n])
\simeq \Hom_{\DD}(X_{i,j},X_{S_{\leqslant k-1}}[-n])^{\ast}
\stackrel{(\rm L2.1)}{=} 0,\] we have 
$\S_{\langle X_{S}\rangle}(X_{i,j})\in \langle X_{S_{\geqslant k}}\rangle$. Thus the assertion 
follows from Proposition 
$\ref{Serre functor}$(c). 
Since the subfamily $(X_{i,j})_{(i,j)\in 
S^{j}\times\{j\}}$ satisfies 
$(\rm L2.1)$, the assertion follows.\\
(b) This is the dual of (a).\end{proof}

\begin{pro}[Gluing \rm I]\label{Glu1}
Let $\XXX=(X_{i,j})_{(i,j)\in S}$ be a family 
of objects in $\DD$ indexed by $S\in \fP_{\fin}(\Z^{2})$ satisfying 
$(\rm L2.1)$. 
For any integer $k$, the following conditions 
are equivalent: 
\begin{enumerate}[\rm (i)]
\item $\XXX$ is an $S$-family. 
\item $\XXX_{\leqslant k}$ is 
an $S_{\leqslant k}$-family and 
$\XXX_{\geqslant k}$ is 
an $S_{\geqslant k}$-family. 
\end{enumerate}
\end{pro}

\begin{proof}
(i)$\Rightarrow$ (ii): This is clear. \\
(ii)$\Rightarrow$ (i): 
We need to show that $\XXX$ satisfies 
the conditions (\rm L1), (\rm L2.2) and 
(\rm S1)-(\rm S3). 
Clearly (\rm L1) is satisfied. 
We show that (\rm L2.2) is satisfied. 
Let $(i,j)$ and $(i',j')$ be two elements in $S$ such that $|j-j'|>1$. 
If $(i,j)\in S_{\leqslant k}\backslash S_{\{k\}}$ and $(i',j')\in S_{\geqslant k}
\backslash S_{\{k\}}$, we have 
$\Hom_{\DD}(X_{i,j},X_{i',j'}[n])\simeq 0 
\ \text{and}\ 
\Hom_{\DD}(X_{i',j'},X_{i,j}[n])\simeq 0$ 
from (\rm L2.1). 
If $(i,j),(i',j')\in S_{\leqslant k}$ or $(i,j),(i',j')\in S_{\geqslant k}$, then 
$\Hom_{\DD}(X_{i,j},X_{i',j'}[n])\simeq 0$ since 
$\XXX_{\leqslant k}$ is 
an $S_{\leqslant k}$-family and 
$\XXX_{\geqslant k}$ is 
an $S_{\geqslant k}$-family. Thus  
$(\rm L2.2)$ is satisfied, 
and so $\XXX$ is a weak $S$-family. 

Since $\XXX_{\leqslant k}$ is 
an $S_{\leqslant k}$-family and 
$\XXX_{\geqslant k}$ is 
an $S_{\geqslant k}$-family, 
$(\rm S1)$ is satisfied.\\ 
Suppose that $(i,j),(i-1,j)\in S$. 
If $(i,j)\in S_{\geqslant k+1}$, then 
$\S_{\langle X^{j}\rangle}(X_{i,j})
\stackrel{\ref{Glu}(a)}{=}
\S_{\langle X_{\geqslant k,j}\rangle}(X_{i,j})\stackrel{(\rm S2)}{=}X_{i-1,j}$. 
If $(i,j)\in S_{\leqslant k}$, then 
$(i-1,j)\in S_{\leqslant k-1}$ and 
$\S^{-1}_{\langle X^{j}\rangle}(X_{i-1,j})
\stackrel{\ref{Glu}(b)}{=}\S^{-1}_{\langle X_{\leqslant k,j}\rangle}(X_{i-1,j})\stackrel{(\rm S2)}{=}X_{i,j}$. 
Thus $(\rm S2)$ is satisfied. \\
Suppose that $(i,j),(i-1,j-1)\in S$. 
If $(i,j)\in S_{\geqslant k+1}$, then 
$(i-1,j-1)\in S_{\geqslant k}$ and 
$\S_{\langle X_{S}\rangle}(X_{i,j})
\stackrel{\ref{Glu}(a)}{=}\S_{\langle X_{S_{\geqslant k}}\rangle}(X_{i,j})\stackrel{(\rm S3)}{=}X_{i-1,j-1}$. 
If $(i,j)\in S_{\leqslant k}$, then 
$(i-1,j-1)\in S_{\leqslant k-1}$ and 
$\S^{-1}_{\langle X_{S}\rangle}(X_{i-1,j-1})
\stackrel{\ref{Glu}(b)}{=}\S^{-1}_{\langle X_{S_{\leqslant k}}\rangle}(X_{i-1,j-1})\stackrel{(\rm S3)}{=}X_{i,j}$. 
Thus $(\rm S3)$ is satisfied. 
\end{proof}
By transposing, 
we have the following result. 
\begin{pro}[Gluing \rm I\hspace{-.01em}I]\label{Glu2}
Let $\XXX=(X_{i,j})_{(i,j)\in S}$ be a family 
of objects in $\DD$ indexed by $S\in \fP_{\fin}(\Z^{2})$ satisfying 
$(\rm L2.2)$. 
For any integer $k$, the following conditions 
are equivalent: 
\begin{enumerate}[\rm (i)]
\item $\XXX$ is an $S$-family. 
\item $\XXX^{\leqslant k}$ is 
an $S^{\leqslant k}$-family and 
$\XXX^{\geqslant k}$ is 
an $S^{\geqslant k}$-family. 
\end{enumerate}\end{pro}

\subsection{Mutations of $S$-families}
In this subsection, we prove results for 
mutations of $S$-families (Theorem \ref{Mutation Lemma 1}, 
Theorem \ref{Mutation Lemma 2}). 
Let $S\in\fP_{\fin}(\Z^{2})$. 
Recall that 
\[S_{k}:=\{j\in\Z\mid (k,j)\in S\},\ \ 
S^{k}:=\{i\in\Z\mid (i,k)\in S\}\]
for any integer $k$. 
Let 
$\sigma_{\leqslant k},\sigma_{\geqslant k},\rho_{\leqslant k},\rho_{\geqslant k}$ be 
permutaions of $\Z^{2}$ such that 
\[\sigma_{\leqslant k}(i,j)=\begin{cases}
(i,j-1) & i\leqslant k,\\
(i,j) & k< i, 
\end{cases}\ \ \sigma_{\geqslant k}(i,j)=\begin{cases}
(i,j-1) & i\geqslant k,\\
(i,j) & k>i, 
\end{cases}\]
\[\rho_{\leqslant k}(i,j)=\begin{cases}
(i-1,j) & j\leqslant k,\\
(i,j) & k< j, 
\end{cases} \ \ 
\rho_{\geqslant k}(i,j)=\begin{cases}
(i-1,j) & j\geqslant k,\\
(i,j) & k>j. 
\end{cases}\]
For any subset $I$ of $\Z$ and any integer  $n\in\Z$, we denote by $I+n$ 
the subset \[I+n=\{j\in \Z\mid 
\exists i\in I; j=i+n\}.\] 

\begin{dfn}
Let $k$ be a nonnegative integer. 
A finite subset $S$ of $\Z^{2}$ 
is called an \emph{$M^{+}_k$-subset} if the following 
conditions are satisfied: 
\begin{enumerate}[\rm (M$^{+}_{k}$1)]
\item\label{M1} If $i\in[0,k+1]$, 
then $S_{i}$ is an interval of $\Z$.
\item\label{M2} If $i\in[1,k]$, 
then $S_{i}=S_{i-1}$ or $S_{i}
=S_{i-1}+1$.
\item\label{M3} $S_{k+1}\subset S_{k}$.
\end{enumerate}
A finite subset $S'$ of $\Z^{2}$ 
is called an \emph{$M^{-}_{k}$-subset} if the following 
conditions are satisfied: 
\begin{enumerate}[\rm (M$^{-}_{k}$1)]
\item \label{M1'} If $i\in[0,k+1]$, 
then $S'_{i}$ is an interval of $\Z$.
\item\label{M2'}If $i\in[1,k]$, 
then $S'_{i}=S'_{i-1}$ or $S'_{i}
=S'_{i-1}+1$.
\item\label{M3'}$S'_{k+1}\subset S'_{k}+1$.
\end{enumerate}
For any $(i,j)\in S$, if $S-(i,j)$ is an 
$M^{+}_{k}$-subset $($resp. $M^{-}_{k}$-subset
$)$, $S$ is called an \emph{$M_{k}^{+}(i,j)$-subset} 
$($resp. \emph{$M^{-}_{k}(i,j)$-subset}$)$. 
\end{dfn} 
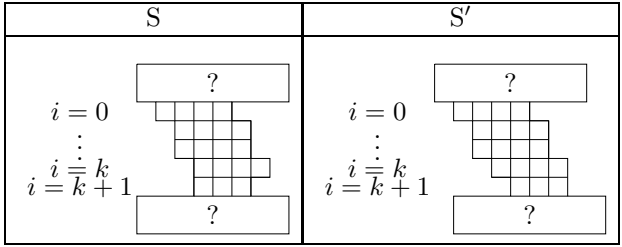
\begin{tabular}{|c|c|}\hline
S & S$'$ \\ \hline 
\begin{tikzpicture}
\draw[ultra thin] (0,1/4) node {};

\draw[ultra thin] (0,0)--(0,-2/4)--(1/4,-2/4)
--(1/4,-3/4)--(2/4,-3/4)--(2/4,-5/4)
--(3/4,-5/4)--(3/4,-7/4)--(0,-7/4)
--(0,-9/4)--(8/4,-9/4)--(8/4,-7/4)
--(6/4,-7/4)--(6/4,-6/4)--(7/4,-6/4)
--(7/4,-5/4)--(6/4,-5/4)
--(6/4,-3/4)--(5/4,-3/4)--(5/4,-2/4)
--(8/4,-2/4)--(8/4,0)--(0,0);

\draw[ultra thin] (1/4,-2/4)--(5/4,-2/4);
\draw[ultra thin] (2/4,-3/4)--(6/4,-3/4);
\draw[ultra thin] (2/4,-4/4)--(6/4,-4/4);
\draw[ultra thin] (2/4,-5/4)--(6/4,-5/4);
\draw[ultra thin] (3/4,-6/4)--(7/4,-6/4);
\draw[ultra thin] (3/4,-7/4)--(6/4,-7/4);

\draw[ultra thin] (2/4,-2/4)--(2/4,-3/4);
\draw[ultra thin] (3/4,-2/4)--(3/4,-7/4);
\draw[ultra thin] (4/4,-2/4)--(4/4,-7/4);
\draw[ultra thin] (5/4,-2/4)--(5/4,-7/4);
\draw[ultra thin] (6/4,-3/4)--(6/4,-7/4);

\draw[ultra thin] (4/4,-1/4) node {?};
\draw[ultra thin] (4/4,-8/4) node {?};

\draw[ultra thin] (-3/4,-5/8) node {$i=0$};
\draw[ultra thin] (-3/4,-8/8) node {$\vdots$};
\draw[ultra thin] (-3/4,-11/8) node {$i=k$};
\draw[ultra thin] (-3/4,-13/8) node {$i=k+1$};
\end{tikzpicture}
&
\begin{tikzpicture}
\draw[ultra thin] (0,0)--(0,-2/4)
--(1/4,-2/4)--(1/4,-3/4)
--(2/4,-3/4)--(2/4,-5/4)
--(3/4,-5/4)--(3/4,-6/4)
--(4/4,-6/4)--(4/4,-7/4)
--(1/4,-7/4)--(1/4,-9/4)
--(9/4,-9/4)--(9/4,-7/4)
--(7/4,-7/4)--(7/4,-5/4)--(6/4,-5/4)
--(6/4,-3/4)--(5/4,-3/4)--(5/4,-2/4)
--(8/4,-2/4)--(8/4,0)--(0,0);

\draw[ultra thin] (1/4,-2/4)--(5/4,-2/4);
\draw[ultra thin] (2/4,-3/4)--(6/4,-3/4);
\draw[ultra thin] (2/4,-4/4)--(6/4,-4/4);
\draw[ultra thin] (2/4,-5/4)--(6/4,-5/4);
\draw[ultra thin] (3/4,-6/4)--(7/4,-6/4);
\draw[ultra thin] (4/4,-7/4)--(7/4,-7/4);

\draw[ultra thin] (2/4,-2/4)--(2/4,-3/4);
\draw[ultra thin] (3/4,-2/4)--(3/4,-5/4);
\draw[ultra thin] (4/4,-2/4)--(4/4,-6/4);
\draw[ultra thin] (5/4,-2/4)--(5/4,-7/4);
\draw[ultra thin] (6/4,-3/4)--(6/4,-7/4);

\draw[ultra thin] (4/4,-1/4) node {?};
\draw[ultra thin] (5/4,-8/4) node {?};

\draw[ultra thin] (-3/4,-5/8) node {$i=0$};
\draw[ultra thin] (-3/4,-8/8) node {$\vdots$};
\draw[ultra thin] (-3/4,-11/8) node {$i=k$};
\draw[ultra thin] (-3/4,-13/8) node {$i=k+1$};\end{tikzpicture}\\ \hline 
\end{tabular}\\ 

By definitions, the following 
result is clear:  
\begin{lem}
Let $k$ be a nonnegative integer, 
$S$ and $S'$ two finite subsets of $\Z^{2}$ such that $S'=\sigma_{\leqslant k}(S)$. 
Then $S$ is an $M^{+}_k$-subset 
if and only if $S'$ is an $M^{-}_k$-subset. \end{lem}

The following result is the one of 
main results in this subsection:
\begin{thm}[Mutation \rm I]
\label{Mutation Lemma 1}
Let $\DD$ be a triangulated category satisfying 
$(\ref{Good})$, 
$S$ an $M^{+}_k$-subset satisfying 
$S_{\leqslant -1}=\emptyset$, and let $S'=\sigma_{\leqslant k}(S)$. 

\begin{enumerate}[\rm (a)]
\item If $\XXX=(X_{i,j})_{(i,j)\in S}$ is an $S$-family, then 
the family $\XXX'=(X'_{i,j})_{(i,j)\in S'}$ 
 is an $S'$-family where 
 \begin{equation*}
X'_{i,j}=\begin{cases}
\S_{\langle X_{i}\rangle}(X_{i,j+1}) 
& i\leqslant k,\\
X_{i,j} & k< i. 
\end{cases}
\end{equation*}
In particular, there exists a triangle equivalence 
$\per \LL(S)\to \per \LL(S')$.

\item If $\YYY=(Y_{i,j})_{(i,j)\in S'}$ is an 
$S'$-family, 
then the family 
$\YYY'=(Y'_{i,j})_{(i,j)\in S}$ is an 
 $S$-family where 
 \begin{equation*}
Y'_{i,j}=\begin{cases}
\S_{\langle Y_{i}\rangle}^{-1}(Y_{i,j-1}) 
& i\leqslant k,\\
Y_{i,j} & k< i.
\end{cases}
\end{equation*}
In particular, there exists a triangle equivalence 
$\per \LL(S')\to \per \LL(S)$.
\end{enumerate}
\end{thm}
\begin{tabular}{|c|c|}\hline
S & S$'$ \\ \hline 
\begin{tikzpicture}
\draw (0,1/4) node {};
\draw (-1,-1/8) node {$i=0$};
\draw (-1,-1/8-7/16) node {$\vdots$};
\draw (-1,-1/8-4/4) node {$i=k$};
\draw (-1,-1/8-5/4) node {$i=k+1$};

\foreach \z in {0,1,2,3} 
\draw[ultra thin] (\z/4,0)--(\z/4,-6/4); 
\foreach \z in {4} 
\draw[ultra thin] (\z/4,0)--(\z/4,-5/4);

\foreach \w in 
{0,1,2,3,4,5}
\draw[ultra thin] (0,-\w/4)--(4/4,-\w/4); 
\foreach \w in 
{6}
\draw[ultra thin] (0,-\w/4)--(3/4,-\w/4); 

\foreach \z in {-1,5} 
\draw[ultra thin] (\z/4,-6/4)--(\z/4,-8/4); 
\foreach \w in {6,8}
\draw[ultra thin] (-1/4,-\w/4)--(5/4,-\w/4); 

\draw (4/8,-7/4) node {?};
\end{tikzpicture}
&
\begin{tikzpicture}
\draw (0,1/4) node {};
\draw (-1,-1/8) node {$i=0$};
\draw (-1,-1/8-7/16) node {$\vdots$};
\draw (-1,-1/8-4/4) node {$i=k$};
\draw (-1,-1/8-5/4) node {$i=k+1$};

\foreach \z in {0} 
\draw[ultra thin] (\z/4,0)--(\z/4,-5/4);
\foreach \z in {1,2,3,4} 
\draw[ultra thin] (\z/4,0)--(\z/4,-6/4); 

\foreach \w in 
{0,1,2,3,4}
\draw[ultra thin] (0,-\w/4)--(4/4,-\w/4); 

\draw[ultra thin] (0,-5/4)--(4/4,-5/4); 
\draw[ultra thin] (1/4,-6/4)--(4/4,-6/4); 

\foreach \z in {0,6} 
\draw[ultra thin] (\z/4,-6/4)--(\z/4,-8/4); 
\foreach \w in {6,8}
\draw[ultra thin] (0,-\w/4)--(6/4,-\w/4); 

\draw (6/8,-7/4) node {?};
\end{tikzpicture}\\ \hline 
\end{tabular}\\ 

In Theorem \ref{Mutation Lemma 1} (a), 
by the definitions of $\XXX$ and $\XXX'$, 
there exists a 
sequence of mutations of 
exceptional sequences 
from $\XXX$ to $\XXX'$. 
\begin{rem}
Let $I_{-1}=\{2\}$, $I_{0}=[1,2]$, 
$I_{1}=\{2\}$ and 
\[S=\bigsqcup_{k\in[-1,1]} \{k\}\times I_{k},\ \ S'=\sigma_{\leqslant 0}(S).\]
Then $S$ is an $M^{+}_{0}$-subset such that 
$S_{\leqslant -1}\neq \emptyset$. 
Since there exist triangle equivalences
\[\text{$\per \LL(S)\to \per K\mathbb{A}_{4}$ 
and 
$\per\LL(S')\to \per K\mathbb{D}_{4}$,}\]
we have that $\per \LL(S)$ and $\per \LL(S')$ are not triangle equivalent to each other. 
\end{rem}

To prove Theorem \ref{Mutation Lemma 1}, we prepare the following three results. 
\begin{lem}\label{Eq}
Let $\DD$ be a triangulated category satisfying 
$(\ref{Good})$, 
$(X_{i,j})_{(i,j)\in \Y(p;q)}$ 
a $\Y(p;q)$-family in $\DD$. 
Then there exists a triangle autoequivalence 
$G:\langle X_{\Y(p;q)}\rangle
\to\langle X_{\Y(p;q)}\rangle$ 
such that 
\begin{gather}
\numberwithin{equation}{section}
\label{Eq1}\text{$G|_{\langle X_{i}\rangle}\simeq 
\S_{\langle X_{i}\rangle}:\langle X_{i}\rangle
\to\langle X_{i}\rangle$ 
for any $i\in[1,p]$.}
\end{gather}
\end{lem}

\begin{proof}
Let $A=\NN(p)\otimes\NN(q)$. 
By Theorem \ref{Triviality}, 
there exists a triangle equivalence 
$F:\langle X_{\Y(p;q)}\rangle
\to \per A$ 
such that $F(X_{i,j})\simeq P(i,j)$ for any 
$(i,j)\in\Y(p;q)$. Then 
$F(X_{i})\simeq P(i)\otimes \NN(q)$ for any 
$i\in[1,p]$. 
By Lemma \ref{Id and Serre}, 
\[G'=(-)\Lotimes_{A}
(\NN(p)\otimes\NN(q)^{\ast}):\per A\to\per A\]
is a triangle autoequivalence such that 
\[G'|_{\langle P(i)\otimes\NN(q)\rangle}:
\langle P(i)\otimes\NN(q)\rangle\to 
\langle P(i)\otimes\NN(q)\rangle\] 
is a Serre functor for any $i\in[1,p]$. 
Thus $G=F^{-1}G'F$ satisfies (\ref{Eq1}). 
\end{proof}

By the following result, 
if $S$ is an $M^{+}_k$ subset 
satisfying $S_{\leqslant -1}=\emptyset$ and $S_{\geqslant k+1}=\emptyset$, 
any $S$-family is a mutation of a 
$\Y(k+1;h)$-family. 

\begin{lem}\label{Mutation Lemma 00}
Let $\DD$ be a triangulated category satisfying 
$(\ref{Good})$, 
$S$ an $M^{+}_k$-subset satisfying 
$S_{\geqslant k+1}=\emptyset$. 
Let $(X_{i,j})_{(i,j)\in S}$ be an $S$-family in $\DD$ and $h=|S_{0}|$. 
\begin{enumerate}[\rm (a)]
\item There exists a $\Y(k+1;h)$-family 
$(Y_{i,j})_{(i,j)\in\Y(k+1;h)}$ such that 
$\langle X_{i}\rangle=\langle Y_{i}\rangle$ for 
any $i\in[0,k]$, and there 
exists a triangle equivalence 
$\langle X_{S}\rangle\to\per \LL(k+1;h)$ such that $F(Y_{i,j})\simeq P(i,j)$.
\item The family $(X'_{i,j})_{(i,j)\in S}$ of objects 
$X'_{i,j}:=\S_{\langle X_{i}\rangle}(X_{i,j})$ 
is an $S$-family such that 
$\langle X_{i}\rangle=\langle X'_{i}\rangle$ 
for any $i\in[0,k]$. 
\item The family $(X''_{i,j})_{(i,j)\in S}$ of objects 
$X''_{i,j}:=\S^{-1}_{\langle X_{i}\rangle}(X_{i,j})$ 
is an $S$-family such that 
$\langle X_{i}\rangle=\langle X''_{i}\rangle$ 
for any $i\in[0,k]$.\end{enumerate}
\end{lem}

\begin{proof}
(a) We show that $(X_{k,j})_{j\in[0,h]}$ is 
an exceptional sequence satisfying 
(\ref{E1})-(\ref{E2}) in Proposition 
\ref{Lad0}. 
From (\rm L2), 
$(X_{k,j})_{j\in[0,h]}$ is an exceptional sequence.  From (\rm S1), we have that (\ref{E2}) is satisfied. 
From (\rm L2), 
\[\langle X_{S}\rangle=\langle X_{0}\rangle\perp\langle X_{2}\rangle\perp\dots\perp\langle X_{k}\rangle.\]
Let $i\in[1,k]$. 
If $S_{i}=S_{i-1}$, then 
$\S(X_{i,j})\stackrel{\ref{LSS1}}{=}
\S_{\langle X_{i-1}\rangle}(X_{i-1,j})$, 
and so we have 
$\langle\S(X_{i})\rangle=
\langle X_{i-1}\rangle$. 
If $S_{i}=S_{i-1}+1$, then 
$\S(X_{i,j})\stackrel{(\rm S3)}{=}X_{i-1,j-1}$, 
and so we have 
$\langle\S(X_{i})\rangle=
\langle X_{i-1}\rangle$. 
So (\ref{E1}) is satisfied. 
Thus the assertion follows from Proposition \ref{Lad0}.\\ 
(b) By Lemma \ref{Eq}, there exists a triangle autoequivalence 
$G:\langle X_{S}\rangle\to\langle X_{S}\rangle$ 
such that 
$G(X_{i,j})
=\S_{\langle X_{i}\rangle}(X_{i,j})$. 
Thus the assertion follows. \\
(c) This is the dual of (b). 
\end{proof}

The following result is equivalent to 
Theorem \ref{Mutation Lemma 1} in the case $k=0$ and $S_{>k+1}=\emptyset$.
\begin{lem}
\label{Mutation Lemma 0}
Let $\DD$ be a triangulated category satisfying 
$(\ref{Good})$, 
and let $I_{0}$ and $I_{1}$ be two 
intervals of $\Z$ such that $I_{1}\subset I_{0}$, 
$S$ and $S'$ two finite subsets of $\Z^{2}$ such that 
\[S=\bigsqcup_{k\in[0,1]}
\{k\}\times I_{k},\ \ 
S'=\sigma_{\leqslant 0}(S).\] 
\begin{enumerate}[\rm (a)]
\item If $\XXX=(X_{i,j})_{(i,j)\in S}$ is an $S$-family, then 
the family $\XXX'=(X'_{i,j})_{(i,j)\in S'}$
is an $S'$-family where 
\begin{equation*}
X'_{i,j}=\begin{cases}
\S_{\langle X_{0}\rangle}(X_{0,j+1}) 
& i=0,\\
X_{1,j} & i=1. 
\end{cases}
\end{equation*}
In particular, there exists a triangle equivalence 
$\per \LL(S)\to \per \LL(S')$.
  
\item If $\YYY=(Y_{i,j})_{(i,j)\in S'}$ 
is an $S'$-family, 
then the family 
$\YYY'=(Y'_{i,j})_{(i,j)\in S}$ 
is an $S$-family where 
\begin{equation*}
Y'_{i,j}=\begin{cases}
\S_{\langle Y_{0}\rangle}^{-1}(Y_{0,j-1}) 
& i=0,\\
Y_{1,j} & i=1.
\end{cases}
\end{equation*}
In particular, there exists a triangle equivalence 
$\per \LL(S')\to \per \LL(S)$.
\end{enumerate}
\end{lem}

\begin{proof}
Without loss of generality, 
we can assume that $I_{0}=[p_{1},p_{2}]$ and 
$I_{1}=[1,q]$. 

(a) Since $\langle X'_{i}\rangle=
\langle \S_{\langle X_{i}\rangle}(X_{i})\rangle=\langle X_{i}\rangle$ for any $i\in[0,1]$, we have that 
$\XXX'$ satisfies (\rm L2.1). 
Let $j$ and $j'$ be two integers such that 
$|j-j'|>1$. Then 
\begin{align*}
\Hom_{\DD}(X'_{0,j},X'_{0,j'})
&=
\Hom_{\DD}(\S_{\langle X_{0}\rangle}(X_{0,j+1}),
\S_{\langle X_{0}\rangle}(X_{0,j'+1}))\\
&\simeq \Hom_{\DD}(X_{0,j+1},X_{0,j'+1})
\stackrel{\text{(\rm L2) for $\XXX$}}{=} 0,
\end{align*}
\begin{align*}
\Hom_{\DD}(X'_{1,j},X'_{1,j'})=
\Hom_{\DD}(X_{1,j},X_{1,j'})
\stackrel{\text{(\rm L2) for $\XXX$}}{=} 0. 
\end{align*}
If $j'\neq p_{1}-1$, 
\begin{align*}
\Hom_{\DD}(X'_{1,j},X'_{0,j'})
&=\Hom_{\DD}(X_{1,j},
\S_{\langle X_{0}\rangle}(X_{0,j'+1}))
\stackrel{\text{(\rm S1) for $\XXX$}}{=}\Hom_{\DD}(X_{1,j},X_{0,j'})\\
&\stackrel{\text{(\rm S2)$'$ for $\XXX$}}{=} \Hom_{\DD}(X_{0,j},
X_{0,j'})
\stackrel{\text{(\rm L2) for $\XXX$}}{=} 0.
\end{align*}
If $j'=p_{1}-1$, 
\begin{align*}
\Hom_{\DD}(X'_{1,j},X'_{0,p_{1}-1})
&=\Hom_{\DD}(X_{1,j},
\S_{\langle X_{0}\rangle}(X_{0,p_{1}}))
\stackrel{\text{(\rm S2)$'$ for $\XXX$}}{=} \Hom_{\DD}(X_{0,j},
\S_{\langle X_{0}\rangle}(X_{0,p_{1}}))\\
&\simeq \Hom_{\DD}(X_{0,p_{1}}, X_{0,j})^{\ast}
\stackrel{\text{(\rm L2) for $\XXX$}}{=} 0.
\end{align*}
Thus $\XXX'$ satisfies (\rm L2.2), and so 
$\XXX'$ is a weak $S'$-family. 

Clearly $\XXX'$ satisfies (\rm S1). 
By Proposition \ref{LSS1}, $\XXX'$ satisfies (\rm S3). 
If $(0,j),(1,j)\in S'$, then $(0,j+1)\in S$ and 
\[\F_{\langle X'_{0}\rangle}(X'_{1,j})
=\F_{\langle X_{0}\rangle}(X_{1,j})
\stackrel{\text{(\rm S2)$'$ for $\XXX$}}{=}X_{0,j}
\stackrel{\text{(\rm S1) for $\XXX$}}{=}
\S_{\langle X_{0}\rangle}(X_{0,j+1})=X'_{0,j}.\]
So $\XXX'$ satisfies (\rm S2)$'$. 
Thus the assertion follows. 

(b) Since $\langle Y'_{i}\rangle=
\langle \S_{\langle Y_{i}\rangle}(Y_{i})\rangle=\langle Y_{i}\rangle$ for any $i\in[0,1]$, we have that 
$\YYY$ satisfies (\rm L2.1). 
Let $j$ and $j'$ be two integers such that 
$|j-j'|>1$. Then 
\begin{align*}
\Hom_{\DD}(Y'_{0,j},Y'_{0,j'})
&= \Hom_{\DD}(\S^{-1}_{\langle Y_{0}\rangle}(Y_{0,j-1}),
\S^{-1}_{\langle Y_{0}\rangle}(Y_{0,j'-1}))\\
&\simeq \Hom_{\DD}(Y_{0,j-1},Y_{0,j'-1})
\stackrel{\text{(\rm L2) for $\YYY$}}{=} 0,
\end{align*}
\begin{align*}
\Hom_{\DD}(Y'_{1,j},Y'_{1,j'})
=\Hom_{\DD}(Y_{1,j},Y_{1,j'})\stackrel{\text{(\rm L2) for $\YYY$}}{=} 0. 
\end{align*}
Since $S'_{1}=S_{1}\subset S_{0}=S'_{0}+1$, 
\begin{align*}
\Hom_{\DD}(Y'_{1,j},Y'_{0,j'})
&=\Hom_{\DD}(Y_{1,j},
\S^{-1}_{\langle Y_{0}\rangle}(Y_{0,j'-1}))
\stackrel{\text{(\rm S3) for $\YYY$}}{=} \Hom_{\DD}(\S^{-1}(Y_{0,j-1}),
\S^{-1}_{\langle Y_{0}\rangle}(Y_{0,j'-1}))\\
&\simeq \Hom_{\DD}(\S^{-1}_{\langle Y_{0}\rangle}(Y_{0,j'-1}), Y_{0,j-1})^{\ast}
\simeq \Hom_{\DD}(Y_{0,j-1}, Y_{0,j'-1})
\stackrel{\text{(\rm L2) for $\YYY$}}{=} 0.
\end{align*}
Thus $\YYY$ satisfies (\rm L2.2), and so 
$\YYY$ is a weak $S'$-family. 

Since $\YYY$ satisfies (\rm S1), we have 
\begin{equation}\label{M0'}
Y_{i,j}=\begin{cases}
\S_{\langle Y_{0}\rangle}^{-1}(Y_{0,p_{2}-1}) 
& (i,j)=(1,p_{2}),\\
Y_{i,j} & \text{otherwise}.
\end{cases} 
\end{equation}
From (\ref{M0'}), 
$\YYY$ satisfies 
(\rm S1) and (\rm S3). 
If $(1,j),(2,j)\in S'$, then 
\[\F_{\langle Y'_{0}\rangle}(Y'_{1,j})
=\F_{\langle Y_{0}\rangle}(Y_{1,j})
\stackrel{\text{(\rm S3) for $\YYY$}}{=} 
\F_{\langle Y_{0}\rangle}\S^{-1}(Y_{0,j-1})
\stackrel{\ref{Serre functor}(\rm b)}{=}\S^{-1}_{\langle Y_{0}\rangle}
(Y_{0,j-1})
=Y'_{0,j},\]
and so $\YYY$ satisfies (\rm S2)$'$. 
Thus the assertion follows. 
\end{proof}
Now we are ready to prove Theorem \ref{Mutation Lemma 1}.
\begin{proof}[Proof of Theorem 
\ref{Mutation Lemma 1}]
(a) Let $\XXX'=(X'_{i,j})_{(i,j)\in S'}$. 
By Lemma \ref{Mutation Lemma 00}, 
$\XXX'_{\leqslant k}$ is 
an $S_{\leqslant k}$-family in $\DD$. 
By Lemma \ref{Mutation Lemma 0}, 
$\XXX'_{\geqslant k}$ is 
an $S_{\geqslant k}$-family in $\DD$. 
Thus the assertion follows from Proposition \ref{Glu1}. \\
(b) Let $\XXX'=(X'_{i,j})_{(i,j)\in S'}$. 
By Lemma \ref{Mutation Lemma 00}, 
$\XXX'_{\leqslant k}$ is 
an $S_{\leqslant k}$-family in $\DD$. 
By Lemma \ref{Mutation Lemma 0}, 
$\XXX'_{\geqslant k}$ is 
an $S_{\geqslant k}$-family in $\DD$. 
Thus the assertion follows from Proposition \ref{Glu1}. 
\end{proof}

A subset $S$ of $\Z^{2}$ 
is called a \emph{${}^{t}M^{+}_k$-subset} 
if ${}^{t}S$ is an $M^{+}_k$-subset. 
By Theorem \ref{Mutation Lemma 1}, 
we obtain the following result:
\begin{thm}[Mutation \rm $^{t}$I]
\label{Mutation Lemma 1op}
Let $\DD$ be a triangulated category satisfying 
$(\ref{Good})$, 
$S$ a ${}^{t}M^{+}_k$-subset satisfying 
$S^{\leqslant -1}=\emptyset$, and let $S'=\rho_{\leqslant k}(S)$. 

\begin{enumerate}[\rm (a)]
\item If $\XXX=(X_{i,j})_{(i,j)\in S}$ is an $S$-family, then 
the family $\XXX'=(X'_{i,j})_{(i,j)\in S'}$ 
 is an $S'$-family where 
 \begin{equation*}
X'_{i,j}=\begin{cases}
\S_{\langle X^{j}\rangle}(X_{i+1,j}) 
& j\leqslant k,\\
X_{i,j} & k<j. 
\end{cases}
\end{equation*}
In particular, there exists a triangle equivalence 
$\per \LL(S)\to \per \LL(S')$.

\item If $\YYY=(Y_{i,j})_{(i,j)\in S'}$ is an $S'$-family, 
then the family 
$\YYY'=(Y'_{i,j})_{(i,j)\in S}$ is an 
 $S$-family where 
 \begin{equation*}
Y'_{i,j}=\begin{cases}
\S_{\langle Y^{j}\rangle}^{-1}(Y_{i-1,j}) 
& j\leqslant k,\\
Y_{i,j} & k<j.
\end{cases}
\end{equation*}
In particular, there exists a triangle equivalence 
$\per \LL(S')\to \per \LL(S)$.
\end{enumerate}
\end{thm}

The following result is the one of main results in this subsection. 

\begin{thm}[Mutation 
\rm I\hspace{-.01em}I]\label{Mutation Lemma 2}
Let $\DD$ be a triangulated category satisfying 
$(\ref{Good})$, 
and let $k$, $h$ be two positive integers such that $k\in (h+1)\Z$, and let 
$s=\frac{(h-1)k}
{h+1}$. 
Let $S$ be an $M^{+}_{k-1}$-subset of $\Z^{2}$ 
satisfying $|S_{0}|=h$ and $S_{i}=S_{i-1}$ for 
any $i\in[1,k-1]$, and let 
$S'=\sigma_{\leqslant 0}\sigma_{\leqslant 1}
\dots\sigma_{\leqslant k-1}(S)$.  
\begin{enumerate}[\rm (a)]
\item If $\XXX=(X_{i,j})_{(i,j)\in S}$ is an $S$-family, then 
the family $\XXX'=(X'_{i,j})_{(i,j)\in S'}$ 
is an $S'$-family such that 
$\langle X_{S}\rangle=\langle X'_{S'}\rangle$ 
where 
\begin{equation*}
X'_{i,j}=
\begin{cases}
X_{i,j}[s] & i<0,\\
\S_{\langle X_{i}\rangle}^{k-i}(X_{i,j}) 
& 0\leqslant i\leqslant k-1,\\
X_{i,j} & i> k-1. 
\end{cases}
\end{equation*}
In particular, there exists a triangle equivalence 
$\per \LL(S)\to \per \LL(S')$. 
\item If $\YYY=(Y_{i,j})_{(i,j)\in S'}$ 
is an $S'$-family, 
then the family 
$\YYY'=(Y'_{i,j})_{(i,j)\in S}$
is an $S$-family such that 
$\langle Y_{S}\rangle=\langle Y'_{S'}\rangle$ 
where 
\begin{equation*}
Y'_{i,j}=
\begin{cases}
Y_{i,j}[-s] & i< 0,\\
\S_{\langle Y_{i}\rangle}^{-k+i}(Y_{i,j}) 
& 0\leqslant i\leqslant k-1,\\
Y_{i,j} & i> k-1. 
\end{cases}
\end{equation*}
In particular, there exists a triangle equivalence 
$\per \LL(S')\to \per \LL(S)$. 
\end{enumerate}
\end{thm}

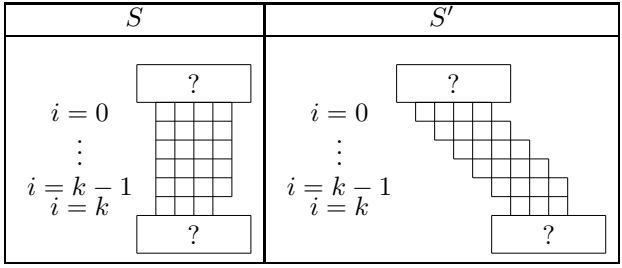
\begin{tabular}{|c|c|}\hline
$S$ & $S'$ \\ \hline 
\begin{tikzpicture}
\draw (0,3/4) node {};
\draw (-1,-1/8) node {$i=0$};
\draw (-1,-1/8-7/16) node {$\vdots$};
\draw (-1,-1/8-4/4) node {$i=k-1$};
\draw (-1,-1/8-5/4) node {$i=k$};

\foreach \z in {0,1,2,3} 
\draw[ultra thin] (\z/4,0)--(\z/4,-6/4); 
\foreach \z in {4} 
\draw[ultra thin] (\z/4,0)--(\z/4,-5/4);

\foreach \w in 
{0,1,2,3,4,5}
\draw[ultra thin] (0,-\w/4)--(4/4,-\w/4); 
\draw[ultra thin] (0,-6/4)--(3/4,-6/4); 

\foreach \w in {-1,5}
\draw[ultra thin] (\w/4,0)--(\w/4,2/4);
\foreach \w in {0,2}
\draw[ultra thin] (-1/4,\w/4)--(5/4,\w/4);

\foreach \z in {-1,5} 
\draw[ultra thin] (\z/4,-6/4)--(\z/4,-8/4); 
\foreach \w in {6,8}
\draw[ultra thin] (-1/4,-\w/4)--(5/4,-\w/4); 

\draw (4/8,1/4) node {?};
\draw (4/8,-7/4) node {?};
\end{tikzpicture}
&
\begin{tikzpicture}
\draw (0,1/4) node {};
\draw (-1,-1/8) node {$i=0$};
\draw (-1,-1/8-7/16) node {$\vdots$};
\draw (-1,-1/8-4/4) node {$i=k-1$};
\draw (-1,-1/8-5/4) node {$i=k$};

\foreach \z in {0,1,2,3} 
\draw[ultra thin] (\z/4,0)--(\z/4,-\z/4-1/4);

\foreach \z in {4} 
\draw[ultra thin] (\z/4,-\z/4+4/4)
--(\z/4,-\z/4-1/4); 

\foreach \z in {5} 
\draw[ultra thin] (\z/4,-\z/4+4/4)
--(\z/4,-6/4); 

\foreach \z in {6,7,8} 
\draw[ultra thin] (\z/4,-\z/4+4/4)
--(\z/4,-6/4);  

\draw[ultra thin] (0,0)--(4/4,0); 
\foreach \w in 
{1,2,3,4}
\draw[ultra thin] (\w/4-1/4,-\w/4)
--(\w/4+4/4,-\w/4); 

\foreach \w in 
{5}
\draw[ultra thin] (4/4,-\w/4)
--(8/4,-\w/4); 
\draw[ultra thin] (5/4,-6/4)
--(8/4,-6/4); 

\foreach \w in {-1,5}
\draw[ultra thin] (\w/4,0)--(\w/4,2/4);
\foreach \w in {0,2}
\draw[ultra thin] (-1/4,\w/4)--(5/4,\w/4);

\foreach \z in {4,10} 
\draw[ultra thin] (\z/4,-6/4)--(\z/4,-8/4); 
\foreach \w in {6,8}
\draw[ultra thin] (4/4,-\w/4)--(10/4,-\w/4);

\draw (4/8,1/4) node {?};
\draw (7/4,-7/4) node {?};
\end{tikzpicture}\\ \hline 
\end{tabular}

\begin{proof}
(a) Since $S_{\geqslant 0}$ is an 
$M^{+}_{k-1}$-subset satisfying $(S_{\geqslant 0})_{\leqslant -1}=\emptyset$, the subfamily 
$\XXX'_{\geqslant 0}$ is 
an $S'_{\geqslant 0}$-family
by Theorem \ref{Mutation Lemma 1}(a). 
Since there exists a triangle equivalence 
$F:\langle X_{0}\rangle
\to\per\NN(S_{0})$ such that $F(X_{0,j})\simeq P(j)$ 
for any $j\in S_{0}$ by Lemma \ref{Row Lemma}, 
we have 
\[X'_{0,j}=\S_{\langle X_{0}\rangle}^{k}(X_{0,j}) \stackrel{\ref{CY lemma}}{=}
X_{0,j}[s],\]
and so we have that 
$\XXX'_{\leqslant 0}=(X_{i,j}[s])_{(i,j)\in S'_{\leqslant 0}}$ is 
an $S'_{\leqslant 0}$-family. 
Thus the assertion follows from 
Proposition \ref{Glu1}. \\
(b) This is the dual of (a). 
\end{proof}
By transposing, we obtain the following result.
\begin{thm}[Mutation 
$^{t}$\rm I\hspace{-.01em}I]\label{Mutation Lemma 2op}
Let $\DD$ be a triangulated category satisfying 
$(\ref{Good})$, 
and let $k$, $h$ be two positive integers such that $k\in (h+1)\Z$, and let $s=\frac{(h-1)k}
{h+1}$. 
Let $S$ be an $^{t}M^{+}_{k-1}$-subset of $\Z^{2}$ 
satisfying $|S^{0}|=h$ and $S^{i}=S^{i-1}$ for 
any $i\in[1,k-1]$, and let 
$S'=\rho_{\leqslant 0}\rho_{\leqslant 1}
\dots\rho_{\leqslant k-1}(S)$.  
\begin{enumerate}[\rm (a)]
\item If $\XXX=(X_{i,j})_{(i,j)\in S}$ is an $S$-family, then 
the family $\XXX'=(X'_{i,j})_{(i,j)\in S'}$ 
is an $S'$-family such that 
$\langle X_{S}\rangle=\langle X'_{S'}\rangle$ 
where 
\begin{equation*}
X'_{i,j}=
\begin{cases}
X_{i,j}[s] & j< 0,\\
\S_{\langle X^{j}\rangle}^{k-j}(X_{i,j}) 
& 0\leqslant j\leqslant k-1,\\
X_{i,j} & j>k-1. 
\end{cases}
\end{equation*}
In particular, there exists a triangle equivalence 
$\per \LL(S)\to \per \LL(S')$. 
\item If $\YYY=(Y_{i,j})_{(i,j)\in S'}$ 
is an $S'$-family, 
then the family 
$\YYY'=(Y'_{i,j})_{(i,j)\in S}$
is an $S$-family such that 
$\langle Y_{S}\rangle=\langle Y'_{S'}\rangle$ 
where 
\begin{equation*}
Y'_{i,j}=
\begin{cases}
Y_{i,j}[-s] & j< 0,\\
\S_{\langle Y^{j}\rangle}^{-k+j}(Y_{i,j}) 
& 0\leqslant j\leqslant k-1,\\
Y_{i,j} & j>k-1. 
\end{cases}
\end{equation*}
In particular, there exists a triangle equivalence 
$\per \LL(S')\to \per \LL(S)$. 
\end{enumerate}
\end{thm}

\subsection{Applications of mutations of 
$S$-families}
We are ready to prove one of our main results 
by applying Theorem \ref{Mutation Lemma 1}
and Theorem \ref{Mutation Lemma 2op}
as follows where 
\[t'=t-1,\ \ u'=u-s,\ \ 
Y=Y(s,t,u),\ \  
Y'=Y(s,t',u').\] 
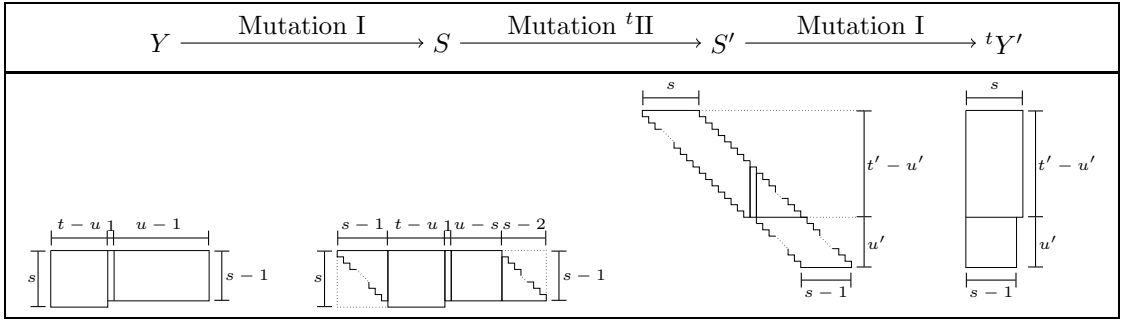
\begin{tabular}{|cccc|}\hline
\multicolumn{4}{|c|}
{\begin{tikzpicture}[auto]
\node (z) at (0,0){};
\node (a) at (0.8,0){$Y$};
\node (b) at (4.5,0){$S$};
\node (c) at (8.2,0){$S'$};
\node (d) at (11.9,0)
{${}^{t}Y'$};
\draw[->] (a) to node
{\text{Mutation \rm I}} (b);
\draw[->] (b) to node
{\text{Mutation \rm  ${}^{t}$I\hspace{-.01em}I}} (c);
\draw[->] (c) to node
{\text{Mutation \rm I}} (d);
\end{tikzpicture}
} \\ \hline
%zu1
\begin{tikzpicture}[scale=1/12]
\draw[ultra thin] 
(0,0)--(0,-9)--(9,-9)--(9,0)--(0,0);
\draw[ultra thin] 
(9,-8)--(25,-8)--(25,0)--(9,0);
\draw[ultra thin] 
(10,0)--(10,-8);

\draw[thin, |-|] (-2,0)--(-2,-9);
\draw (-3,-9/2) node[auto=left] {\tiny$s$};

\draw[thin, |-|] (27,0)--(27,-8);
\draw (31,-8/2) node[auto=right] {\tiny$s-1$};

\draw[thin, |-|] (0,2)--(9,2);
\draw (9/2,4) node[auto=left] {\tiny$t-u$};

\draw[thin, -|] (9,2)--(10,2);
\draw (19/2,4) node[auto=left] {\tiny$1$};
\draw[thin, -|] (10,2)--(25,2);
\draw (17,4) node[auto=left] {\tiny$u-1$};
\end{tikzpicture} 
 & 
 %zu2
\begin{tikzpicture}[scale=1/12]
\draw[ultra thin] 
(0,0)--(0,-1)--(1,-1)--(1,-2)--(2,-2)--(2,-3)
--(3,-3);
\draw[very thin, densely dotted] 
(3,-3)--(5,-5);
\draw[ultra thin] 
(5,-5)--(5,-6)--(6,-6)--(6,-7)--(7,-7)
--(7,-8)--(8,-8)--(8,0)--(0,0);

\draw[ultra thin] 
(8,0)--(8,-9)--(17,-9)--(17,0)--(8,0);
%\draw (25/2,-9/2) node {\tiny $S^{I}$};

\draw[ultra thin] 
(17,-8)--(18,-8)--(18,0)--(17,0)--(17,-8);

\draw[ultra thin] 
(18,-8)--(26,-8)--(26,0)--(18,0)--(18,-8);
%\draw (44/2,-8/2) node {\tiny $S^{J}$};

\draw[ultra thin] 
(26,-8)--(33,-8)--(33,-7)--(32,-7)--(32,-6)
--(31,-6)--(31,-5)--(30,-5);
\draw[very thin, densely dotted] 
(30,-5)--(29,-4);
\draw[ultra thin] 
(29,-4)--(29,-3)--(28,-3)--(28,-2)
--(27,-2)--(27,-1)--(26,-1)--(26,-8);

\draw[thin, |-|] (-2,0)--(-2,-9);
\draw (-3,-9/2) node[auto=left] {\tiny$s$};
\draw[very thin, densely dotted] 
(0,-9)--(8,-9);
\draw[very thin, densely dotted] 
(0,-1)--(0,-9);

\draw[thin, |-|] (35,0)--(35,-8);
\draw (39,-8/2) node[auto=right] {\tiny$s-1$};
\draw[very thin, densely dotted] 
(26,0)--(33,0);

\draw[thin, |-|] (0,2)--(8,2);
\draw (8/2,4) node {\tiny$s-1$};

\draw[thin, -|] (8,2)--(17,2);
\draw (25/2,4) node {\tiny$t-u$};

\draw[thin, -|] (17,2)--(18,2);
\draw (35/2,4) node[auto=left] {\tiny$1$};

\draw[thin, -|] (18,2)--(26,2);
\draw (22,4) node[auto=left] {\tiny$u-s$};

\draw[thin, -|] (26,2)--(33,2);
\draw (59/2,4) node[auto=left] {\tiny$s-2$};
\draw[very thin, densely dotted] 
(33,0)--(33,-7);
\end{tikzpicture}
& 
%zu3
\begin{tikzpicture}[scale=1/12]
\draw[ultra thin] 
(8,0)--(0,0)--(0,-1)--(1,-1)--(1,-2)--(2,-2)--(2,-3)
--(3,-3);
\draw[very thin, densely dotted] 
(3,-3)--(5,-5);
\draw[ultra thin] 
(5,-5)--(5,-6)--(6,-6)--(6,-7)--(7,-7)
--(7,-8)--(8,-8);

\draw[ultra thin] 
(8,-8)--(8,-9)--(9,-9)--(9,-10)--(10,-10)
--(10,-11)--(11,-11)--(11,-12)--(12,-12)
--(12,-13)--(13,-13)--(13,-14)--(14,-14)
--(14,-15)--(15,-15)--(15,-16)--(16,-16)
--(16,-17)--(17,-17)--(17,-8);
\draw[ultra thin] 
(8,0)--(9,0)--(9,-1)--(10,-1)
--(10,-2)--(11,-2)--(11,-3)--(12,-3)
--(12,-4)--(13,-4)--(13,-5)--(14,-5)
--(14,-6)--(15,-6)--(15,-7)--(16,-7)
--(16,-8)--(17,-8);

\draw[ultra thin] 
(17,-17)--(18,-17)--(18,-9)--(17,-9)--(17,-17);

\draw[ultra thin] 
(18,-18)--(19,-18)--(19,-19)--(20,-19)--(20,-20)
--(21,-20);
\draw[very thin, densely dotted] 
(21,-20)--(23,-22);
\draw[ultra thin]
(23,-22)--(23,-23)--(24,-23)--(24,-24)
--(25,-24)--(25,-25)--(26,-25);

\draw[ultra thin] 
(18,-18)--(18,-10)--(19,-10)--(19,-11)
--(20,-11)--(20,-12)
--(21,-12)--(21,-13);

\draw[ultra thin]
(22,-14)--(23,-14)--(23,-15)--(24,-15)
--(24,-16)--(25,-16)--(25,-17)--(26,-17)
--(26,-18);
\draw[very thin, densely dotted] 
(21,-13)--(22,-14);

\draw[ultra thin] 
(26,-25)--(33,-25)--(33,-24)--(32,-24)--(32,-23)
--(31,-23)--(31,-22)--(30,-22);
\draw[very thin, densely dotted] 
(30,-22)--(28,-20);
\draw[ultra thin]
(28,-20)--(28,-19)
--(27,-19)--(27,-18)--(26,-18);

\draw[thin, |-|] (35,0)--(35,-17);
\draw (40,-17/2) node[auto=right] 
{\tiny$t'-u'$};

\draw[very thin, densely dotted] 
(9,0)--(35,0);
\draw[very thin, densely dotted] 
(26,-17)--(34,-17);

\draw[ultra thin]
(18,-17)--(26,-17);

\draw[thin, -|] (35,-17)--(35,-25);
\draw (37,-21) node[auto=right] {\tiny$u'$};

\draw[thin, |-|] (0,2)--(9,2);
\draw (9/2,4) node {\tiny$s$};
\draw[thin, |-|] (25,-27)--(33,-27);
\draw (29,-29) node {\tiny$s-1$};
\end{tikzpicture}

& %zu4
\begin{tikzpicture}[scale=1/12]
\draw[ultra thin] 
(0,0)--(0,-25)--(8,-25)--(8,-17)--(9,-17)
--(9,0)--(0,0);
\draw[thin, |-|] (11,0)--(11,-17);
\draw (16,-17/2) node[auto=right] 
{\tiny $t'-u'$};

\draw[thin, -|] (11,-17)--(11,-25);
\draw (13,-42/2) node[auto=right] {\tiny$u'$};
\draw[ultra thin]
(0,-17)--(8,-17);

\draw[thin, |-|] (0,2)--(9,2);
\draw (9/2,4) node {\tiny$s$};
\draw[thin, |-|] (0,-27)--(8,-27);
\draw (4,-29) node {\tiny$s-1$};
\end{tikzpicture}
  \\
\hline\end{tabular}

\begin{thm}\label{Main 1}
Let $s$, $t$, $u$ be 
three positive integers such that $1\leqslant u\leqslant t$. 
Suppose that one of the following 
conditions is satisfied. 
\begin{enumerate}[\rm (a)]
\item $u\in \mathbb{Z}s$ 
and $t-u\in \mathbb{Z}(s+1)$. 
\item $s=2$ and 
$t-u\in 3\mathbb{Z}$. 
\end{enumerate}
Then there exist triangle equivalences 
\[\per\LL(s,t,u)\to\per \LL(S)\to\per \LL(S')
\to\per\LL(s,t-1,u-s)\]
where 
\[S=\sigma_{\leqslant 1}
\sigma_{\leqslant 2}\dots\sigma_{\leqslant s-1}
(\Y(s,t,u)),\]
\[S'=(\rho_{\leqslant 1}
\rho_{\leqslant 2}\dots
\rho_{\leqslant t-u})
(\rho^{-1}_{\geqslant t-s+1}
\rho^{-1}_{\geqslant t-s}
\dots
\rho^{-1}_{\geqslant t-u+2})(S).\]
\end{thm}

\begin{proof}
(a) 
There exists a triangle equivalence 
$\per \LL(s,t,u)\to \per \LL(S)$ by 
Theorem \ref{Mutation Lemma 1}.

Let $I=[1,t-u]$, $J=[t-u+2,t-s+1]$. 
Since $S^{I}\equiv 
\Y(s;t-u)$ and $S^{t-u+1}\subset S^{t-u}$, 
we have that $S$ is a 
${}^{t}M^{+}_{t-u-1}(1,1)$-subset 
satisfying $t-u\in \Z (s+1)$. 
Since $S^{J}\equiv 
\Y(s-1;u-s)$ and 
$S^{t-u+2}=S^{t-u+1}$, we have that 
$-S$ is a 
${}^{t}M^{+}_{u-s-1}(-1,-t+s-1)$-subset satisfying $u-s\in \Z s$. 
Thus there exists a triangle equivalence 
$\per \LL(S)\to \per \LL(S')$ by 
Theorem \ref{Mutation Lemma 2op}. 
 
Since 
\[{}^{t}\Y(s,t-1,u-s)
\equiv(\sigma^{-1}_{\leqslant s-t+u}
\sigma^{-1}_{\leqslant s-t+u+1}\dots
\sigma^{-2}_{\leqslant s-1})
(\sigma_{\geqslant u}
\sigma_{\geqslant u-1}
\dots\sigma_{\geqslant s+1})(S'),\]
there exists a triangle equivalence 
$\per \LL(S')\to \per \LL(s,t-1,u-s)$ by 
Theorem \ref{Mutation Lemma 1}. \\
(b) There exists a triangle equivalence 
$\per \LL(2,t,u)\to \per \LL(S)$ by 
Theorem \ref{Mutation Lemma 1}. 

Let $I=[1,t-u]$, $J=[t-u+2,t-1]$. 
Since $S^{I}\equiv 
\Y(2;t-u)$ and $S^{t-u+1}\subset S^{t-u}$, 
we have that $S$ is a 
${}^{t}M^{+}_{t-u-1}(1,1)$-subset 
satisfying $t-u\in 3\Z$. 
Since $S^{J}\equiv 
\Y(1;u-s)$, we have that 
$-S$ is a ${}^{t}M^{+}_{u-3}(-1,-t+1)$-subset. 
Thus there exists a triangle equivalence 
$\per \LL(S)\to \per \LL(S')$ by 
Theorem \ref{Mutation Lemma 1op} 
and Theorem \ref{Mutation Lemma 2op}. 

Since 
\[{}^{t}\Y(2,t-1,u-2)
\equiv
(\sigma^{-1}_{\leqslant -t+u+2}
\sigma^{-1}_{\leqslant -t+u+3}\dots
\sigma^{-2}_{\leqslant 1})
(\sigma_{\geqslant u}
\sigma_{\geqslant u-1}
\dots\sigma_{\geqslant 3})(S'),\]
there exists a triangle equivalence 
$\per \LL(S')\to \per \LL(2,t-1,u-2)$ by 
Theorem \ref{Mutation Lemma 1}. 
\end{proof} 
By Theorem \ref{Nak}, we have the following result. 
\begin{cor}\label{Main 2}
Let $p,q$ be two integers such that 
$p\geqslant 2$, $q\geqslant 1$. 
Suppose that one of the following 
conditions is satisfied. 
\begin{enumerate}[\rm (a)]
\item $r\in\Z_{\geqslant 0}$.
\item $p=2$ and 
$r\in \frac{1}{2}\Z_{\geqslant 0}$. 
\end{enumerate}
Then there exists a triangle equivalence 
\[\per\AA(n,\ell+1)
\to \per\AA(n,\ell) 
\text{ where 
$n=p(p+1)q+p(p-1)r$, $\ell=(p+1)q+pr$}.\]
\end{cor}

\begin{proof}
Let $(s,t,u)=(p,(p+1)q+pr,pr)$. 
By Theorem \ref{Nak}, there exist 
triangle equivalences 
\[\per N(n,\ell+1)\to \per \LL(s,t,u),\ \ 
\per N(n,\ell)\to \per \LL(s,t-1,u-s).\]
Thus the assertion follows from 
Theorem \ref{Main 1}. 
\end{proof}

\begin{figure}
\begin{tabular}{|c|c|c|c|c|}\hline
$(s,t,u)$ & $\Y(s,t,u)$ & $S$ & $S'$ 
& ${}^{t}\Y(s,t-1,u-s)$\\ 
\hline
\begin{tikzpicture}
\draw (0,7/8) node {(2,8,5)};
\draw (0,0) node {};
\end{tikzpicture}
 & 
\begin{tikzpicture}
\draw (0,1/4) node {};

\foreach \z in {0,1,2,3} 
\draw[ultra thin] (\z/4,0)--(\z/4,-2/4);
\foreach \z in {4,5,6,7,8} 
\draw[ultra thin] (\z/4,0)--(\z/4,-1/4);

\foreach \w in 
{0,1}
\draw[ultra thin] (0,-\w/4)--(8/4,-\w/4); 
\draw[ultra thin] (0,-2/4)--(3/4,-2/4); \end{tikzpicture} 
& 
\begin{tikzpicture}
\draw (0,1/4) node {};

\draw[ultra thin] (0,0)--(0,-1/4);
\foreach \z in {1,2,3,4} 
\draw[ultra thin] (\z/4,0)--(\z/4,-2/4);
\foreach \z in {5,6,7,8} 
\draw[ultra thin] (\z/4,0)--(\z/4,-1/4);

\foreach \w in {0,1}
\draw[ultra thin] (0,-\w/4)--(8/4,-\w/4); 
\draw[ultra thin] (1/4,-2/4)--(4/4,-2/4); \end{tikzpicture} 
& 
\begin{tikzpicture}
\draw (0,1/4) node {};

\draw[ultra thin] (0,0)--(2/4,0); 
\foreach \z in {0,1} 
\draw[ultra thin] 
(\z/4,-\z/4-1/4)--(\z/4+3/4,-\z/4-1/4);
\draw[ultra thin] 
(2/4,-3/4)--(5/4,-3/4);
\draw[ultra thin] 
(3/4,-4/4)--(6/4,-4/4); 

\draw[ultra thin] 
(0,0)--(0,-1/4);
\draw[ultra thin] 
(1/4,0)--(1/4,-2/4);
\draw[ultra thin] 
(2/4,0)--(2/4,-3/4);
\draw[ultra thin] 
(3/4,-1/4)--(3/4,-4/4);
\draw[ultra thin] 
(4/4,-2/4)--(4/4,-4/4);
\draw[ultra thin] 
(5/4,-3/4)--(5/4,-4/4);

\draw[ultra thin] 
(5/4,-4/4)--(5/4,-5/4);
\draw[ultra thin] 
(6/4,-4/4)--(6/4,-6/4);
\draw[ultra thin] 
(7/4,-5/4)--(7/4,-7/4);
\draw[ultra thin] 
(8/4,-6/4)--(8/4,-7/4);
 
\draw[ultra thin] 
(5/4,-5/4)--(7/4,-5/4);
\draw[ultra thin] 
(6/4,-6/4)--(8/4,-6/4);
\draw[ultra thin] 
(7/4,-7/4)--(8/4,-7/4);

\end{tikzpicture}
 & 
\begin{tikzpicture}
\draw (0,1/4) node {};

\foreach \z in {0,1} 
\draw[ultra thin] (\z/4,0)--(\z/4,-7/4);
\draw[ultra thin] (2/4,0)--(2/4,-4/4);

\foreach \z in {0,1,2,3,4} 
\draw[ultra thin] (0,-\z/4)--(2/4,-\z/4);
\foreach \z in {5,6,7} 
\draw[ultra thin] (0,-\z/4)--(1/4,-\z/4);
\end{tikzpicture}\\ 

\hline
\begin{tikzpicture}
\draw (0,4/4) node {(3,10,6)};
\draw (0,0) node {};
\end{tikzpicture}
& 
\begin{tikzpicture}
\foreach \z in {0,1,2,3,4} 
\draw[ultra thin] (\z/4,0)--(\z/4,-3/4);
\foreach \z in {5,6,7,8,9,10} 
\draw[ultra thin] (\z/4,0)--(\z/4,-2/4);

\foreach \w in 
{0,1,2}
\draw[ultra thin] (0,-\w/4)--(10/4,-\w/4); 

\draw[ultra thin] (0,-3/4)--(4/4,-3/4); \end{tikzpicture}
& 
\begin{tikzpicture}
\foreach \z in {0,1,2} 
\draw[ultra thin] (\z/4,0)--(\z/4,-\z/4-1/4);

\foreach \z in {3,4,5,6} 
\draw[ultra thin] (\z/4,0)--(\z/4,-3/4); 
\foreach \z in {7,8,9,10} 
\draw[ultra thin] (\z/4,0)--(\z/4,-2/4);
\draw[ultra thin] (11/4,-1/4)--(11/4,-2/4);

\draw[ultra thin] (0,0)
--(10/4,0);
\foreach \w in 
{1,2}
\draw[ultra thin] (\w/4-1/4,-\w/4)
--(11/4,-\w/4); 

\draw[ultra thin] (2/4,-3/4)
--(6/4,-3/4);
\end{tikzpicture}
& 
\begin{tikzpicture}
\draw (0,1/4) node {};

\foreach \z in {0,1,2,3} 
\draw[ultra thin] (\z/4,0)--(\z/4,-\z/4-1/4);

\foreach \z in {4,5} 
\draw[ultra thin] (\z/4,-\z/4+3/4)
--(\z/4,-\z/4-1/4);

\foreach \z in {6,7,8,9} 
\draw[ultra thin] (\z/4,-\z/4+3/4)
--(\z/4,-\z/4); 

\foreach \z in {10,11} 
\draw[ultra thin] (\z/4,-\z/4+3/4)
--(\z/4,-9/4);

\draw[ultra thin] (0,0)--(3/4,0); 
\foreach \w in 
{1,2,3,4,5,6}
\draw[ultra thin] (\w/4-1/4,-\w/4)
--(\w/4+3/4,-\w/4); 

\draw[ultra thin] (7/4,-7/4)--(10/4,-7/4); 
\foreach \w in 
{8,9}
\draw[ultra thin] (\w/4,-\w/4)
--(11/4,-\w/4); 

\end{tikzpicture} 
& 
\begin{tikzpicture}
\draw (0,1/4) node {};

\foreach \z in {0,1,2} 
\draw[ultra thin] (\z/4,0)--(\z/4,-9/4);
\draw[ultra thin] (3/4,0)--(3/4,-6/4);

\foreach \w in 
{0,1,2,3,4,5,6}
\draw[ultra thin] (0,-\w/4)--(3/4,-\w/4); 

\foreach \w in 
{7,8,9}
\draw[ultra thin] (0,-\w/4)--(2/4,-\w/4); 
\end{tikzpicture}\\ 
\hline
\begin{tikzpicture}
\draw (0,4/4) node {(4,9,4)};
\draw (0,0) node {};
\end{tikzpicture}
& 
\begin{tikzpicture}
\foreach \z in {0,1,2,3,4,5} 
\draw[ultra thin] (\z/4,0)--(\z/4,-4/4);
\foreach \z in {6,7,8,9} 
\draw[ultra thin] (\z/4,0)--(\z/4,-3/4);

\foreach \w in 
{0,1,2,3}
\draw[ultra thin] (0,-\w/4)--(9/4,-\w/4); 

\draw[ultra thin] (0,-4/4)--(5/4,-4/4); \end{tikzpicture}
& 
\begin{tikzpicture}
\foreach \z in {0,1,2} 
\draw[ultra thin] (\z/4,0)--(\z/4,-\z/4-1/4);

\foreach \z in {3,4,5,6,7,8} 
\draw[ultra thin] (\z/4,0)--(\z/4,-4/4); 

\foreach \z in {9,10,11} 
\draw[ultra thin] (\z/4,-\z/4+9/4)--(\z/4,-3/4);

\draw[ultra thin] (0,0)--(9/4,0); 

\foreach \w in 
{1,2}
\draw[ultra thin] (\w/4-1/4,-\w/4)
--(\w/4+9/4,-\w/4); 

\draw[ultra thin] (2/4,-3/4)
--(11/4,-3/4);
\draw[ultra thin] (3/4,-4/4)
--(8/4,-4/4);
\end{tikzpicture}
&
\begin{tikzpicture}
\draw (0,1/4) node {};

\foreach \z in {0,1,2,3} 
\draw[ultra thin] (\z/4,0)--(\z/4,-\z/4-1/4);

\foreach \z in {4,5,6,7} 
\draw[ultra thin] (\z/4,-\z/4+4/4)
--(\z/4,-\z/4-1/4); 

\foreach \z in {8,9,10,11} 
\draw[ultra thin] (\z/4,-\z/4+4/4)
--(\z/4,-8/4);

\draw[ultra thin] (0,0)--(4/4,0); 
\foreach \w in 
{1,2,3,4,5,6,7}
\draw[ultra thin] (\w/4-1/4,-\w/4)
--(\w/4+4/4,-\w/4); 

\draw[ultra thin] (7/4,-8/4)--(11/4,-8/4); 
\end{tikzpicture}
&
\begin{tikzpicture}
\draw (0,1/4) node {};
\foreach \z in {0,1,2,3,4} 
\draw[ultra thin] (\z/4,0)--(\z/4,-8/4);

\foreach \w in 
{0,1,2,3,4,5,6,7,8}
\draw[ultra thin] (0,-\w/4)--(4/4,-\w/4); 
\end{tikzpicture}\\ 
\hline
\end{tabular}
\caption{Mutations from $\Y(s,t,u)$ to 
${}^{t}\Y(s,t-1,u-s)$}
\end{figure}

The following result is proved 
by applying Theorem \ref{Mutation Lemma 1} 
and Theorem \ref{Mutation Lemma 1op} as follows where 
\[Y=\Y(p+1,q,q-1),\ \ Y'=\Y(q+1,p,p-1).\] 
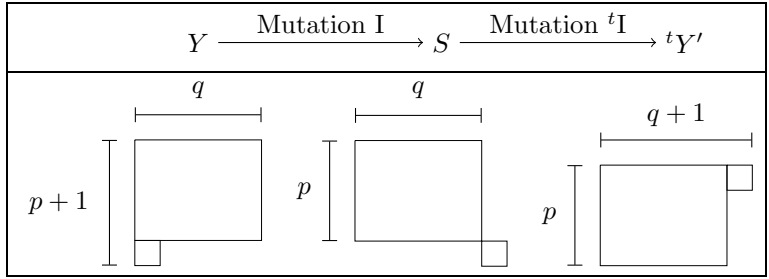
\begin{tabular}{|ccc|}\hline
\multicolumn{3}{|c|}
{\begin{tikzpicture}[auto]
\node (z) at (0,0){};
\node (a) at (1.7,0){$Y$};
\node (b) at (4.9,0){$S$};
\node (c) at (8.1,0){${}^{t}Y'$};
\draw[->] (a) to node
{\text{Mutation \rm I}} (b);
\draw[->] (b) to node
{\text{Mutation \rm  ${}^{t}$I}} (c);
\end{tikzpicture}
} \\ \hline
\begin{tikzpicture}[scale=1/3]
\draw[ultra thin] (0,0)--(5,0)--(5,-4)--(0,-4)
--(0,0); 
\draw[ultra thin] (0,-4)--(1,-4)--(1,-5)--(0,-5)--(0,-4);

\draw[thin, |-|] (-1,0)--(-1,-5);
\draw (-3,-5/2) node[auto=left] {$p+1$};

\draw[thin, |-|] (0,1)--(5,1);
\draw (5/2,2) node[auto=left] {$q$};
\end{tikzpicture}
 & 
\begin{tikzpicture}[scale=1/3]
\draw[ultra thin] (0,0)--(5,0)--(5,-4)--(0,-4)
--(0,0); 
\draw[ultra thin] (5,-4)--(6,-4)--(6,-5)--(5,-5)--(5,-4);

\draw[thin, |-|] (-1,0)--(-1,-4);
\draw (-2,-4/2) node[auto=left] {$p$};

\draw[thin, |-|] (0,1)--(5,1);
\draw (5/2,2) node[auto=left] {$q$};
\end{tikzpicture} 
& 
\begin{tikzpicture}[scale=1/3]
\draw[ultra thin] (0,0)--(5,0)--(5,-4)--(0,-4)
--(0,0); 
\draw[ultra thin] (5,0)--(6,0)
--(6,-1)--(5,-1)--(5,0);
\draw[thin, |-|] (-1,0)--(-1,-4);
\draw (-2,-4/2) node[auto=left] {$p$};

\draw[thin, |-|] (0,1)--(6,1);
\draw (6/2,2) node[auto=left] {$q+1$};
\end{tikzpicture} \\ \hline
\end{tabular}

\begin{thm}\label{Main3}
Let $p$, $q$ be two integers such that 
$p\geqslant 2$, $q\geqslant 2$. 
Then there exists a full 
$\Y(q+1,p,p-1)$-family in 
$\per \LL(p+1,q,q-1)$. 
In particular, 
there exists a triangle equivalence 
\[\per\LL(p+1,q,q-1)
\to \per\LL(q+1,p,p-1).\]
\end{thm}

\begin{proof}
Let $S$ be a subset of $\Z^{2}$ such that 
$S=\sigma_{\leqslant p}^{q}(Y(p+1,q,q-1))$. 
Since 
\[{}^{t}Y(q+1,p,p-1)\equiv \rho_{\leqslant 0}^{-p}(S),\]
the assertion follows from Theorem 
\ref{Mutation Lemma 1} and Theorem 
\ref{Mutation Lemma 1op}.
\end{proof}

\begin{cor}\label{Main 4}
There exists a triangle equivalence 
\[\text{$\per\AA(pq+1,q+1)
\to \per\AA(pq+1,p+1)$\ \ 
for any integers $p\geqslant 2$, 
$q\geqslant 2$.}\]
\end{cor}

\begin{proof} 
By Theorem \ref{Nak}, there exist 
triangle equivalences 
\[\per N(pq+1,q+1)\to \per \LL(p+1,q,q-1),\ \ 
\per N(pq+1,p+1)\to \per \LL(q+1,p,p-1).\]
Thus the assertion follows from 
Theorem \ref{Main3}. 
\end{proof}

\end{document}